\documentclass[reqno,a4paper,twoside]{amsart}

\usepackage[a4paper,margin=0.65in,footskip=0.25in]{geometry}

\usepackage{graphicx}
\usepackage{enumitem}

\setenumerate{label=\textnormal{(\arabic*)}}

\allowdisplaybreaks[4]

\usepackage{amsmath,amssymb,dsfont,verbatim,bm,mathrsfs,stmaryrd,mathabx}

\usepackage{mathtools}
\usepackage[raggedright]{titlesec}

\usepackage[utf8]{inputenc}
\usepackage[T1]{fontenc}

\usepackage[x11names,dvipsnames]{xcolor}
	\definecolor{egyptianblue}{rgb}{0.06, 0.2, 0.65}
	\definecolor{green(ncs)}{rgb}{0.0, 0.62, 0.42}

\usepackage{hyperref}
\hypersetup{
	colorlinks=true,
	citecolor=green(ncs),
	linkcolor= egyptianblue,
	filecolor=magenta,
	urlcolor=cyan,
	unicode=true,
	pdfpagemode=FullScreen,
}

\titleformat{\chapter}[display]
{\normalfont\huge\bfseries}{\chaptertitlename\\thechapter}{20pt}{\Huge}
\titleformat{\section}
{\normalfont\Large\bfseries\center}{\thesection}{1em}{}
\titleformat{\subsection}
{\normalfont\large\bfseries}{\thesubsection}{1em}{}
\titleformat{\subsubsection}
{\normalfont\normalsize\bfseries}{\thesubsubsection}{1em}{}
\titleformat{\paragraph}[runin]
{\normalfont\normalsize\bfseries}{\theparagraph}{1em}{}
\titleformat{\subparagraph}[runin]
{\normalfont\normalsize\bfseries}{\thesubparagraph}{1em}{}

\titlespacing*{\chapter} {0pt}{50pt}{40pt}
\titlespacing*{\section} {0pt}{3.5ex plus 1ex minus .2ex}{2.3ex plus .2ex}
\titlespacing*{\subsection} {0pt}{3.25ex plus 1ex minus .2ex}{1.5ex plus .2ex}
\titlespacing*{\subsubsection}{0pt}{3.25ex plus 1ex minus .2ex}{1.5ex plus .2ex}
\titlespacing*{\paragraph} {0pt}{3.25ex plus 1ex minus .2ex}{1em}
\titlespacing*{\subparagraph} {\parindent}{3.25ex plus 1ex minus .2ex}{1em}

\usepackage{float, subfig}
\usepackage{amsrefs}

\usepackage{titletoc}

\contentsmargin{2.55em}
\dottedcontents{section}[1.5em]{}{2em}{1pc}
\dottedcontents{subsection}[4.35em]{}{2.8em}{1pc}
\dottedcontents{subsubsection}[7.6em]{}{3.2em}{1pc}
\dottedcontents{paragraph}[10.3em]{}{3.2em}{1pc}

\usepackage[columns=3,rule=0pt]{idxlayout}

\makeindex

\newtheorem{theorem}{Theorem}[section]
\newtheorem{lemma}[theorem]{Lemma}
\newtheorem{proposition}[theorem]{Proposition}
\newtheorem{corollary}[theorem]{Corollary}

\theoremstyle{definition}
\newtheorem{definition}[theorem]{Definition}

\newtheorem{notation}[theorem]{Notation}
\newtheorem{example}[theorem]{Example}

\theoremstyle{remark}
\newtheorem{remark}[theorem]{Remark}

\usepackage{tikz}
\usetikzlibrary{matrix,shapes}
\usetikzlibrary{cd}
\tikzcdset{diagrams={nodes={inner sep=2pt}}}
\usetikzlibrary{arrows}
\tikzcdset{arrow style=tikz, diagrams={>={Latex[round,length=3pt,width=2pt,inset=1.75pt]}}}

\DeclareMathOperator{\End}{End}

\DeclareMathOperator{\ide}{id}
\DeclareMathOperator{\Ide}{Id}
\DeclareMathOperator{\Ext}{Ext}
\DeclareMathOperator{\EXT}{EXT}

\DeclareMathOperator{\Soc}{Soc}

\DeclareMathOperator{\ima}{Im}

\DeclareMathOperator{\Hom}{Hom}

\DeclareMathOperator{\Tot}{Tot}

\DeclareMathOperator{\sh}{sh}

\DeclareMathOperator{\sg}{sg}

\DeclareMathOperator{\ho}{H}

\newcommand{\ov}{\overline}
\newcommand{\ot}{\otimes}
\newcommand{\wh}{\widehat}
\newcommand{\wt}{\widetilde}

\newcommand{\xcirc}{\circ}

\numberwithin{equation}{section}

\DeclareMathAlphabet{\mathpzc}{OT1}{pzc}{m}{it}

\usepackage{mathtools}
\newtagform{brackets}{[}{]}
\usetagform{brackets}

\let\origmaketitle\maketitle
\def\maketitle{
	\begingroup
	\def\uppercasenonmath##1{} 
	\let\MakeUppercase\relax 
	\origmaketitle
	\endgroup
}

\usepackage{fancyhdr}
\usepackage[foot]{amsaddr}

\usepackage{adjustbox}

\begin{document}
	
\title{\large Extensions of a family of Linear Cycle Sets}

\author{Jorge A. Guccione}
\address{Pontificia Universidad Cat\'olica del Per\'u, Av. Universitaria 1801, San Miguel, Lima 32, Per\'u.}
\email{vander@dm.uba.ar}
	
\author{Juan J. Guccione}
\address{Pontificia Universidad Cat\'olica del Per\'u, Av. Universitaria 1801, San Miguel, Lima 32, Per\'u.}
\email{jjguccione@gmail.com}
	
\author[C. Valqui]{Christian Valqui}
\address{Pontificia Universidad Cat\'olica del Per\'u - Instituto de Matem\'atica y Ciencias Afi\-nes, Secci\'on Matem\'aticas, PUCP, Av. Universitaria 1801, San Miguel, Lima 32, Per\'u.}
\email{cvalqui@pucp.edu.pe}

\thanks{Jorge A. Guccione, Juan J. Guccione and Christian Valqui were supported by PUCP-CAP 2023-PI0991}

	
\begin{abstract}
This paper explores the cohomology of linear cycle sets, focusing on extensions of a specific linear cycle set $H$ by an abelian group $I$. We derive explicit formulas for the second cohomology group, which classifies these extensions, and establish conditions under which the extensions are fully determined. Key results include a characterization of extensions when $I$ lies in the socle of the extended structure and $H$ is trivial, and the construction of explicit examples for both trivial and non-trivial cases. The paper provides a systematic approach to understanding the structure of these extensions, with applications to various families of abelian groups.
\end{abstract}

\maketitle
	
\tableofcontents
	
\section*{Introduction}	

The relevance of the Yang-Baxter equation, (or equivalently the braid equation), led Drinfeld in~\cite{Dr}  to propose the study of the set-theoretic solutions, which are related with many important mathematical structures such as affine torsors, solvable groups, Biberbach groups and groups of I-type, Artin-Schelter regular rings, Garside structures, biracks, Hopf algebras, left symmetric algebras, etcetera; see for example~\cites{CR, CJO, CJR, Ch, DG, De1, DDM, GI2, GI4, GIVB, JO, S}.

An important class of set-theoretical solutions of the braid equation are the non-degenerate involutive solutions. The study of these solutions led to the introduction of the linear cycle set structures by Rump in~\cite{R}. This notion is equivalent to the notion of braces and to the notion of bijective $1$-cocycles (see~\cites{CJO1, ESS, LV, R, R3}). The study of extensions of linear cycle sets was done in the following papers~\cites{B,BG,LV,GGV}.

In the article~\cite{GGV} it was proved that, in order to classify the class of extension of a lineal cycle set $H$ by~a~trivial cycle set $I$, we must determine all the maps $\blackdiamond\colon H\times I\to I$ and $\Yleft \colon I\times H\to I$ satisfying suitable conditions, and then, for each one of this pair of maps, to compute the second cohomology group $\ho_{\blackdiamond,\Yleft}^2(H,I)$, of a filtrated cochain complex ${\mathcal{C}}_{\blackdiamond,\Yleft}(H,I)$. In the same paper it was proved that the classification of the extensions $(\iota,B,\pi)$, in which $\iota(H)$ is included in the socle of $B$, is obtained considering the cases in which $\Yleft = 0$. The aim of this paper is to classify all the extensions of $H$ by a trivial linear cycle set $I$ in the case in which $H$ is the linear cycle set with underlying additive group $(\mathds{Z}_{p^{\eta}},+)$ and operation $\cdot$, defined by $i\cdot j\coloneqq (1-p^{\nu} i)j$, where $p$ is a prime number, the juxtaposition denotes the multiplication in $\mathds{Z}_{p^{\eta}}$,  and $0<\nu\le\eta \le 2\nu$. The main results obtained in this paper are Proposition~\ref{condiciones sin potencia general}, where (for this case) we obtain all the maps $\blackdiamond$ and $\Yleft$ satisfying the suitable conditions mentioned above, and Corollary~\ref{calculo de la homo B'}, where we prove that each one of the cohomologies $\ho_{\blackdiamond,\Yleft}^2(H,I)$ is isomorphic to a quotient $\frac{\ker(F_1)\cap \ker(F_2)}{\ima(G)}$, where $F_1,F_2:I\times I\to I$ and $G:I\to I\times I$ are certain maps that are explicitly computable in many cases. Applying this result, in Subsection~\ref{seccion 4.1} we explicitly compute all extensions with $\Yleft=0$, of $H$, endowed with the trivial linear cycle set structure (case $\eta=\nu$), by a finite cyclic $p$-group $I=\mathds{Z}_{p^r}$. Finally, in Subsection~\ref{seccion 4.2} we use the main result in order to construct some extensions when $H$ is non-trivial (case $\nu<\eta$), or $\Yleft\ne 0$.

The article is organized as follows. In Section~\ref{seccion 1} we make a brief review of the notions of braces, linear cycle sets, extensions of linear cycle sets and the cohomology of linear cycle sets $\ho_{\blackdiamond,\Yleft}^*(H,I)$ (see~\cite{GGV}*{Section 6}). In Section~\ref{seccion 2} we introduce the linear cycle set $H$ mentioned above and we compute the actions $\blackdiamond: H\times I\to I$ and $\Yleft: I\times H\to I$. The main result of this section is Proposition~\ref{condiciones sin potencia general}. In Section~\ref{seccion 3} we first prove that a cocycle $(\beta,g)$ is equivalent modulo coboundaries to a cocycle of the form $(\alpha_1(\gamma),f)$, where $\alpha_1(\gamma)$ is a standard cocycle used classically to compute the group cohomology of a cyclic group. Then we prove that such a cocycle depends only on $f_0=f(1,1)$ and $\gamma$, both in $I$. We determine the necessary and sufficient conditions on $(f_0,\gamma)$, in order that there exists a cocycle $(\alpha_1(\gamma),f)$ with $f(1,1)=f_0$. Then we prove that the correspondence $(\alpha_1(\gamma),f)\leftrightarrow (f_0,\gamma)$ yields a linear isomorphism $\frac{\ker(F_1)\cap \ker(F_2)}{\ima(G)}\simeq \ho_{\blackdiamond,\Yleft}^2(H,I)$. The construction of $f$ out of $f_0$ and $\gamma$ is very complicated. Thus, for pedagogical reasons, we first study in Subsection 3.2 the construction of extensions with $\Yleft=0$, which means that the ideal of the extension $(\iota,B,\pi)$ is contained in the socle of~$B$. Such extensions are still of interest, and many known examples fall into this category. On the other hand, they are technically easier to handle than the general extensions (Subsection 3.3). We therefore found it instructive to present this “reduced” case before the general one. Finally, in Section~\ref{seccion 4} we apply our results to the construction of several families of extensions of $H$ by abelian groups.

\section{Preliminaries}\label{section:Preliminaries}\label{seccion 1}

In this section we recall the definitions of braces and linear cycle sets, and describe the relationship between the cohomology of linear cycle sets and the extensions. In this paper we work in the category of abelian groups, all the maps are $\mathds{Z}$-linear, $\ot$ means $\ot_{\mathds{Z}}$ and $\Hom$ means $\Hom_{\mathds{Z}}$.

\subsection[Braces and linear cycle sets]{Braces and linear cycle sets}\label{Braces and linear cycle sets}
A {\em left brace} is an abelian group $(A,+)$, endowed with an additional group operation $\times$, such that
\begin{equation}\label{compatibilidad braza}
a\times (b+c)= a\times b+a\times c-a\quad\text{for all $a,b,c\in A$.}
\end{equation}
The two group structures necessarily share the same neutral element, denoted by $0$. Of course, the additive inverse of $a$ is denoted by $-a$, while its multiplicative inverse is denoted by $\cramped{a^{\times -1}}$. In particular $\cramped{0^{\times -1}} = 0$.

Let $A$ and $B$ be left braces. A map $f\colon A\to B$ is an homomorphism of left braces if $f(a+a') = f(a)+f(a')$ and$~f(a\times a') = f(a)\times f(a')$, for all $a,a'\in A$.

\smallskip

A {\em linear cycle set} is an abelian group $(A,+)$, endowed with a binary operation $\cdot$, having bijective left translations $b\mapsto a\cdot b$, and satisfying the conditions
\begin{equation}\label{compatibilidades cdot suma y suma cdot}
a\cdot (b+c)=a\cdot b + a\cdot c\quad\text{and}\quad (a+b)\cdot c = (a\cdot b)\cdot (a\cdot c)\qquad\text{for all $a,b,c\in A$.}
\end{equation}
It is well known that each linear cycle set is a cycle set. That is:
\begin{equation}\label{condicion de cycle set}
(a\cdot b)\cdot (a\cdot c)=(b\cdot a)\cdot (b\cdot c)\quad\text{for all $a,b,c\in A$}.
\end{equation}
Let $A$ and $B$ be linear cycle sets. A map $f\colon A\to B$ is an homomorphism of linear cycle sets if $f(a+a') = f(a)+f(a')$ and $f(a\cdot a')=f(a)\cdot f(a')$, for all $a,a'\in A$.

The notions of left braces and linear cycle sets were introduced by Rump, who proved that they are equivalent via the relations
\begin{equation}
a\cdot b = a^{\times -1}\times (a+b)\quad \text{and}\quad a\times b = {}^ab + a,\label{equivalencia linear cycle set braza}
\end{equation}
where the map $b\mapsto {}^ab$ is the inverse of the map $b\mapsto a\cdot b$.

\smallskip

Each additive abelian group $A$ is a linear cycle set via $a\cdot b\coloneqq b$. These ones are the so called {\em trivial linear cycle sets}. Note that by~\eqref{equivalencia linear cycle set braza}, a linear cycle set $A$ is trivial if and only if $a\times b = a+b$, for all $a,b\in A$.

\smallskip

Let $A$ be a linear cycle set. An {\em ideal} of $A$ is  a subgroup $I$ of $(A,+)$ such that $a\cdot y \in I$ and $y\cdot a - a\in I$, for all~$a\in A$ and $y\in I$. An ideal $I$ of $A$ is called a {\em central ideal} if $a\cdot y=y$ and $y\cdot a=a$, for all $a\in A$ and $y\in I$. The {\em socle} of $A$ is the ideal $\Soc(A)$, of all $y\in A$ such that $y\cdot a = a$, for all $a\in A$. The {\em center} of $A$ is the set $Z(A)$, of all~$y\in \Soc(A)$ such that $a\cdot y = y$, for all $a\in A$. Thus, the center of $A$ is a central ideal of $A$, and each central ideal of~$A$ is included in $Z(A)$.

\begin{notation}\label{Yleft} Given an ideal $I$ of a linear cycle set $A$, for each $y\in I$ and $a\in A$, we set $y\Yleft a\coloneqq y\cdot a-a$.
\end{notation}

\subsection[Extensions of linear cycle sets]{Extensions of linear cycle sets}\label{Extensions of linear cycle sets}
Let $I$ and $H$ be additive abelian groups. Recall that an extension $(\iota, B, \pi)$, of $H$ by $I$, is a short exact sequence
\begin{equation}
\begin{tikzpicture}
\begin{scope}[yshift=0cm,xshift=0cm, baseline]
			\matrix(BPcomplex) [matrix of math nodes, row sep=0em, text height=1.5ex, text
			depth=0.25ex, column sep=2.5em, inner sep=0pt, minimum height=5mm, minimum width =6mm]
			{0 & I & B & H & 0,\\};
			\draw[-{latex}] (BPcomplex-1-1) -- node[above=1pt,font=\scriptsize] {} (BPcomplex-1-2);
			\draw[-{latex}] (BPcomplex-1-2) -- node[above=1pt,font=\scriptsize] {$\iota$} (BPcomplex-1-3);
			\draw[-{latex}] (BPcomplex-1-3) -- node[above=1pt,font=\scriptsize] {$\pi$} (BPcomplex-1-4);
			\draw[-{latex}] (BPcomplex-1-4) -- node[above=1pt,font=\scriptsize] {} (BPcomplex-1-5);
\end{scope}\label{sucesion exacta corta de grupos abelianos}
\end{tikzpicture}
\end{equation}
of additive abelian groups. Recall also that two such extensions, $(\iota, B, \pi)$ and $(\iota', B', \pi')$, are equivalent if there exists a morphism $\phi\colon B\to B'$ such that $\pi'\xcirc \phi=\pi$ and $\phi\xcirc \iota=\iota'$. Then, $\phi$ is necessarily an isomorphism.

\smallskip

Let $(\iota, B, \pi)$ be an extensions of $H$ by $I$ and let $s\colon H\to B$ be a set theoretic section of $\pi$, with $s(0)=0$. For each $b\in B$, there exist unique $y\in I$ and $h\in H$ such that $b=\iota(y)+s(h)$. Moreover, since~$\iota$ and $\pi$ are morphisms and $\iota$ is injective, there exists a unique map $\beta\colon H\times H\to I$ such that
\begin{equation}\label{construccion de beta}
\iota(y)+s(h) + \iota(y')+s(h')= \iota(y) + \iota(y')+ s(h) + s(h')= \iota(y+y' + \beta(h,h')) + s(h+h').
\end{equation}
It is well known that, for all $h,h',h''\in H$,
\begin{align}
& \beta(h,h')+\beta(h+h',h'') = \beta(h',h'')+\beta(h,h'+h''),\label{condicion de cociclo}\\
&\beta(h,0) = \beta(0,h) = 0\qquad\text{and}\qquad \beta(h,h') = \beta(h',h).\label{normalidad y cociclo abeliano}
\end{align}	
In fact, the first condition is equivalent to the associativity of the sum in $B$; the second one, to the fact that $s(0)=0$, and the third one, to the commutativity of the sum in $B$. Conversely, given a map $\beta\colon H\times H\to I$ satisfying~\eqref{condicion de cociclo} and~\eqref{normalidad y cociclo abeliano}, we have an extension
\begin{equation}\label{sucesion exacta corta de grupos abelianos 2}
\begin{tikzpicture}
\begin{scope}[yshift=0cm,xshift=0cm, baseline]
			\matrix(BPcomplex) [matrix of math nodes, row sep=0em, text height=1.5ex, text
			depth=0.25ex, column sep=2.5em, inner sep=0pt, minimum height=5mm, minimum width =6mm]
			{0 & I & I\times_{\beta} H & H & 0,\\};
			\draw[->] (BPcomplex-1-1) -- node[above=1pt,font=\scriptsize] {} (BPcomplex-1-2);
			\draw[->] (BPcomplex-1-2) -- node[above=1pt,font=\scriptsize] {$\iota$} (BPcomplex-1-3);
			\draw[->] (BPcomplex-1-3) -- node[above=1pt,font=\scriptsize] {$\pi$} (BPcomplex-1-4);
			\draw[->] (BPcomplex-1-4) -- node[above=1pt,font=\scriptsize] {} (BPcomplex-1-5);
\end{scope}
\end{tikzpicture}
\end{equation}
of $H$ by $I$, in which  $I\times_{\beta} H$ is $I\times H$, endowed with the sum
\begin{equation*}
(y,h)+(y',h')\coloneqq (y+y'+\beta(h,h'),h+h'),
\end{equation*}
and the maps $\iota$ and $\pi$ are the canonical ones. For the sake of simplicity, here and subsequently, we will write $y$ in\-stead of $(y,0)$ and $w_h$ instead of $(0,h)$. So, $(y,h)=(y,0)+(0,h)=y+w_h$. With this notation $\iota(y) = y$, $\pi(y+w_h) =h$ and the sum in $I\times_{\beta} H$ reads
\begin{equation}\label{formula para suma en I times H}
(y+w_h) + (y'+w_{h'}) = y+y'+\beta(h,h')+ w_{h+h'}.
\end{equation}
It is evident that each extension of abelian groups $(\iota,B,\pi)$, endowed with a set theoretic section $s$ of $\pi$ such that $s(0)=0$, is equivalent to the extension \eqref{sucesion exacta corta de grupos abelianos 2} associated with the map $\beta$ obtained in~\eqref{construccion de beta}. Concretely, the map $\phi\colon I\times_{\beta} H\to B$ is given by $\phi(y+w_h)\coloneqq\iota(y)+s(h)$. Moreover, it is well known that two extensions $(\iota,I\times_{\beta} H,\pi)$ and $(\iota,I\times_{\beta'} H,\pi)$ are equivalent if and only if there exists a map $\varphi\colon H\to I$, such that
\begin{equation}\label{equivalencia a nivel de qrupos}
\varphi(0) = 0\qquad\text{and}\qquad \varphi(h)-\varphi(h+h')+\varphi(h')=\beta(h,h') - \beta'(h,h')\quad\text{for all $h,h'\in H$.}
\end{equation}
The map $\phi\colon I\times_{\beta} H\to I\times_{\beta'} H$ performing the equivalence is given by $\phi(y+w_h)= y+\varphi(h)+w_h$.

\begin{definition}\label{extensiones de linear cycle sets} Let $I$ and $H$ be linear cycle sets. An {\em extension} $(\iota,B, \pi)$, of $H$ by $I$, is a short exact sequence of the underlying additive abelian group of $H$ by the underlying additive abelian groups of $I$ as~\eqref{sucesion exacta corta de grupos abelianos}, in which,~additional\-ly,~$B$ is a linear cycle set and both, $\iota$ are $\pi$, are linear cycle set morphisms. Two extensions $(\iota,B,\pi)$ and $(\iota',B',\pi')$, of $H$ by $I$, are {\em equivalent} if there exists a linear cycle set morphism $\phi\colon B\to B'$, such that $\pi'\xcirc \phi=\pi$ and $\phi\xcirc \iota=\iota'$. Then, $\phi$ is an isomorphism. We let $\EXT(H;I)$ denote the set of equivalence classes of extensions of $H$ by $I$.
\end{definition}

Assume that~\eqref{sucesion exacta corta de grupos abelianos} is an extension of linear cycle sets and let $s$ be a section of $\pi$ with $s(0)=0$. For the sake of simplicity, in the rest of this section we identify $I$ with $\iota(I)$ and we write $y$ instead of $\iota(y)$. Let $\blackdiamond\colon H\times I\to I$ and $\Yleft \colon I\times H\to I$ be the maps defined by $h\blackdiamond y \coloneqq s(h)\cdot y$ and $y\Yleft h\coloneqq y\Yleft s(h) = y\cdot s(h)-s(h)$, respectively. From now on we assume that $I$ is trivial. By \cite{GGV}*{Proposition~3.3}, this implies that $\blackdiamond$ and $\Yleft$ do not depend on~$s$. Moreover, as shown in \cite{GGV}*{Subsection~3.1}, there is a map $f\colon H\times H\to I$, such that:
\begin{equation}\label{formula para cdot en I times H}
(y+s(h))\cdot (y'+s(h'))= h\blackdiamond y' + f(h,h') + h\blackdiamond y \Yleft h\cdot h' + s(h\cdot h').
\end{equation}

\begin{notation}\label{notacion ext} Fix maps $\blackdiamond\colon H\times I\to I$ and $\Yleft \colon I\times H\to I$. We let $\Ext_{\blackdiamond,\Yleft}(H;I)$ denote the set of equivalence classes of  extensions $(\iota,B,\pi)$, of $H$ by $I$, such that $s(h)\cdot y = h\blackdiamond y$ and $y\cdot s(h) -s(h) = y\Yleft h$, for every section $s$ of $\pi$ satisfying $s(0) = 0$. When $\Yleft$ is the zero map, we write $\Ext_{\blackdiamond}(H;I)$ instead of $\Ext_{\blackdiamond,0}(H;I)$.
\end{notation}

\subsection{Building extensions of linear cycle sets}\label{Building extensions of linear cycle sets}
Let $(\iota, I\times_{\beta} H, \pi)$ be the short exact sequence of abelian groups~\eqref{sucesion exacta corta de grupos abelianos 2}. We assume that  $H$ is a linear cycle set and we fix maps $\blackdiamond\colon H\times I\to I$, $\Yleft \colon I\times H\to I$ and $f\colon H\times H\to I$. We let $I\times_{\beta,f}^{\blackdiamond,\Yleft} H$ denote $I\times_{\beta} H$, endowed with the binary operation $\cdot$, defined by
\begin{equation}\label{formula para cdot en I times H'}
(y+w_h)\cdot (y'+w_{h'})= h\blackdiamond y' + f(h,h') + h\blackdiamond y \Yleft h\cdot h' + w_{h\cdot h'}.
\end{equation}

\begin{proposition}\label{condiciones para extension} Let $\triangleleft\colon I \times H\to I$ be the map defined by $y\triangleleft h \coloneqq  h\blackdiamond (y - y\Yleft h)$. For each $y\in I$ and $h\in H$, we set $y^h\coloneqq h^{\times -1}\blackdiamond (y \triangleleft h)$. By \cite{GGV}*{Theorem~5.6 and Remark~5.12} we know that $(\iota, I\times_{\beta,f}^{\blackdiamond,\Yleft} H, \pi)$ is an extension of linear cycle sets (where $I$ has the trivial structure) if and only if the following conditions are fulfilled:
\begin{align}
& y \triangleleft (h\times h') = (y \triangleleft h) \triangleleft h',\quad (y + y')\triangleleft h  = y\triangleleft h + y'\triangleleft h,\quad y \triangleleft 0 = y,\label{eqq 1}\\
& (h'\times h)\blackdiamond y  = h\blackdiamond (h' \blackdiamond y),\quad h\blackdiamond (y + y') = h\blackdiamond y + h\blackdiamond y',\quad 0\blackdiamond y = y,\label{eqq 2}\\
& y^{h+h'} + y = y^h + y^{h'},\quad 0^h = 0,\label{eqq 3}\\
& h\blackdiamond  \beta(h',h'') + f(h,h'+h'') = f(h,h')+f(h,h'')+\beta(h\cdot h',h\cdot h'')\label{eqq 4}
\shortintertext{and}
&\begin{aligned}\label{eqq5}
f(h+h',h'') + (h+h')\blackdiamond \beta(h,h')\Yleft (h+h')\cdot h'' & = (h\cdot h')\blackdiamond f(h,h'') + f(h\cdot h',h\cdot h'')\\
&+ (h\cdot h')\blackdiamond f(h,h') \Yleft (h\cdot h')\cdot(h\cdot h''),
\end{aligned}
\end{align}
for all $h,h',h''\in H$ and $y,y'\in I$.
\end{proposition}

\begin{remark}\label{comentario} Note that, by~\eqref{eqq 2}, we have $y^h = y - y\Yleft h$.
\end{remark}

\begin{remark}\label{linealidad de Yleft} Assume that conditions~\eqref{eqq 1}--\eqref{eqq 3} are satisfied and let $y\in I$ and $h\in H$. By~\eqref{eqq 3}, we have $y^0 = y$. Since $y\Yleft h = y-y^h = y - h^{\times -1}\blackdiamond (y\triangleleft h)$, by~\eqref{eqq 1}, \eqref{eqq 2} and~\eqref{eqq 3} the map $\Yleft$ is bilinear. So, $0\Yleft h = y\Yleft 0 = 0$.
\end{remark}

\begin{remark} Let $\iota\colon I\to B$ and $\pi\colon B\to H$ be morphisms of linear cycle sets. Assume that $I$ is trivial and $(\iota,B,\pi)$ is an extension of linear cycle sets. Let $s$ be a set theoretic section of $\pi$, such that $s(0)=0$. Let $\beta$ be the map obtained in~\eqref{construccion de beta} and let $\blackdiamond$, $\Yleft$ and $f$ be the maps obtained above Notation~\ref{notacion ext}. Consider now the short exact sequence of abelian groups~\eqref{sucesion exacta corta de grupos abelianos 2}, with $I\times_{\beta} H$ endowed with the operation $\cdot$, introduced in~\eqref{formula para cdot en I times H'}. By~\cite{GGV}*{Remark~4.9}, the map $\phi\colon I\times_{\beta} H\to B$, defined by $\phi(y+w_h)\coloneqq \iota(y)+s(h)$, is a bijection compatible with $+$ and~$\cdot$. Hence,~\eqref{sucesion exacta corta de grupos abelianos 2} is an extension of linear cycle sets and $\phi$ is an equivalence of extensions.
\end{remark}

\begin{remark}\label{equivalecia entre extensiones producto, caso I trivial} By~\cite{GGV}*{Remark~5.14}, two extensions $(\iota,I\times_{\beta,f}^{\blackdiamond,\Yleft} H,\pi)$ and $(\iota,I\times_{\beta',f'}^{\blackdiamond,\Yleft} H,\pi)$, of $H$ by $I$, are equivalent if and only if there exists a map $\varphi\colon H\to I$ satisfying~\eqref{equivalencia a nivel de qrupos} and
\begin{equation}\label{tercera condicion caso I trivial}
\varphi(h\cdot h') + f(h,h') = h\blackdiamond \varphi(h') + f'(h,h') + h\blackdiamond \varphi(h)\Yleft h\cdot h'\quad\text{for all $h,h'\in H$.}
\end{equation}
\end{remark}

\begin{remark}\label{condicion para I incluido en Soc} Since $I$ is trivial, $I\subseteq \Soc\bigl(I\times_{\beta,f}^{\blackdiamond,\Yleft} H\bigr)$ if and only if $y\Yleft h= y\cdot w_h-w_h = 0$, for all $y\in I$ and~$h\in H$. Since $y^h=y-y\Yleft h$, this is equivalent to the fact that $y^h=y$, for all $y\in I$ and $h\in H$. Moreover, in this case $y\triangleleft h=h\blackdiamond y$, condition~\eqref{eqq5} becomes
\begin{equation}\label{condicion equivalente compatibilidad suma cdot con hmassimp3 caso simple}
f(h+h',h'')  = (h\cdot h')\blackdiamond f(h,h'') + f(h\cdot h',h\cdot h''),
\end{equation}
and condition~\eqref{tercera condicion caso I trivial} becomes
$$
\varphi(h\cdot h') + f(h,h') = h\blackdiamond \varphi(h') + f'(h,h').
$$
Consequently, conditions~\eqref{eqq 1}--\eqref{eqq5} reduce to~\eqref{eqq 2}, \eqref{eqq 4} and~\eqref{condicion equivalente compatibilidad suma cdot con hmassimp3 caso simple}.
\end{remark}

\subsection[Cohomology of linear cycle sets]{Cohomology of linear cycle sets}\label{Cohomology of linear cycle sets}

Let $H$ be a linear cycle set. For each $s\ge 1$, we let $\mathds{Z}[H^s]$ denote the free abelian group with basis $H^s$. Let $\sh(\mathds{Z}[H^s])$ be the subgroup of $\mathds{Z}[H^s]$, generated by the shuffles
\begin{equation*}
\sum_{\sigma\in \sh_{l,s-l}} \sg(\sigma)(h_{\sigma^{-1}(1)},\dots, h_{\sigma^{-1}(s)}),
\end{equation*}
taken for all $1\le l<s$. Here $\sh_{l,s-l}$ is the subset of all the permutations $\sigma$ of $s$ elements satisfying $\sigma(1)<\cdots<\sigma(l)$ and $\sigma(l+1)<\cdots<\sigma(s)$. For each $r\ge 0$ and $s\ge 1$, let $C^N_{rs}(H)\coloneqq \mathds{Z}[H^r]\ot \ov{M}(s)$, where $\ov{M}(s)\coloneqq \frac{\mathds{Z}[H^s]}{\sh(\mathds{Z}[H^s)]}$. Given $h_1,\dots,h_s\in H$, we let $[h_1,\dots,h_s]$ denote the class of $(h_1,\dots,h_s)\in \mathds{Z}[H^s]$ in $\ov{M}(s)$. Let $I$ be an abelian group endowed with a map $\blackdiamond \colon H\times I \to I$ satisfying conditions~\eqref{eqq 2}. In this subsection we consider~$I$ as a trivial linear cycle set. Let $\wh{C}_N^{rs}(H,I)$ be the subgroup of $\Hom(C^N_{rs}(H),I)$,~con\-sisting of the morphisms $f\colon C^N_{rs}(H)\to I$, such that $f(h_1,\dots, h_r\ot [h_{r+1},\dots,h_{r+s}]) = 0$, if some $h_i=0$. Notice that $\wh{C}_N^{rs}(H,I)$ is isomorphic to the subgroup of $\Hom(\mathds{Z}[H^{r+s}], I)$, formed by all the maps $f$, that vanish on all the $(r+s)$-tuples, with some coordinate equal to zero, and that have the following property:
\begin{equation*}
\sum_{\sigma\in \sh_{l,s-l}} \sg(\sigma) f(h_1,\dots,h_r, h_{r+\sigma^{-1}(1)},\dots, h_{r+\sigma^{-1}(s)})=0,
\end{equation*}
for all $h_1,\dots,h_{r+s}\in H$ and $1\le l<s$. Consider the diagram $\bigl(\wh{C}_N^{**}(H,I),\partial_{\mathrm{h}}^{**}, \partial_{\mathrm{v}}^{**}\bigr)$, where
\begin{equation*}
\partial_{\mathrm{h}}^{r+1,s}\colon \wh{C}_N^{rs}(H,I)\to \wh{C}_N^{r+1,s}(H,I)\quad\text{and}\quad \partial_{\mathrm{v}}^{r,s+1}\colon \wh{C}_N^{rs}(H,I)\to \wh{C}_N^{r,s+1}(H,I)
\end{equation*}
are the maps defined by
\begin{align*}
& \partial_{\mathrm{h}}^{r+1,s}(f)(h_1,\dots, h_{r+1},h_{r+2},\dots,h_{r+s+1}) \coloneqq f(h_1\cdot h_2,\dots,h_1\cdot h_{r+1},h_1\cdot h_{r+2},\dots,h_1\cdot h_{r+s+1})\\
&\phantom{\partial_{\mathrm{h}}^{r+1,s}(f)(h_1,\dots,h_{r+1}} + \sum_{\jmath=1}^r (-1)^{\jmath} f(h_1,\dots, h_{\jmath-1},h_{\jmath}+h_{\jmath+1}, h_{\jmath+2},\dots,h_{r+1},h_{r+2},\dots, h_{r+s+1})\\
& \phantom{\partial_{\mathrm{h}}^{r+1,s}(f)(h_1,\dots,h_{r+1}} + (-1)^{r+1} \bigl((h_1+\cdots + h_r)\cdot h_{r+1}\bigr) \blackdiamond f(h_1,\dots,h_r,h_{r+2},\dots, h_{r+s+1})
\shortintertext{and}
& \partial_{\mathrm{v}}^{r,s+1}(f)(h_1,\dots,h_r,h_{r+1},\dots,h_{r+s+1}) \coloneqq (-1)^r f(h_1,\dots,h_r,h_{r+2},\dots,h_{r+s+1})\\
& \phantom{\partial_{\mathrm{v}}^{r,s+1}(f)(h_1,\dots,h_r} + \sum_{\jmath=r+1}^{r+s} (-1)^{\jmath} f(h_1,\dots,h_r,h_{r+1},\dots,h_{\jmath-1}, h_{\jmath}+h_{\jmath+1},h_{\jmath+2},\dots, h_{r+s+1})\\
& \phantom{\partial_{\mathrm{v}}^{r,s+1}(f)(h_1,\dots,h_r} + (-1)^{r+s+1} f(h_1,\dots,h_r,h_{r+1},\dots,h_{r+s}),
\end{align*}
respectively. Note that $\bigl(\wh{C}_N^{**}(H,I),\partial_{\mathrm{h}}^{**}, \partial_{\mathrm{v}}^{**}\bigr)$ is the diagram obtained from Figure~\ref{ddiagonales}, by deleting the arrows $D$.

\begin{theorem}[\cite{GGV}*{Theorem 6.1}]\label{principal} The diagram $\bigl(\wh{C}_N^{**}(H,I),\partial_{\mathrm{h}}^{**},\partial_{\mathrm{v}}^{**}\bigr)$ is a double cochain complex.
\end{theorem}

Let $\blackdiamond \colon H\times I \to I$ and $\Yleft \colon I\times H\to I$ be morphisms. Assume that conditions~\eqref{eqq 1}--\eqref{eqq 3} are satisfied. According to \cite{GGV}*{Section 7}, there exist maps
\begin{equation*}
D_{rs}^{r+s,1}\colon \wh{C}_N^{rs}(H,I) \to \wh{C}_N^{r+s,1}(H,I),
\end{equation*}
that yield a degree~$1$ endomorphism $D$ of $\wh{C}_N^*(H,I)\coloneqq \Tot\bigl(\wh{C}_N^{**}(H,I)\bigl)$ such that ${\mathcal{C}}_{\blackdiamond,\Yleft}(H,I)\coloneqq \bigl(\wh{C}_N^*(H,I),\partial+D\bigr)$ is a cochain complex, where $\partial\coloneqq \partial_{\mathrm{h}}+\partial_{\mathrm{v}}$. Next we recall the definition of $D$.

\begin{notation} Given $h_1,\dots,h_{r+s}$ we set $\mathrm{h}_{1,r+s} \coloneqq (h_1,\dots,h_{r+s})$.
\end{notation}

For all $r\ge0$ and $s\ge 1$, we define $D_{rs}^{r+s,1}\colon \wh{C}_N^{rs}(H,I)\to \wh{C}_N^{r+s,1}(H,I)$, by
\begin{equation*}
D_{rs}^{r+s,1}(f)(\mathrm{h}_{1,r+s+1})=(-1)^{r+s} \bigl(\bigl((h_1+\dots+h_r)\cdot (h_{r+1}+\dots+h_{r+s})\bigr)\blackdiamond f(\mathrm{h}_{1,r+s})\bigr) \Yleft (h_1+\dots+h_{r+s})\cdot h_{r+s+1}.
\end{equation*}
Finally, for $n\ge 1$, we let $D_n^{n+1}\colon \wh{C}_N^n(H,I)\to \wh{C}_N^{n+1}(H,I)$ denote the map given by
\begin{equation*}
D_n^{n+1}(f_{0n},\dots,f_{n-1,1})\coloneqq \sum_{k=1}^n D_{n-k,k}^{n1}(f_{n-k,k}).
\end{equation*}
Figure~\ref{ddiagonales} shows in black the diagram $\bigl(\wh{C}_N^{**}(H,I),\partial_{\mathrm{h}}^{**}, \partial_{\mathrm{v}}^{**}\bigr)$, and in green the maps $D_{rs}^{r+s,1}$. For the sake of legibility, in that diagram we write $\wh{C}_N^{rs}$ instead of $\wh{C}_N^{rs}(H,I)$.

\begin{figure}[H]
	\begin{center}
		\adjustbox{scale=1.4,center}{%
			\begin{tikzcd}[row sep=2em, column sep=2 * 1.61803398874988em, scale=2]
				\vdots \arrow[rrrrddd, green(ncs), shift left=0.2ex, end anchor={[yshift=0.4ex]}, end anchor={[xshift=-0.1ex]}, start anchor={[xshift=0.5ex]}, start anchor={[yshift=0.3ex]}, outer sep=-1.6pt, "D" pos=0.07] & \vdots \arrow[rrrrddd, green(ncs), shift left=0.2ex, end anchor={[yshift=1.4ex]north west}, end anchor={[xshift=-1.6ex]}, start anchor={[xshift=0.5ex]}, start anchor={[yshift=0.3ex]}, outer sep=-1.6pt, "D" pos=0.07] & \vdots \arrow[rrrdd, dash, green(ncs), shift left=0.3ex, end anchor={[yshift=-1.25ex]north west}, start anchor={[xshift=0.3ex]}, start anchor={[yshift=0.05ex]}, outer sep=-1.6pt, "D" pos=0.105] & \vdots \arrow[rrd, dash, green(ncs), shift left=0.2ex, end anchor={[yshift=-3.2ex]north west}, start anchor={[xshift=0.5ex]}, start anchor={[yshift=0.05ex]}, outer sep=-1.6pt, "D" pos=0.18]  & \vdots \arrow[rd, dash, green(ncs), shift left= 0.4ex, end anchor={[yshift=2.4ex]north west}, outer sep=-1.6pt, "D" pos=0.35] & \\
\wh{C}_N^{03} \arrow[r] \arrow[u] \arrow[rrrdd, green(ncs), shift left=0ex, end anchor={[yshift=0.6ex]}, start anchor={[xshift=-0.3ex]}, start anchor={[yshift=0.3ex]}, outer sep=-1.6pt, "D" pos=0.08] & \wh{C}_N^{13} \arrow[r, crossing over] \arrow[u, crossing over] \arrow[rrrdd, green(ncs), shift left=0ex, end anchor={[yshift=0.6ex]}, start anchor={[xshift=-0.3ex]}, start anchor={[yshift=0.3ex]}, outer sep=-1.6pt, "D" pos=0.08]& \wh{C}_N^{23} \arrow[r, crossing over] \arrow[u, crossing over] \arrow[rrrdd, green(ncs), shift left=0ex, end anchor={[yshift=1.45ex]north west}, start anchor={[xshift=-0.3ex]}, start anchor={[yshift=0.3ex]}, outer sep=-1.6pt, "D" pos=0.08] & \wh{C}_N^{33} \arrow[r, crossing over] \arrow[u, crossing over] \arrow[rrd, dash, green(ncs), shift left=0ex, end anchor={[yshift=-1.25ex]north west}, start anchor={[xshift=-0.3ex]}, start anchor={[yshift=0.05ex]}, outer sep=-1.8pt, "D" pos=0.14] & \wh{C}_N^{43} \arrow[r, crossing over] \arrow[u, crossing over] \arrow[r, dash, green(ncs), shift left=-0.5ex, end anchor={[yshift=-3.4ex]north west}, outer sep=-1.8pt, "D" pos=0.35] & \dots \\
\wh{C}_N^{02} \arrow[r] \arrow[u] \arrow[rrd, green(ncs), shift left=0ex, end anchor={[yshift=0.5ex]}, outer sep=-1.6pt, "D" pos=0.11] & \wh{C}_N^{12} \arrow[r, crossing over] \arrow[u, crossing over]  \arrow[rrd, green(ncs), shift left=0ex, end anchor={[yshift=0.5ex]}, outer sep=-1.6pt, "D" pos=0.12]& \wh{C}_N^{22} \arrow[r, crossing over] \arrow[u, crossing over]  \arrow[rrd, green(ncs), shift left=0ex, end anchor={[yshift=0.5ex]}, outer sep=-1.6pt, "D" pos=0.12]& \wh{C}_N^{32} \arrow[r, crossing over] \arrow[u, crossing over] \arrow[rrd, green(ncs), shift left=0ex, end anchor={[yshift=1.1ex]north west}, outer sep=-1.6pt, "D" pos=0.12] & \wh{C}_N^{42} \arrow[r, crossing over] \arrow[u, crossing over] \arrow[r, dash, green(ncs), shift left=-0.5ex, end anchor={[yshift=-2.5ex]north west}, outer sep=-1.6pt, "D" pos=0.36] & \dots \\
\wh{C}_N^{01} \arrow[r] \arrow[u] \arrow[r, green(ncs), shift left=0.5ex, outer sep=-1.4pt, "D" pos=0.41]  & \wh{C}_N^{11} \arrow[r] \arrow[u, crossing over] \arrow[r, green(ncs), shift left=0.5ex, outer sep=-1.4pt, "D" pos=0.43] & \wh{C}_N^{21} \arrow[r] \arrow[u, crossing over] \arrow[r, green(ncs), shift left=0.5ex, outer sep=-1.4pt, "D" pos=0.42] & \wh{C}_N^{31} \arrow[r] \arrow[u, crossing over] \arrow[r, green(ncs), shift left=0.5ex, outer sep=-1.4pt, "D" pos=0.42] & \wh{C}_N^{41} \arrow[r] \arrow[u, crossing over, outer sep=-1.6pt, pos=0.11] \arrow[r, green(ncs), shift left=0.5ex, outer sep=-1.4pt, "D" pos=0.42] & \dots
			\end{tikzcd}
		}
	\end{center}
	\caption{$\bigl(\wh{C}_N^{**}(H,I),\partial_{\mathrm{h}},\partial_{\mathrm{v}}, D\bigr)$}
	\label{ddiagonales}
\end{figure}

\begin{proposition}[\cite{GGV}*{Proposition 7.9}]\label{significado de cociclos caso general} A pair $(\beta,f)\in \wh{C}_N^2(H,I)=\wh{C}_N^{02}(H,I)\oplus \wh{C}_N^{11}(H,I)$ is a $2$-cocycle of the com\-plex ${\mathcal{C}}_{\blackdiamond,\Yleft}(H,I)$ if and only if $(\beta,-f)$ satisfies conditions~\eqref{condicion de cociclo},~\eqref{normalidad y cociclo abeliano},~\eqref{eqq 4} and~\eqref{eqq5}.
\end{proposition}

\begin{corollary}[\cite{GGV}*{Corollary 7.10}]\label{equiv caso general} There is a bijective correspondence $\ho_{\blackdiamond,\Yleft}^2(H,I) \longleftrightarrow \Ext_{\blackdiamond,\Yleft}(H,I)$. Concretely, this correspondence sends the class of a $2$-cocycle $(\beta,f)\in \wh{C}_N^2(H,I)=\wh{C}_N^{02}(H,I)\oplus \wh{C}_N^{11}(H,I)$ module coboundaries into the class of the extension $(\iota, I\times_{\beta,-f}^{\blackdiamond,\Yleft} H,\pi)$.
\end{corollary}

When $\Yleft=0$, we will write ${\mathcal{C}}_{\blackdiamond}(H,I)$ instead of ${\mathcal{C}}_{\blackdiamond,\Yleft}(H,I)$ and $\ho_{\blackdiamond}^2(H,I)$ instead of $\ho_{\blackdiamond,\Yleft}^2(H,I)$.

\section[The actions \texorpdfstring{$\blackdiamond$}{*} and \texorpdfstring{$\Yleft$}{-<} for a specific linear cycle set]{The actions \texorpdfstring{$\blackdiamond$}{*} and \texorpdfstring{$\pmb\Yleft$}{-<} for a specific linear cycle set}\label{seccion 2}
Let $p$ be a prime number, let $0<\nu\le\eta \le 2\nu$ and let $H$ be the set $\mathds{Z}_{p^{\eta}}$, endowed with the usual sum and the~oper\-ator~$\cdot$, defined by $i\cdot j\coloneqq (1-p^{\nu} i)j$, where the juxtaposition denotes the multiplication in $\mathds{Z}_{p^{\eta}}$. Using~\eqref{equivalencia linear cycle set braza} one can check that the brace associated with $H$ is $\left(\mathds{Z}_{p^{\eta}},+,\times\right)$, where $i\times j = i+j+p^{\nu} ij$. Note that, since $i^{\times p^{\eta}} = 0$, the expression $i^{\times j}$ is well defined, for $i,j\in \mathds{Z}_{p^{\eta}}$. We claim that
\begin{equation}\label{i times j}
i^{\times j} = i j + \binom{j}{2} p^{\nu} i^2\qquad\text{for all $i,j\in \mathds{Z}_{p^{\eta}}$.}
\end{equation}
To prove this it suffices to check this equality for $j\in \mathds{N}_0$, which follows by an inductive argument. Let $I$ be an abelian group. In this subsection we are going to describe the operations $\blackdiamond\colon H\times I\to I$ and $\Yleft \colon I\times H\to I$, such that the conditions~\eqref{eqq 1}--\eqref{eqq 3} are satisfied. We will use again and again the fact that
\begin{align}
&(j+j')\cdot c= (1-(j+j')p^{\nu})c=(1-jp^{\nu})(1-j'p^{\nu})c=j\cdot(j'\cdot c),\label{(j+1).c}\\
&1^{\times j}\cdot c = \left(1-\left(j + \binom{j}{2} p^{\nu}\right)p^{\nu}\right)c=(1-jp^{\nu})c=j\cdot c\label{1^times j.c}
\shortintertext{and}
&i\cdot j = j&&\text{if $p^{\eta-\nu}\mid i$.}\label{i cdot j = j}
\end{align}

\begin{remark}  Later, we will need to consider maps with variables in $H$ as maps with variables in $\mathds{Z}$ via the canonical map $\pi\colon \mathds{Z}\to H$. For the sake of simplicity we will not write the map $\pi$ in these cases. Thus, for instance, if $a\in H$ and $b\in \mathds{Z}$, the expression $a=b$ means that $a$ is the class of $b$ in $H$. Moreover, we set $i\cdot j\coloneqq (1-p^{\nu} i)j\in \mathds{Z}$, for all~$i,j\in \mathds{Z}$. Note that $\pi(i)\cdot \pi(j) = \pi(i\cdot j)$.
\end{remark}

\subsection[Case \texorpdfstring{$p$}{p} odd or \texorpdfstring{$p=2$}{p=2} and \texorpdfstring{$(\nu,\eta)\ne (1,2)$}{(v,n)=/(1,2)}]{Case \texorpdfstring{$\mathbf{p}$}{p} odd or \texorpdfstring{$\mathbf{p}\pmb{=}\mathbf{2}$}{p=2} and \texorpdfstring{$\pmb{(\nu,\eta)\ne (1,2)}$}{(v,n)=/(1,2)}}\label{subsection 2.1}

We claim that $(\mathds{Z}_{p^{\eta}},\times)$ is a cyclic group whenever $p$ is odd or $p=2$ and $(\nu,\eta)\ne (1,2)$. Let $0<i<p^{\eta}$ and let $|i|$ denote the order of $i$ in $(\mathds{Z}_{p^{\eta}},\times)$. Assume that $p\nmid i$ and write $|i|=p^j$. Consider first that $p$ is odd. By~\eqref{i times j}, $j$ is the smallest number such that
\begin{equation*}
p^{\eta} \mid p^j i+\binom{p^{j}}{2}p^{\nu}i^2= p^j i\left(1+\frac{p^j-1}{2}p^{\nu}i\right).
\end{equation*}
Since $\nu>0$, from this formula we conclude that, necessarily, $j=\eta$. Thus, $|i|=p^{\eta}$, and so $(\mathds{Z}_{p^{\eta}},\times)$ is a cyclic group having $i$ as a generator. Similarly, if $p=2$, then $j$ is the smallest number such that
\begin{equation*}
2^{\eta} \mid 2^j i+\binom{2^{j}}{2}2^{\nu}i^2= 2^j i\left(1+(2^j-1) 2^{\nu-1}i\right).
\end{equation*}
Since $\nu>1$ or $\nu=\eta =1$, necessarily, $j=\eta$, and, as before, $i$ is a generator of $(\mathds{Z}_{p^{\eta}},\times)$.

\smallskip

Since $1$ generates $(\mathds{Z}_{p^{\eta}},\times)$, for each $h\in \mathds{Z}_{p^{\eta}}$, there exists a unique $l(h)\in \mathds{Z}_{p^{\eta}}$, such that $1^{\times l(h)}=h$. A direct computation shows that $l(h)$ is given by
\begin{equation}\label{formula para l}
l(h)=\begin{cases} h-\binom{h}{2}p^{\nu}&\text{if $p$ is odd or $\eta<2\nu$,}\\ h-\binom{h}{2}(2^{\nu}+2^{2\nu-1})& \text{if $p$ is $2$ and $\eta=2\nu>2$.}
\end{cases}
\end{equation}
Thus, it suffices to define $\blackdiamond\colon H\times I\to I$ and $\Yleft\colon I\times H\to I$, for $h=1^{\times l}$ with $l\in \mathds{Z}_{p^{\eta}}$, and $y\in I$.

\smallskip

For the rest of the subsection we will assume that $p$ odd or $p=2$ and $(\nu,\eta)\ne (1,2)$. In the computations we will freely use that $\binom{a+a'}{2} = \binom{a}{2}+\binom{a'}{2}+aa'$, for all $a,a'\in \mathds{Z}$.

\begin{remark}\label{es inversa} By~\eqref{i times j}, we have $l\bigl(k+\binom{k}{2}p^{\nu}\bigr) = l\bigl(1^{\times k}\bigr) = k$, for all $k\in \mathds{Z}_{p^{\eta}}$.
\end{remark}

\begin{remark}\label{def de L(l,l')} Since $1$ generates $(\mathds{Z}_{p^{\eta}},\times)$, for each $l,l'\in \mathds{Z}$, there exists $L(l,l')$ in $\mathds{Z}_{p^{\eta}}$ such that $1^{\times l}+1^{\times l'} = 1^{\times L(l,l')}$. Moreover $L(l,l')$ is unique and by the first equality in~\eqref{equivalencia linear cycle set braza}, we have
\begin{equation}\label{a.b=}
1^{\times l}\cdot 1^{\times l'} = 1^{\times -l}\times (1^{\times l}+ 1^{\times l'}) = 1^{\times -l}\times 1^{\times L(l,l')} = 1^{\times L(l,l')-l}.
\end{equation}
\end{remark}

\begin{proposition}\label{para condiciones sin potencia} If $p$ is odd, or $p=2$ and $\eta<2\nu$, then $L(l,l') = l+l'-ll'p^{\nu}$, for all $l,l'\in \mathds{Z}$.
\end{proposition}

\begin{proof} By~\eqref{formula para l}, we have $L(l,l')= U-\binom{U}{2}p^{\nu}$, where
\begin{equation*}
U\coloneqq 1^{\times l}+1^{\times l'}=l+\binom{l}{2}p^{\nu}+l'+\binom{l'}{2}p^{\nu} = l+l'+\left(\binom{l+l'}{2} - ll'\right)p^{\nu}.
\end{equation*}
Set $x\coloneqq l+l'$ and $y\coloneqq \binom{l+l'}{2} - ll' = \binom{x}{2} - ll'$. Since $p^{\eta}\mid \binom{yp^{\nu}}{2}p^{\nu}+xyp^{2\nu}$, we have
\begin{equation*}
L(l,l') = x+yp^{\nu} - \binom{x+yp^{\nu}}{2}p^{\nu} = x+yp^{\nu} - \binom{x}{2}p^{\nu} - \binom{yp^{\nu}}{2}p^{\nu} - xyp^{2\nu} = x+yp^{\nu} - \binom{x}{2}p^{\nu} = l+l'-ll'p^{\nu},
\end{equation*}
as desired.
\end{proof}

\begin{proposition}\label{para condiciones sin potencia2} If $p=2$ and $\eta=2\nu>2$, then $L(l,l') = l+l'-ll'2^{\nu}-ll'2^{2\nu-1}$, for all $l,l'\in \mathds{Z}$.
\end{proposition}

\begin{proof} By~\eqref{formula para l},
\begin{equation*}
L(l,l')= U-U(U-1)2^{\nu-1}-U(U-1)2^{2\nu-2}=U(1+2^{\nu-1}+2^{2\nu-2})-U^2 2^{\nu-1}(1+2^{\nu-1}),
\end{equation*}
where
\begin{equation*}
U\coloneqq 1^{\times l}+1^{\times l'}=l+\binom{l}{2}2^{\nu}+l'+\binom{l'}{2}2^{\nu}=l+l'+\left(\binom{l+l'}{2}-ll'\right)2^{\nu}.
\end{equation*}
Set $x\coloneqq l+l'$ and $y\coloneqq \binom{l+l'}{2} - ll'$. Then $U=x+y2^{\nu}$, and so $U^2=x^2+xy2^{\nu+1}$. Therefore,
\begin{align*}
L(l,l')&=U(1+2^{\nu-1}+2^{2\nu-2})-U^2 2^{\nu-1}(1+2^{\nu-1})\\
&= x(1+2^{\nu-1}+2^{2\nu-2})+y2^{\nu}(1+2^{\nu-1})-U^2 2^{\nu-1}(1+2^{\nu-1})&& \text{since $2^{3\nu-2}=0$}\\
&= x + y2^{\nu}(1+2^{\nu-1})-(x^2-x)2^{\nu-1}(1+2^{\nu-1})&& \text{since $U^2 2^{\nu-1}=x^2 2^{\nu-1}$}\\
&=x + y2^{\nu}(1+2^{\nu-1})- (y+ll') 2^{\nu}(1+2^{\nu-1})&& \text{since $x^2-x=2(y+ll')$}\\
&= x-ll' 2^{\nu}-ll'2^{2\nu-1},
\end{align*}
as desired.
\end{proof}

\begin{lemma}\label{general conditions} Let $A,B\in \End(I)$ be such that $p^{\eta}B = 0$ and $[B,A] = BAB+p^{\nu}AB$. We have
\begin{equation}\label{condicion compuesta}
(A-AB)^l=A^l-\left(l+\binom{l}{2}p^{\nu}\right)A^l B\qquad\text{for all $l\in \mathds{Z}_{p^{\eta}}$}.
\end{equation}
\end{lemma}

\begin{proof} Since the right side of~\eqref{condicion compuesta} is periodic of period $p^{\eta}$, it is sufficient to prove this equality for $l\in \mathds{N}_0$. For $l=0$ and $l=1$ this is trivially true. Assume by induction that~\eqref{condicion compuesta} holds, for some $l\in \mathds{N}$. Then
\begin{align*}
(A-AB)^{l+1}&=\left(A^l-\left(l+\binom{l}{2}p^{\nu}\right)A^l B\right)(A-AB)\\
&=A^{l+1}-A^{l+1}B-\left(l+\binom{l}{2}p^{\nu}\right)A^l B(A-AB)\\
&= A^{l+1}-A^{l+1}B-\left(l+\binom{l}{2}p^{\nu}\right)A^l (1+p^{\nu})AB && \text{since $[B,A] = BAB+p^{\nu}AB$}\\
&= A^{l+1}-\left(1+l+\binom{l}{2}p^{\nu}+lp^{\nu}+\binom{l}{2}p^{2\nu}\right) A^{l+1}B\\
&= A^{l+1}-\left(1+l+\binom{l+1}{2}p^{\nu}\right) A^{l+1}B &&\text{since $p^{2\nu}B=0$ and $\binom l2+l=\binom{l+1}{2}$.}
\end{align*}
This completes the inductive step and concludes the proof of~\eqref{condicion compuesta}.
\end{proof}

\begin{proposition}\label{condiciones sin potencia general} Let $A,B\in \End(I)$ be such that
\begin{equation}\label{igualdad exponencial general}
A^{p^{\eta}}= \Ide,\quad p^{\eta}B = 0\quad\text{and}\quad [B,A] = BAB+p^{\nu}AB.
\end{equation}
Then, the maps $\blackdiamond\colon H\times I\to I$ and $\Yleft \colon I\times H\to I$, defined by
\begin{equation}\label{def de diamante general}
1^{\times l}\blackdiamond y \coloneqq A^ly\qquad\text{and}\qquad y\Yleft h\coloneqq hBy,
\end{equation}
satisfy the conditions~\eqref{eqq 1}--\eqref{eqq 3}, where $y \triangleleft h$ and $y^h$ are defined as in Proposition~\ref{condiciones para extension}. Conversely, if $\blackdiamond\colon H\times I\to I$ and $\Yleft \colon I\times H\to I$ are operations that satisfy~the~con\-ditions~\eqref{eqq 1}--\eqref{eqq 3}, then the maps $A\colon I\to I$ and $B\colon I\to I$, defined by $Ay\coloneqq 1\blackdiamond y$ and $ By\coloneqq y\Yleft 1$, are endomorphisms of $I$, that satisfy~\eqref{igualdad exponencial general} and~\eqref{def de diamante general}.
\end{proposition}

\begin{proof} $\Rightarrow$)\enspace We know that $(\mathds{Z}_{p^{\eta}},\times)$ is generated by $1$ subject to the relation $1^{\times p^{\eta}} = 0$. Consequently, since $A^{p^{\eta}} = \Ide$, the map $\blackdiamond$ is a well defined $\mathds{Z}$-linear left group action of $(\mathds{Z}_{p^{\eta}},\times)$ on $I$ (thus, conditions~\eqref{eqq 2} are satisfied). On the other hand, by~\eqref{condicion compuesta}, there is a well defined $\mathds{Z}$-linear right group action $\triangleleft$, of $(\mathds{Z}_{p^{\eta}},\times)$ on $I$ (in other words $\triangleleft$ satisfy conditions~\eqref{eqq 1}), given by
\begin{equation}\label{formula de triangleleft}
y\triangleleft 1^{\times l}\coloneqq (A-AB)^l y.
\end{equation}
Note that, since $p^{\eta}B = 0$, the map $\Yleft$ is well defined. We must prove that $\triangleleft$ is the map defined at the be\-gin\-ning of Proposition~\ref{condiciones para extension}. In other words, that
\begin{equation}\label{equivalencia entre las dos definicionesv'}
h\blackdiamond (y - y\Yleft h) =y\triangleleft h\qquad\text{for all $h\in \mathds{Z}_{p^{\eta}}$ and $y\in I$.}
\end{equation}
Concretely, we must show that
\begin{equation}\label{equivalencia entre las dos definiciones'}
A^{l(h)}\left(\Ide - h B\right)= (A -AB)^{l(h)} \qquad\text{for all $h\in \mathds{Z}_{p^{\eta}}$.}
\end{equation}
But this follows immediately from~\eqref{condicion compuesta} and the fact that $\bigl(l(h)+\binom{l(h)}{2}p^{\nu}\bigr) = 1^{\times l(h)} = h$, by~\eqref{i times j}. Consequently,
\begin{equation*}
y^l = y - y\Yleft l = (\Ide - lB )y.
\end{equation*}
where the first equality holds by Remark~\ref{comentario}. Now, a direct computation proves that this operation satisfies~\eqref{eqq 3}.

\smallskip

\noindent $\Leftarrow$)\enspace To begin with, note that by~\eqref{eqq 2} and Remark~\ref{linealidad de Yleft}, the maps $A$ and $B$ are endomorphisms of $I$ and equalities~\eqref{def de diamante general} are satisfied. Consequently, since $1^{\times p^{\eta}} = 0$, by conditions~\eqref{eqq 2} we have $A^{p^{\eta}} = \Ide$. Moreover, by the $H$-linearity of~$\Yleft$, we have $p^{\eta} B y= p^{\eta} (y\Yleft 1) =  y\Yleft p^{\eta} =  y\Yleft 0 = 0$. Finally, the last equality in~\eqref{igualdad exponencial general} follows from the fact that
\begin{equation*}
A^2(\Ide-(2+p^{\nu}) B)y = 1^{\times 2}\blackdiamond (y - y\Yleft (2+p^{\nu})) = 1^{\times 2} \blackdiamond (y - y\Yleft 1^{\times 2}) = y\triangleleft 1^{\times 2} = (A-AB)^2y,
\end{equation*}
for all $y\in I$.
\end{proof}

\begin{proposition}\label{general conditions'} Let $A,B\in \End(I)$ be such that $A^{p^{\eta}}=\Ide$. If $AB = BA$, then the last two conditions in~\eqref{igualdad exponencial general} hold if and only if one of the following mutually excluding cases holds:

\begin{enumerate}[itemsep=0.3ex, topsep=0.6ex, label=\emph{(\arabic*)}]

\item $B=0$.

\item $B\ne 0$, $B^2=0$ and $p^{\nu}B=0$.

\item $B^2\ne 0$, $p^{\eta}B=0$, $B^3=0$ and $B^2+p^{\nu} B=0$.
\end{enumerate}
\end{proposition}

\begin{proof} Since $A$ and $B$ commutes, $[B,A] = BAB+p^{\nu} AB$ if and only if $B^2+p^{\nu} B=0$. Assume this and that $p^{\eta}B=0$. Since $\eta\le 2\nu$, we have
\begin{equation*}
B^3=-p^{\nu}B^2=p^{2\nu}B=0\quad\text{and}\quad B^2 = 0 \Rightarrow p^{\nu}B = 0.
\end{equation*}
So, the last two conditions in~\eqref{igualdad exponencial general} hold if and only if one of the items~(1), (2) and~(3) is satisfied.
\end{proof}

\begin{remark}\label{para ejemplo con I ciclico} Let $I\coloneqq \mathds{Z}_n$, for some $n\coloneqq p^r c$ with $p\nmid c$. Let $A,B\in \End(I)$ and write $A\coloneqq a\Ide$ and $B\coloneqq b\Ide$ with $0\le a,b<n$. Clearly $A^{p^{\eta}} = \Ide$ if and only if $a^{p^{\eta}}\equiv 1\pmod{n}$. If $b>0$, write $b\coloneqq p^sd$ with $p\nmid d$. We have:

\begin{enumerate}

\item $B=0$ if and only if $b=0$.

\item $B\ne 0$, $B^2=0$ and $p^{\nu}B=0$ if and only if $s<r\le \min(\nu+s,2s)$ and $c\mid d$.

\item $B^2\ne 0$, $p^{\eta}B=0$, $B^3=0$ and $B^2+p^{\nu} B=0$ if and only if $s=\nu$,  $2\nu<r\le \eta+\nu$, $p^{r-2\nu}\mid d+1$ and $c\mid d$.
\end{enumerate}
\end{remark}

\subsection[Case \texorpdfstring{$p=2$}{p=2}, \texorpdfstring{$\nu=1$}{v=1} and \texorpdfstring{$\eta=2$}{n=2}]{Case \texorpdfstring{$\mathbf{p}\pmb{=}\mathbf{2}$}{p=2}, \texorpdfstring{$\pmb{\nu=1}$}{v=1} and \texorpdfstring{$\pmb{\eta=2}$}{n=2}}

Let $I$ be an abelian group. In this brief subsection we determine the operations $\blackdiamond\colon H\times I\to I$ and $\Yleft\colon I\times H\to I$, when $p=2$, $\nu = 1$ and $\eta=2$. A direct computation using that $i\times j = i+j+2ij$ shows that $(\mathds{Z}_4,\times)\simeq \mathds{Z}_2\times \mathds{Z}_2$.

\begin{proposition}\label{condiciones sin potencia particular} Let $A_1,A_2,B\in \End(I)$ be such that
\begin{equation}\label{igualdad exponencial particular}
\begin{aligned}
&4B=0, && A_1^2=A_2^2=(A_1-A_1B)^2=(A_2-2A_2B)^2=(A_1A_2-3A_1A_2B)^2 = \Ide,\\
& A_1A_2 = A_2A_1, &&(A_1-A_1B)(A_2-2A_2B) = (A_2-2A_2B)(A_1-A_1B) = A_1A_2-3A_1A_2B.
\end{aligned}
\end{equation}
Define operations $\blackdiamond\colon H\times I\to I$ and $\Yleft \colon I\times H\to I$ by
\begin{align}
& 0\blackdiamond y \coloneqq y,\quad 1\blackdiamond y \coloneqq A_1y,\quad 2\blackdiamond y \coloneqq A_2y,\quad 3\blackdiamond y \coloneqq A_1A_2y,\label{def de diamante particular}\\
& y\Yleft h\coloneqq h By\quad\text{for $0\le h< 4$.}\label{definicion de yleft particular}
\end{align}
Then, the conditions~\eqref{eqq 1}--\eqref{eqq 3} are satisfied. Conversely, if $\blackdiamond\colon H\times I\to I$ and $\Yleft \colon I\times H\to I$ are operations satisfying~\eqref{eqq 1}--\eqref{eqq 3}, then the maps $A_1\colon I\to I$, $A_2\colon I\to I$ and $B\colon I\to I$, defined by $A_1y\coloneqq 1\blackdiamond y$, $A_2y\coloneqq 2\blackdiamond y$ and $By\coloneqq y\Yleft 1$, are endomorphisms of $I$, that  satisfy~\eqref{igualdad exponencial particular}, \eqref{definicion de yleft particular} and the last equality in~\eqref{def de diamante particular}.
\end{proposition}

\begin{proof} $\Leftarrow$)\enspace The fact that $A_1$, $A_2$ and $B$ are endomorphisms of $I$, follows from Remark~\ref{linealidad de Yleft} and the second identity in~\eqref{eqq 2}. Moreover, again by Remark~\ref{linealidad de Yleft}, we know that $4B=0$. On the other hand, by the first and third conditions in~\eqref{eqq 2}, we have $A_1^2 = A_2^2 = \Ide$ and $A_1A_2=A_2A_1$. Finally, since $y\triangleleft h \coloneqq  h\blackdiamond (y - y\Yleft h)$, we have $y\triangleleft 1 = (A_1-A_1B) y$, $y\triangleleft 2 = (A_2-2A_2B) y$ and $y\triangleleft 3 = (A_1A_2-3A_1A_2B) y$. So, by the first and third conditions in~\eqref{eqq 1}, it follows that the remaining conditions in~\eqref{igualdad exponencial particular} are satisfied.

\smallskip

\noindent $\Rightarrow$)\enspace Left to the reader.
\end{proof}

\section{Computing the cohomology in some cases}\label{seccion 3}
Let $H$ be as at the beginning of Section~\ref{seccion 2}, let $I$ be an abelian group and let $\blackdiamond\colon H\times I\to I$ and $\Yleft \colon I\times H\to I$ be operations satisfying~\eqref{eqq 1}--\eqref{eqq 3}. In this section we compute $\ho_{\blackdiamond,\Yleft}^2(H,I)$ when $(\mathds{Z}_{p^{\eta}},\times)$ is cyclic, which, as we saw in Subsection~\ref{subsection 2.1}, happens if and only if $p$ is odd or $p=2$ and $(\nu,\eta)\ne (1,2)$. Consequently throughout this section we assume that we are under one of these conditions. As in Proposition~\ref{condiciones sin potencia general}, we let $A$ and $B$ denote the endomorphism of~$I$ given by $Ay\coloneqq 1\blackdiamond y$ and $ By\coloneqq y\Yleft 1$, respectively.

\begin{lemma}\label{cociclos verticales} For $\beta\in \wh{C}_N^{02}(H,I)$, we have $\beta \in \ker(\partial_{\mathrm{v}}^{03})$ if and only if \begin{equation}\label{eq1}
\beta(i,j) = \sum_{k=j}^{i+j-1} \beta(1,k)- \sum_{k=1}^{i-1} \beta(1,k)\qquad\text{for $1\le i,j<p^{\eta}$.}
\end{equation}
\end{lemma}

\begin{proof} Assume that $\beta \in \ker(\partial_{\mathrm{v}}^{03})$. Then, for all $i,j$, we have
\begin{equation*}
0 = \partial_{\mathrm{v}}^{03} \beta(1,i,j) = \beta(i,j) - \beta(i+1,j) + \beta(1,i+j) - \beta(1,i).
\end{equation*}
Thus, $\beta(i+1,j) = \beta(i,j) + \beta(1,i+j) - \beta(1,i)$. An inductive argument on $i$ using this fact proves that~\eqref{eq1} is true. Conversely, assume that~\eqref{eq1} holds. We must show that
\begin{equation*}
\beta(b,c) - \beta(a+b,c) + \beta(a,b+c) - \beta(a,b) = 0\qquad\text{for all $0\le a,b,c< p^{\eta}$}.
\end{equation*}
But this is a consequence of that $\beta(i,j) = \sum_{k=j}^{i+j-1} \beta(1,k) - \sum_{k=1}^{i-1} \beta(1,k)$ for all $i,j\in \mathds{N}$ (which follows from~\eqref{eq1}).
\end{proof}

\begin{remark}\label{generadores canonicos de los cociclos} Lemma~\ref{cociclos verticales} implies that each $\beta \in \ker(\partial_{\mathrm{v}}^{03})$ is uniquely determined by $\beta(1,1), \dots,\beta(1,p^{\eta}-1)$. For~ex\-ample, for each $1\le k <p^{\eta}$ and each $\gamma\in I$, the element $\beta_k(\gamma)$, given, for each $0\le i,j<p^{\eta}$, by
\begin{equation}\label{bala1}
\beta_k(\gamma)(i,j) \coloneqq \begin{cases} \phantom{-}\gamma &\text{if $0< i,j\le k<i+j$,}\\ -\gamma &\text{if $i,j>k$ and $i+j-p^{\eta}\le k$,}\\ \phantom{-} 0 &\text{otherwise,}\end{cases}
\end{equation}
is the unique $\beta \in \ker(\partial_{\mathrm{v}}^{03})$ with $\beta(1,k) = \gamma$ and $\beta(1,j)=0$, for $j\ne k$. Note that $\beta_k$'s are linear in $\gamma$ and
\begin{equation*}
\beta= \sum_{k=1}^{p^{\eta}-1} \beta_k(\beta(1,k))\qquad\text{for each $\beta \in \ker(\partial_{\mathrm{v}}^{03})$.}
\end{equation*}
In particular $\{\beta_k(\gamma): 1\le k<p^{\eta}\text{ and } \gamma\in I\}$ generate $\ker(\partial_{\mathrm{v}}^{03})$.
\end{remark}

For each $1\le k<p^{\eta}$ and $\gamma\in I$, let $\chi_k(\gamma)\in \wh{C}_N^{01}(H,I)$ be the map given by $\chi_k(\gamma)(i)\coloneqq\begin{cases} \gamma &\text{if $i=k$,}\\ 0 & \text{otherwise.}\end{cases}$

\begin{remark}\label{los bordes verticales} Note that $\chi_k(\gamma)$'s generate $\wh{C}_N^{01}(H,I)$. Moreover, a direct computation using Remark~\ref{generadores canonicos de los cociclos}, shows that
\begin{equation*}
\partial_{\mathrm{v}}^{02}(\chi_k(\gamma)) = \begin{cases}  2 \beta_1(\gamma) + \beta_2(\gamma) +\cdots + \beta_{p^{\eta}-1}(\gamma) & \text{if $k = 1$,}\\  \beta_k(\gamma) - \beta_{k-1}(\gamma) & \text{if $k\ne 1$.}
\end{cases}
\end{equation*}
This implies that the $\beta_{p^{\eta}-1}(\gamma)$'s generate $\frac{\ker(\partial_{\mathrm{v}}^{03})} {\ima(\partial_{\mathrm{v}}^{02})}$ and $\frac{\ker(\partial_{\mathrm{v}}^{03})} {\ima(\partial_{\mathrm{v}}^{02})}\simeq \frac{I}{p^{\eta}I}$.
\end{remark}

For the sake of brevity, from now on we will write $\alpha_1(\gamma)$ instead of $\beta_{p^{\eta}-1}(\gamma)$. Note that, by~\eqref{bala1},
\begin{equation}\label{sime}
\alpha_1(\gamma)(i,j) = \alpha_1(\gamma)(j,i) = \begin{cases} \gamma & \text{if $0\le i,j<p^{\eta}$ and $p^{\eta}\le i+j$,}\\ 0 & \text{if $0\le i,j<p^{\eta}$ and $i+j< p^{\eta}$.}\end{cases}
\end{equation}

\begin{corollary}\label{basta tomar b_1(gamma)} Each $2$-cocycle $(\beta,g) \in \wh{C}^{02}_N(H,I)\oplus \wh{C}^{11}_N(H,I)$ is equivalent to a $2$-cocycle of the form $(\alpha_1(\gamma),f)$,~for some $\gamma\in I$.
\end{corollary}

\begin{proof} This follows immediately from Remark~\ref{los bordes verticales}.
\end{proof}

In order to perform our computations we will need the following facts about $\alpha_1$.

\begin{proposition}\label{calculos de beta 1} For $0\le b<p^{\eta}$, we have:
\begin{equation}\label{acerca de beta1}
\alpha_1(\gamma)(b,1) = \begin{cases} \gamma & \text{if $b=p^{\eta}-1$,}\\ 0 &\text{otherwise,}\end{cases} \qquad\text{and}\qquad \alpha_1(\gamma) (1\cdot b,1\cdot 1) = \begin{cases} \gamma & \text{if $r(b)\ge p^{\nu}-1$,}\\ 0 &\text{otherwise,}\end{cases}
\end{equation}
where the map $r\colon \mathds{Z}_{p^{\eta}}\to \{0,\dots,p^{\eta}-1\}$, evaluated at $b$, is the canonical representative of $1\cdot b=(1-p^{\nu})b$. Consequent\-ly, $\alpha_1(\gamma)(1\cdot b,1\cdot 1) = \gamma$, for $p^{\eta}-p^{\nu}+1$ elements of $\mathds{Z}_{p^{\eta}}$.
\end{proposition}

\begin{proof} The first equality in~\eqref{acerca de beta1} follows immediately from~\eqref{sime}. We next prove the second one. Again by~\eqref{sime}, we know that $\alpha_1(\gamma) ((1-p^{\nu})b,1-p^{\nu})$ can be only $\gamma$ or zero. Moreover, $\alpha_1(\gamma) ((1-p^{\nu})b,1-p^{\nu})=\gamma$ if and only if $r(b)+p^{\eta}+1-p^{\nu}\ge p^{\eta}$, or, equivalently, if and only if $p^{\nu}-1\le r(b)<p^{\eta}$. Since the map $b\mapsto r(b)$ is a permutation, we obtain that $\alpha_1(\gamma)(1\cdot b,1\cdot 1) = \gamma$, for $p^{\eta}-p^{\nu}+1$ elements of $\mathds{Z}_{p^{\eta}}$.
\end{proof}

\begin{remark}\label{cociclos generales} A pair $(\alpha_1(\gamma),f) \in \wh{C}^{02}_N(H,I)\oplus \wh{C}^{11}_N(H,I)$ is a $2$-cocycle if and only if
\begin{equation*}
\partial_{\mathrm{v}}^{12}(f)= -\partial^{12}_{\mathrm{h}}(\alpha_1(\gamma))\quad\text{and}\quad D_{02}^{21}(\alpha_1(\gamma))+ D_{11}^{21}(f)+\partial_{\mathrm{h}}^{21}(f)=0,
\end{equation*}
which, by the definition of the differentials and equalities~\eqref{def de diamante general}, occurs if and only if, for all $a,b,c\in H$,
\begin{align}
&f(a,b+c) - f(a,b)-f(a,c)=A^{l(a)}\alpha_1(\gamma)(b,c)-\alpha_1(\gamma) (a\cdot b,a\cdot c)\label{cociclo doble}
\shortintertext{and}
&f(a+b,c)=A^{l(a\cdot b)}f(a,c)+f(a\cdot b,a\cdot c)+((a+b)\cdot c)B A^{l(a\cdot b)}f(a,b)+((a+b)\cdot c)B A^{l(a+b)}\alpha_1(\gamma)(a,b),\label{cociclo horizontal con D}
\end{align}
where the map $l\colon H\to \mathds{Z}_{p^{\eta}}$ is as in~\eqref{formula para l} and $A,B$ are as in Proposition~\ref{condiciones sin potencia general}. A direct computation shows that
\begin{equation*}
l(a\cdot b)  = \begin{cases} b-abp^{\nu}-\binom{b}{2}p^{\nu} & \text{if $p$ is odd or $\eta<2\nu$,}\\
b - ab (2^{\nu}+2^{2\nu-1})-\binom{b}{2}(2^{\nu}+2^{2\nu-1}) & \text{if $p=2$ and $\eta=2\nu>2$.}\end{cases}
\end{equation*}
\end{remark}

\subsection{First properties}

In this subsection we will prove some properties of $2$-cocycles of the form $(\alpha_1(\gamma),f)$, using only the equality~\eqref{cociclo doble}.

\begin{notation}\label{Gamma} For $b\in H$, we set $\Gamma(\gamma,b)\coloneqq A\alpha_1(\gamma)(b,1)-\alpha_1(\gamma) (1\cdot b,1\cdot 1)$.
\end{notation}

\begin{remark}\label{primeros calculos de Gamma} Note that $\Gamma(\gamma,0) = 0$ and that if $\nu=\eta$ (i.e., $\cdot$ is trivial), then, by~Proposition~\ref{calculos de beta 1},
\begin{equation*}
\Gamma(\gamma,b) = \begin{cases} 0 &\text{for $0\le b<p^{\eta}-1$,}\\ (A-\Ide)\gamma &\text{for $b=p^{\eta}-1$.}\end{cases}
\end{equation*}
\end{remark}
\begin{proposition}\label{necesarias} Assume $(\alpha_1(\gamma), f) \in \wh{C}^{02}_N(H,I)\oplus \wh{C}^{11}_N(H,I)$ is a $2$-cocycle. Then
\begin{equation}\label{obligado 1}
p^{\eta}f(1,1) = (p^{\eta}\Ide-p^{\nu}\Ide+\Ide-A)\gamma
\end{equation}
and
\begin{equation}\label{recursiva para 1'}
f(1,c)=cf(1,1)+\sum_{b=0}^{c-1} \Gamma(\gamma,b)\quad\text{for $0\le c\le p^{\eta}$.}
\end{equation}
Consequently, $f(1,c)$ is uniquely determined by $f(1,1)$ and $\gamma$.
\end{proposition}

\begin{proof} Since $(\alpha_1(\gamma), f)\in \wh{C}^{02}_N(H,I)\oplus \wh{C}^{11}_N(H,I)$ is a $2$-cocycle, by~\eqref{cociclo doble} with $a=c=1$, we have
\begin{equation}\label{recursiva para 1}
f(1,b+1) = f(1,b)+f(1,1)+ \Gamma(\gamma,b)\quad\text{for $0\le b<p^{\eta}$.}
\end{equation}
An inductive argument using this, shows that~\eqref{recursiva para 1'} holds. So, by Proposition~\ref{calculos de beta 1},
\begin{equation*}
f(1,0) = f(1,p^{\eta})= p^{\eta}f(1,1)+A\gamma - (p^{\eta}-p^{\nu}+1)\gamma,
\end{equation*}
which implies~\eqref{obligado 1}, since $f(1,0)=0$.
\end{proof}

\begin{definition}\label{def de fhat} Let $f_0,\gamma\in I$. Motivated by~\eqref{recursiva para 1'}, we define $\hat{f}\colon \mathds{N}_0\to I$ by
\begin{equation}\label{hat{f} en forma recursive}
\hat{f}(h)\coloneqq h f_0+\sum_{b=0}^{h-1} \Gamma(\gamma,b)\qquad\text{for all $h\ge 0$.}
\end{equation}
Note that $\hat{f}$ is determined by $\gamma$ and $f_0$. If convenient, we will write $\hat{f}_{f_0,\gamma}$ instead of $\hat{f}$.
\end{definition}

\begin{remark} By~\eqref{recursiva para 1'}, for each $2$-cocycle $(\alpha_1(\gamma), f)$, we have $f(1,c)=\hat{f}_{f_0,\gamma}(c)$, where $f_0=f(1,1)$ and $0\le c\le p^{\eta}$.
\end{remark}

\begin{proposition}\label{periodicoo} The following facts are equivalent:

\begin{enumerate}

\item The map $\hat{f}$ is periodic with period $p^{\eta}$,

\item $\hat{f}(p^{\eta})=0$,

\item $p^{\eta}f_0 = (p^{\eta}\Ide-p^{\nu}\Ide+\Ide-A)\gamma$.

\end{enumerate}
\end{proposition}

\begin{proof} By definition,
\begin{equation*}
\hat{f}(h+p^{\eta}) = (h+p^{\eta}) f_0+\sum_{b=0}^{p^{\eta}+h-1} \Gamma(\gamma,b) = h f_0+\sum_{b=0}^{h-1} \Gamma(\gamma,b) + p^{\eta} f_0+ \sum_{b=h}^{p^{\eta}+h-1} \Gamma(\gamma,b) = \hat{f}(h) + \hat{f}(p^{\eta}).
\end{equation*}
Consequently, $\hat{f}$ is periodic, with period $p^{\eta}$, if and only if $\hat{f}(p^{\eta})=0$. But, by Proposition~\ref{calculos de beta 1},
\begin{equation*}
\hat{f}(p^{\eta})=p^{\eta}f_0+\sum_{b=0}^{p^{\eta}-1}\bigl(A\alpha_1(\gamma)(b,1)-\alpha_1(\gamma)(1\cdot b,1\cdot 1)\bigr) = p^{\eta}f_0+A\gamma - (p^{\eta}-p^{\nu}+1)\gamma.
\end{equation*}
Hence, item~(3) is satisfied if and only if $\hat{f}(p^{\eta})=0$, which concludes the proof.
\end{proof}

In the rest of this subsection we assume that the pair $(f_0,\gamma)$ satisfies the condition in item~(3) of Proposition~\ref{periodicoo}.

\begin{proposition}\label{prop para definir f} The map $\tilde{f}\colon H\to I$, induced by $\hat{f}$, satisfies
\begin{equation}\label{cociclo doble para a uno f tilde}
\tilde{f}(b+c) - \tilde{f}(b)-\tilde{f}(c)=A\alpha_1(\gamma)(b,c)-\alpha_1(\gamma) (1\cdot b,1\cdot c)\qquad\text{for all $b,c\in H$.}
\end{equation}
\end{proposition}

\begin{proof} We proceed by induction on $c$. For $c=0$ this is trivial and, for $c=1$, this is clear by~\eqref{hat{f} en forma recursive}. Assume it is satisfied for all $b$ and some $1\le c<p^{\eta}-1$. Let $L(x,y)\coloneqq \tilde{f}(x+y)- \tilde{f}(x)-\tilde{f}(y)$. We have,
\begin{align*}
L&(b,c+1) = L(b+c,1)+L(b,c) -L(c,1)\\
&=A\alpha_1(\gamma)(b+c,1)-\alpha_1(\gamma) (1\cdot (b+c),1\cdot 1) + A\alpha_1(\gamma)(b,c)-\alpha_1(\gamma) (1\cdot b,1\cdot c)- A\alpha_1(\gamma)(c,1)+ \alpha_1(\gamma) (1\cdot c,1\cdot 1)\\
& = A\alpha_1(\gamma)(b,c+1)-\alpha_1(\gamma)(1\cdot b,1\cdot (c+1)),
\end{align*}
where the last equality follows using~\eqref{compatibilidades cdot suma y suma cdot} and evaluating the identity $\partial_{\mathrm{v}}^{03}(\alpha_1(\gamma)) = 0$ in $(b,c,1)$ and $(1\cdot b,1\cdot c,1\cdot 1)$. Thus~\eqref{cociclo doble para a uno f tilde} is satisfied, for all $b,c\in H$.
\end{proof}

The map $\tilde{f}$ depends on $\gamma$ and $f_0$. If necessary, we will write $\tilde{f}_{f_0,\gamma}$ instead of $\tilde{f}$.

\begin{lemma}\label{le prev periodico} For $0<t\le p^{\eta-\nu}$, we have $\tilde{f}(t)=t(f_0-\gamma) +\gamma$.
\end{lemma}

\begin{proof} By~\eqref{acerca de beta1}, \eqref{hat{f} en forma recursive} and the definition of $\Gamma(\gamma,b)$, we have
\begin{equation*}
\tilde{f}(t)= \hat{f}(t) = t f_0+\sum_{b=0}^{t-1} \Gamma(\gamma,b) = t f_0  - \sum_{b=0}^{t-1} \alpha_1(\gamma)(1\cdot b,1\cdot 1).
\end{equation*}
Thus, in order to check the statement, it suffices to prove that
\begin{equation}\label{ecua 3}
\alpha_1(\gamma)(1\cdot b,1\cdot 1)=\gamma\qquad\text{for $1\le b\le p^{\eta-\nu}-1$.}
\end{equation}
By~\eqref{acerca de beta1}, for this it will be sufficient to show that $p^{\nu}-1\le r(b)\le  p^{\eta}-1$, where $r$ is as in Proposition~\ref{calculos de beta 1}. But this follows from the fact that
\begin{equation*}
p^{\nu}-1\le p^{\eta-\nu} + p^{\nu} - 1 = p^{\eta} + (1-p^{\nu})(p^{\eta-\nu}-1)\le p^{\eta} + (1-p^{\nu}) b \le p^{\eta}+1-p^{\nu}\le p^{\eta}-1,
\end{equation*}
for $1\le b\le p^{\eta-\nu}-1$.
\end{proof}

\begin{lemma}\label{le periodico} For  $j\in \mathds{N}_0$, we have
\begin{equation}\label{fle periodico}
\tilde{f}(jp^{\eta-\nu}) = jp^{\eta-\nu}(f_0-\gamma) + j\gamma + \left\lfloor \frac{j}{p^{\nu}}\right\rfloor (A-\ide)\gamma.
\end{equation}
\end{lemma}

\begin{proof}  Let $b,c\in \mathds{N}_0$. Note that if $p^{\eta-\nu}$ divides $b$ and $c$, then $1\cdot b = b$ and $1\cdot c = c$ in $H$, and so, by equality~\eqref{sime},
\begin{equation*}
\alpha_1(\gamma)(1\cdot b,1\cdot c)=\alpha_1(\gamma)(b,c) = \begin{cases} \gamma & \text{if $r_{\!p^{\eta}}(b)+r_{\!p^{\eta}}(c)\ge p^{\eta}$,}\\ 0
& \text{if $r_{\!p^{\eta}}(b)+r_{\!p^{\eta}}(c) < p^{\eta}$,}\end{cases}
\end{equation*}
where $r_{\!p^{\eta}}(k)$ is the remainder of the integer division of $k$ by $p^{\eta}$. A direct inductive argument using these facts and~\eqref{cociclo doble para a uno f tilde} with $b\coloneqq (j-1)p^{\eta-\nu}$ and $c\coloneqq p^{\eta-\nu}$, gives
\begin{equation*}
\tilde{f}(jp^{\eta-\nu}) = j\tilde{f}(p^{\eta-\nu}) + \left\lfloor \frac{j}{p^{\nu}}\right\rfloor (A-\ide)\gamma.
\end{equation*}
Now~\eqref{fle periodico} follows immediately from Lemma~\ref{le prev periodico}.
\end{proof}

\begin{proposition}\label{formula mas simple de tilde f} Let $0\le i<p^{\eta}$. Write $i = sp^{\eta-\nu}+t$ with $0\le t<p^{\eta-\nu}$. We have
\begin{equation}\label{segunda}
\tilde{f}(i)=\begin{cases} i(f_0-\gamma)+s\gamma & \text{if $tp^{\nu}\le i$,}\\ i(f_0-\gamma)+(s+1)\gamma & \text{otherwise.}\end{cases}
\end{equation}
\end{proposition}

\begin{proof} Lemma~\ref{le periodico} yields the result in the case $t=0$. Assume $t>0$. By identity~\eqref{cociclo doble para a uno f tilde}
\begin{equation}\label{ecuacion doblelinea}
\begin{aligned}
\tilde{f}(sp^{\eta-\nu}+t)& = \tilde{f}(sp^{\eta-\nu})+\tilde{f}(t)+A\alpha_1(\gamma)(sp^{\eta-\nu},t)-\alpha_1(\gamma)(1\cdot sp^{\eta-\nu},1\cdot t)\\
& = \tilde{f}(sp^{\eta-\nu})+\tilde{f}(t)-\alpha_1(\gamma)(sp^{\eta-\nu}, t-tp^{\nu}),
\end{aligned}
\end{equation}
where the last equality holds because $sp^{\eta-\nu}+t=i<p^{\eta}$ (see equality~\eqref{sime}).  Combining~\eqref{ecuacion doblelinea} with Lemmas~\ref{le prev periodico} and~\ref{le periodico}, we obtain that
\begin{equation*}
\tilde{f}(sp^{\eta-\nu}+t) = (sp^{\eta-\nu}+t)(f_0-\gamma) + (s+1)\gamma - \alpha_1(\gamma)(sp^{\eta-\nu}, t-tp^{\nu}).
\end{equation*}
Thus~\eqref{segunda} follows from the equalities
\begin{equation}\label{setima}
\alpha_1(\gamma)(sp^{\eta-\nu}, t-tp^{\nu}) = \alpha_1(\gamma)(sp^{\eta-\nu}, p^{\eta}+t-tp^{\nu})= \begin{cases} \gamma & \text{if $sp^{\eta-\nu}+ p^{\eta}+t-tp^{\nu}\ge p^{\eta}$,}\\ 0 & \text{otherwise,}\end{cases}
\end{equation}
and the fact that $sp^{\eta-\nu}+ p^{\eta}+t-tp^{\nu}\ge p^{\eta}$ if and only if $tp^{\nu}\le i$.
\end{proof}

\begin{proposition}\label{coro de prop para definir f} For all $0\le k<p^{\eta}$ and $b,c\in H$, we have
\begin{equation*}
\sum_{j=0}^{k-1}A^{k-1-j}\Bigl(\tilde{f}\bigl(j\cdot (b+c)\bigr)- \tilde{f}\bigl(j\cdot b\bigr) - \tilde{f}\bigl(j\cdot c\bigr)\Bigr) = A^k\alpha_1(\gamma)(b,c)-\alpha_1(\gamma)\bigl(1^{\times k}\cdot b,1^{\times k}\cdot c\bigr).
\end{equation*}
\end{proposition}

\begin{proof} By~\eqref{compatibilidades cdot suma y suma cdot}, \eqref{(j+1).c}, \eqref{1^times j.c} and~\eqref{cociclo doble para a uno f tilde},
\begin{align*}
\sum_{j=0}^{k-1}\! A^{k-1-j}\Bigl(\!\tilde{f}\bigl(j\cdot (b\!+\!c)\bigr)\!-\!\tilde{f}\bigl(j\cdot b\bigr)\!-\!\tilde{f}\bigl(j\cdot c\bigr)\!\Bigr) &= \sum_{j=0}^{k-1}A^{k-1-j}\Bigl(A\alpha_1(\gamma)\bigl(j\cdot b,j\cdot c\bigr)-\alpha_1(\gamma) \bigl(1\cdot (j\cdot b),1\cdot (j\cdot c)\bigl)\Bigr)\\
&=\sum_{j=0}^{k-1}A^{k-j}\alpha_1(\gamma)\bigl(j\cdot b,j\cdot c\bigr)\!-\!\sum_{j=0}^{k-1}A^{k-1-j}\alpha_1(\gamma) \bigl((j\!+\!1)\cdot b,(j\!+\!1)\cdot c\bigr)\\
&= A^k\alpha_1(\gamma)(b,c)-\alpha_1(\gamma)\bigl(k\cdot b,k\cdot c\bigr)\\
&= A^k\alpha_1(\gamma)(b,c)-\alpha_1(\gamma)\bigl(1^{\times k}\cdot b,1^{\times k}\cdot c\bigr),
\end{align*}
as desired.
\end{proof}

\begin{proposition} For $1\le s\le p^{\nu}$, we have
\begin{equation*}
\sum_{j=0}^{sp^{\eta-\nu}-1} A^{sp^{\eta-\nu}-1-j}\tilde{f}(j\cdot 1) = A^{sp^{\eta-\nu}-1}(f_0 + s\gamma-\gamma) + \sum_{j=1}^{sp^{\eta-\nu}-1} (1+jp^{\nu})A^{j-1}(f_0-\gamma) + p^{2\nu-\eta} \sum_{j=1}^{sp^{\eta-\nu}-1} j A^{j-1} \gamma,
\end{equation*}
\end{proposition}

\begin{proof} Note that
\begin{equation}\label{formula base'''}
\sum_{j=0}^{sp^{\eta-\nu}-1}A^{sp^{\eta-\nu}-1-j}\tilde{f}(j\cdot 1) = \sum_{j=0}^{sp^{\eta-\nu}-1}A^{sp^{\eta-\nu}-1-j}\tilde{f}(1-jp^{\nu}) = A^{sp^{\eta-\nu}-1} f_0 + \sum_{j=1}^{sp^{\eta-\nu}-1}A^{j-1} \tilde{f}(1+jp^{\nu}).
\end{equation}
By Proposition~\ref{prop para definir f}, we have
\begin{equation}\label{descomposicion'''}
\tilde{f}(1+jp^{\nu})=\tilde{f}(1)+\tilde{f}(jp^{\nu})+A\alpha_1(\gamma)(1,jp^{\nu})-\alpha_1(\gamma)(1\cdot 1,1\cdot jp^{\nu}).
\end{equation}
So, we proceed by calculating the sum corresponding to each of these four terms. By~\eqref{acerca de beta1}, we know that
\begin{equation}\label{tercer sumando'''}
\sum_{j=1}^{sp^{\eta-\nu}-1} A^j \alpha_1(\gamma)(1,jp^{\nu}) = 0.
\end{equation}
Let $r$ be as in Proposition~\ref{calculos de beta 1}. For $0<j<sp^{\eta-\nu}$, write $j=up^{\eta-\nu}+t$ with $0\le t<p^{\eta-\nu}$. Since, in $\mathds{Z}_{p^{\eta}}$,
\begin{equation*}
1\cdot j p^{\nu}=(1-p^{\nu})jp^{\nu} = jp^{\nu} = up^{\eta}+tp^{\nu},
\end{equation*}
we have $r(jp^{\nu}) = tp^{\nu} \ge p^{\nu}-1$ if and only if $0<t<p^{\eta-\nu}$. Consequently, by~\eqref{acerca de beta1},
\begin{equation}\label{cuarto sumando'''}
\sum_{j=1}^{sp^{\eta-\nu}-1} A^{j-1} \alpha_1(\gamma)(1\cdot 1,1\cdot jp^{\nu}) = \sum_{u=0}^{s-1}\sum_{t=1}^{p^{\eta-\nu}-1} A^{up^{\eta-\nu}+t-1}\gamma.
\end{equation}
In order to compute $\sum_{j=1}^{sp^{\eta-\nu}-1}A^{j-1} \tilde{f}(jp^{\nu})$, we note that, by Lemma~\ref{le periodico},
\begin{equation*}
\tilde{f}(jp^{\nu}) =\tilde{f}(jp^{2\nu-\eta}p^{\eta-\nu}) = jp^{\nu} (f_0-\gamma) + jp^{2\nu-\eta} \gamma + \left\lfloor \frac{j}{p^{\eta-\nu}}\right\rfloor(A-\Ide)\gamma\qquad\text{for $0<j<sp^{\eta-\nu}$.}
\end{equation*}
Hence
\begin{equation}\label{f tilde de p a la nu'''}
\sum_{j=1}^{sp^{\eta-\nu}-1} A^{j-1} \tilde{f}(jp^{\nu}) = p^{\nu} \sum_{j=1}^{sp^{\eta-\nu}-1} j A^{j-1} (f_0-\gamma) + p^{2\nu-\eta} \sum_{j=1}^{sp^{\eta-\nu}-1} j A^{j-1} \gamma + \sum_{j=1}^{sp^{\eta-\nu}-1} \left\lfloor \frac{j}{p^{\eta-\nu}}\right\rfloor(A^j-A^{j-1})\gamma.
\end{equation}
Note now that
\begin{equation*}
\sum_{j=1}^{sp^{\eta-\nu}-1} \left\lfloor \frac{j}{p^{\eta-\nu}}\right\rfloor(A^j-A^{j-1})\gamma = \sum_{u=1}^{s-1} u A^{(u+1)p^{\eta-\nu}-1}\gamma -  \sum_{u=1}^{s-1} u A^{up^{\eta-\nu}-1}\gamma = (s-1) A^{sp^{\eta-\nu}-1}\gamma - \sum_{u=1}^{s-1} A^{up^{\eta-\nu}-1}\gamma.
\end{equation*}
Combining this with~\eqref{formula base'''}, \eqref{descomposicion'''}, \eqref{tercer sumando'''}, \eqref{cuarto sumando'''} and~\eqref{f tilde de p a la nu'''}, and taking into account that $\tilde{f}(1) = f_0$, we obtain
\begin{equation*}
\sum_{j=0}^{sp^{\eta-\nu}-1} A^{sp^{\eta-\nu}-1-j}\tilde{f}(j\cdot 1) = A^{sp^{\eta-\nu}-1}(f_0 + s\gamma-\gamma) + \sum_{j=1}^{sp^{\eta-\nu}-1} (1+jp^{\nu})A^{j-1}(f_0-\gamma) + p^{2\nu-\eta} \sum_{j=1}^{sp^{\eta-\nu}-1} j A^{j-1} \gamma,
\end{equation*}
as desired.
\end{proof}

From now on we set
\begin{equation}\label{los polinomios1}
P(X)\coloneqq \sum_{j=0}^{p^{\eta}-1} (jp^{\nu}+1)X^j\quad \text{and}\quad Q(X)\coloneqq  \sum_{j=0}^{p^{\eta}-1}jp^{2\nu-\eta}X^j + p^{\nu}.
\end{equation}
Let $(\alpha_1(\gamma),f)\in \wh{C}^{02}_N(H,I)\oplus \wh{C}^{11}_N(H,I)$ be a $2$-cocycle. Assuming that $\Yleft=0$, in Proposition~\ref{necesarias con Yleft cero} we will prove that $f(0,1)=\sum_{j=0}^{p^{\eta}-1}A^{p^{\eta}-1-j}\tilde{f}(j\cdot 1)$, where $\tilde f(c)=f(1,c)$. Motivated by this~equal\-ity, we establish the following result:

\begin{corollary}\label{periodico} The following equality holds:
\begin{equation}\label{formula periodico}
\sum_{j=0}^{p^{\eta}-1}A^{p^{\eta}-j}\tilde{f}(j\cdot 1) = P(A) (f_0-\gamma) + Q(A)\gamma.
\end{equation}
\end{corollary}

\begin{proof} Follows directly from the previous proposition taking $s\coloneqq p^{\nu}$,  multiplying by $A$ and using $A^{p^{\eta}}=\Ide$.
\end{proof}

\begin{proposition}\label{previa a cociclos} Let $c\in \mathds{Z}$ and let $0\le s \le p^{\nu}$. If $0\le c < p^{\eta}$ or $s=p^{\nu}$, then
\begin{equation}\label{formula para sum_{j=0}^{p^{eta}-1} A^{p^{eta}-1-j} tilde{f}(j cdot c)1}
\sum_{j=0}^{sp^{\eta-\nu}-1} A^{sp^{\eta-\nu}-j} \tilde{f}(j\cdot c) = c \sum_{j=0}^{sp^{\eta-\nu}-1} A^{sp^{\eta-\nu}-j} \tilde{f}(j\cdot 1).
\end{equation}
\end{proposition}

\begin{proof}  Assume that $s=p^{\nu}$. First we prove the case $c\in \mathds{N}_0$, by induction on $c$. For $c\in\{0,1\}$ it is trivial.  Suppose that it is true for some $c\in \mathds{N}$. By Proposition~\ref{prop para definir f}~and~equal\-ities~\eqref{compatibilidades cdot suma y suma cdot} and~\eqref{(j+1).c},
\begin{align*}
\tilde{f}(j\cdot (c+1)) & = \tilde{f}(j\cdot c+ j\cdot 1)\\
& = \tilde{f}(j\cdot c)+ \tilde{f}(j\cdot 1)+ A\alpha_1(\gamma)(j\cdot c,j\cdot 1) - \alpha_1(\gamma)(1\cdot (j\cdot c),1\cdot (j\cdot 1))\\
& = \tilde{f}(j\cdot c)+ \tilde{f}(j\cdot 1)+ A\alpha_1(\gamma)(j\cdot c,j\cdot 1) - \alpha_1(\gamma)((j+1)\cdot c,(j+1)\cdot 1).
\end{align*}
On the other hand
\begin{align*}
& \sum_{j=0}^{sp^{\eta-\nu}-1} A^{sp^{\eta-\nu}-j} \Bigl(A\alpha_1(\gamma)(j\cdot c,j\cdot 1) - \alpha_1(\gamma)((j+1)\cdot c,(j+1)\cdot 1)\Bigr)\\
& = \sum_{j=0}^{s p^{\eta-\nu}-1} A^{sp^{\eta-\nu}-j+1}\alpha_1(\gamma)(j\cdot c,j\cdot 1) - \sum_{j=1}^{sp^{\eta-\nu}} A^{sp^{\eta-\nu}-j+1} \alpha_1(\gamma)(j\cdot c,j\cdot 1)\\
& = (A^{sp^{\eta-\nu}+1}-A)\alpha_1(\gamma)(c,1)\\
& = 0,
\end{align*}
where the last equality  holds since $A^{p^{\eta}}=\ide$, and the second one holds by~\eqref{i cdot j = j}. From this it follows that
\begin{equation*}
\sum_{j=0}^{sp^{\eta-\nu}-1} A^{sp^{\eta-\nu}-j} \tilde{f}(j\cdot (c+1)) = \sum_{j=0}^{sp^{\eta-\nu}-1} A^{sp^{\eta-\nu}-j} \tilde{f}(j\cdot c) + \sum_{j=0}^{sp^{\eta-\nu}-1} A^{sp^{\eta-\nu}-j} \tilde{f}(j\cdot 1) = (c+1)\sum_{j=0}^{sp^{\eta-\nu}-1} A^{sp^{\eta-\nu}-j} \tilde{f}(j\cdot 1),
\end{equation*}
where the last equality follows from the inductive hypothesis. This finishes the proof of~\eqref{formula para sum_{j=0}^{p^{eta}-1} A^{p^{eta}-1-j} tilde{f}(j cdot c)1} for $c\in \mathds{N}_0$. Thus
\begin{equation*}
p^{\eta} \sum_{j=0}^{sp^{\eta-\nu}-1} A^{sp^{\eta-\nu}-j} \tilde{f}(j\cdot 1) = \sum_{j=0}^{sp^{\eta-\nu}-1} A^{sp^{\eta-\nu}-j} \tilde{f}(j\cdot 0) = 0,
\end{equation*}
and so~\eqref{formula para sum_{j=0}^{p^{eta}-1} A^{p^{eta}-1-j} tilde{f}(j cdot c)1} is true for $c\in \mathds{Z}$. When $s\ne p^{\nu}$, but
$0\le c< p^{\eta}$, the same proof works, noting that, by Proposition~\ref{calculos de beta 1}, we have $\alpha_1(\gamma)(c,1)=0$, for $0\le c <p^{\eta}-1$.
\end{proof}

\begin{corollary}\label{previa a cociclos1}  For all $c\in \mathds{Z}$, we have
\begin{equation}\label{formula para sum_{j=0}^{p^{eta}-1} A^{p^{eta}-1-j} tilde{f}(j cdot c)}
\sum_{j=0}^{p^{\eta}-1} A^{p^{\eta}-j} \tilde{f}(j\cdot c) = c\bigl(P(A) (f_0-\gamma) + Q(A)\gamma\bigr).
\end{equation}
\end{corollary}

\begin{proof} By Corollary~\ref{periodico} and Proposition~\ref{previa a cociclos}.
\end{proof}

\begin{corollary}\label{anularse implica invarianza} For all $M,c\in \mathds{Z}$, we have:
\begin{equation*}
\sum_{j=0}^{p^{\eta}-1}A^{M-j}\tilde{f}(j\cdot c)= (M\cdot c)\sum_{j=0}^{p^{\eta}-1}A^{p^{\eta}-j}\tilde{f}(j\cdot 1).
\end{equation*}
\end{corollary}

\begin{proof} By Proposition~\ref{previa a cociclos} and equality~\eqref{(j+1).c}, we have
\begin{equation*}
(M\cdot c)\sum_{j=0}^{p^{\eta}-1}A^{p^{\eta}-j}\tilde{f}(j\cdot 1) = \sum_{j=0}^{p^{\eta}-1}A^{p^{\eta}-j}\tilde{f}(j\cdot (M\cdot c)) = \sum_{j=0}^{p^{\eta}-1}A^{p^{\eta}-j}\tilde{f}((j+M)\cdot c) = \sum_{j=0}^{p^{\eta}-1} A^{M-j}\tilde{f}(j\cdot c),
\end{equation*}
as desired.
\end{proof}

\begin{remark}\label{hat f para cociclo} Let $(\alpha_1(\gamma),f)\in\wh{C}^{02}_N(H,I)\oplus \wh{C}^{11}_N(H,I)$ be a $2$-cocycle. Set $f_0\coloneqq f(1,1)$ and let $\hat{f}$ and $\tilde{f}$ be as in~Defi\-nition~\ref{def de fhat} and Proposition~\ref{prop para definir f}, respectively. By Proposition~\ref{necesarias}, the pair $(f_0,\gamma)$ satisfies the condition required above Proposition~\ref{prop para definir f}. Comparing~\eqref{recursiva para 1'} with~\eqref{hat{f} en forma recursive}, we get $\tilde{f}(b) = \hat{f}(b) = f(1,b)$, for all $0\le b<p^{\eta}$.
\end{remark}

\subsection[Cocycles when \texorpdfstring{$\Yleft=0$}{-<=0}]{Cocycles when \texorpdfstring{$\pmb{\Yleft=0}$}{-<=0}}\label{subsection caso < = 0}

In this subsection we assume $\Yleft=0$ (equivalently, $B=0$). As we saw above, a pair $(\alpha_1(\gamma),f) \in \wh{C}^{02}_N(H,I)\oplus \wh{C}^{11}_N(H,I)$ is  a $2$-cocycle if and only if conditions~\eqref{cociclo doble} and~\eqref{cociclo horizontal con D} hold. Under our assumption~\eqref{cociclo doble} is not modified, but~\eqref{cociclo horizontal con D} be\-comes
\begin{equation}
f(a+b,c)  = A^{l(a\cdot b)} f(a,c)+f(a\cdot b,a\cdot c)\qquad\text{for all $a,b,c\in H$,}\label{cociclo horizontal}
\end{equation}
where the map $l$ is as in~\eqref{formula para l} and $A$ is as in Proposition~\ref{condiciones sin potencia general}.

\begin{proposition}\label{necesarias con Yleft cero} Assume that $(\alpha_1(\gamma), f) \in \wh{C}^{02}_N(H,I)\oplus \wh{C}^{11}_N(H,I)$ is a $2$-cocycle. Then
\begin{equation}
f(1^{\times k},c) = \sum_{j=0}^{k-1}A^{k-1-j}f(1,j\cdot c)\qquad\text{for all $k\in \mathds{N}_0$ and $c\in H$.}\label{recursiva general}
\end{equation}
Consequently, by Proposition~\ref{necesarias}, the map $f$ is uniquely determined by $f(1,1)$ and $\gamma$.
\end{proposition}

\begin{proof} For $k=0$, the equality~\eqref{recursiva general} holds since $f(0,c) = 0$. Assume this is true for some $k\in \mathds{N}_0$. Since, in $\mathds{Z}_{p^{\eta}}$,
\begin{equation*}
b\coloneqq k + \binom{k+1}{2}p^{\nu} = k+\binom{k}{2}p^{\nu}+kp^{\nu}\qquad\text{and}\qquad bp^{\nu} = kp^{\nu},
\end{equation*}
by~\eqref{i times j}, we have
\begin{equation*}
1\cdot b = (1-p^{\nu})b = k + \binom{k}{2}p^{\nu}  = 1^{\times k}.
\end{equation*}
Thus,  $l(1\cdot b)$ is the class of $k$ in $\mathds{Z}_{p^{\eta}}$. So,
\begin{align*}
f(1^{\times (k+1)},c) & = f(1+b,c)&&\text{by~\eqref{i times j}}\\
& = A^{l(1\cdot b)} f(1,c)+ f\bigl(1\cdot b,1\cdot c\bigr) &&\text{by~\eqref{cociclo horizontal}}\\
& = A^k f(1,c)+ f\bigl(1^{\times k},1\cdot c\bigr) \\
& = A^k f(1,c)+\sum_{j=0}^{k-1}A^{k-1-j}f(1,j\cdot (1\cdot c)) &&\text{by the inductive hypothesis}\\
& = A^k f(1,c)+\sum_{j=0}^{k-1}A^{k-1-j}f(1,(j+1)\cdot c) &&\text{by~\eqref{(j+1).c}}\\
& = \sum_{j=0}^{k}A^{k-j}f(1,j\cdot c),
\end{align*}
which completes the inductive step and proves~\eqref{recursiva general}.
\end{proof}

Let $\gamma,f_0\in I$. Our next goal is to construct a $2$-cocycle $(\alpha_1(\gamma), f)\in \wh{C}^{02}_N(H,I)\oplus \wh{C}^{11}_N(H,I)$ with $f(1,1)= f_0$,~when\-ever it is possible.

\begin{definition}\label{definicion de breve f} Assume that the condition in item~(3) of Proposition~\ref{periodicoo} is satisfied and let $\tilde{f}$ be as in Proposition~\ref{prop para definir f}. Motivated by~\eqref{recursiva general}, we define $f\colon H\times H\to I$, by
\begin{equation}\label{recursiva general definicion}
f(1^{\times k},c) \coloneqq \sum_{j=0}^{k-1}A^{k-1-j}\tilde{f}(j\cdot c)\qquad\text{for $0\le k< p^{\eta}$ and $c\in H$.}
\end{equation}
Note that, by~\eqref{hat{f} en forma recursive}, the map $f$ is determined by $f_0$ and $\gamma$. If convenient, we will~write~$f_{f_0,\gamma}$~in\-stead of $f$.
\end{definition}

\begin{remark}\label{si cociclo da lo mismo} Suppose $(\alpha_1(\gamma),f) \in \wh{C}^{02}_N(H,I)\oplus \wh{C}^{11}_N(H,I)$ is a $2$-cocycle and set $f_0\coloneqq f(1,1)$. Then,~by~Remark~\ref{hat f para cociclo} and Propo\-sition~\ref{necesarias con Yleft cero}, the map $f_{f_0,\gamma}$ in Definition~\ref{definicion de breve f}, coincides with $f$.
\end{remark}

\begin{theorem}\label{cociclos} Each $2$-cocycle $(\beta,g)\in \wh{C}^{02}_N(H,I)\oplus \wh{C}^{11}_N(H,I)$ is equivalent modulo coboundaries to a $2$-cocycle of the form $(\alpha_1(\gamma), f)$. Moreover, if we set $f_0\coloneqq f(1,1)$, then
\begin{equation}\label{necesarias y suficientes}
p^{\eta}f_0 = (p^{\eta}\Ide-p^{\nu}\Ide + \Ide -A)\gamma \qquad\text{and}\qquad P(A)(f_0-\gamma) + Q(A)\gamma = 0,
\end{equation}	
where $P$ and $Q$ are as in~\eqref{los polinomios1}. Conversely, if $\gamma, f_0\in I$ satisfy conditions~\eqref{necesarias y suficientes}, then $\bigl(\alpha_1(\gamma),f_{f_0,\gamma}\bigr)$ is a $2$-cocycle.
\end{theorem}

\begin{proof} By Corollary~\ref{basta tomar b_1(gamma)}, we know that there exist $\gamma$ and $f$ such that  $(\alpha_1(\gamma),f)$ is a $2$-cocycle equivalent to $(\beta,g)$. By Proposi\-tion~\ref{necesarias}, the map $f$ satisfies the first condition in~\eqref{necesarias y suficientes}. Also,  by Corollary~\ref{periodico}, Remark~\ref{hat f para cociclo} and Proposition~\ref{necesarias con Yleft cero},
\begin{equation*}
P(A)(f_0-\gamma) + Q(A) \gamma = \sum_{j=0}^{p^{\eta}-1}A^{p^{\eta}-j}\tilde{f}(j\cdot 1) = \sum_{j=0}^{p^{\eta}-1}A^{p^{\eta}-j}f(1,j\cdot 1) = Af(0,1) = 0,
\end{equation*}
as desired. Assume conversely that $f_0$ and $\gamma$ satisfy~\eqref{necesarias y suficientes} and let $f$ be as in~\eqref{recursiva general definicion}. By definition,
\begin{equation*}
f(0,c) = f(1^{\times 0},c) = 0\quad\text{for all $c\in H$.}
\end{equation*}
Moreover,  $\tilde{f}(j\cdot 0) = \tilde{f}(0) = 0$, for all $j\in H$, and so
\begin{equation*}
f(1^{\times k},0) = 0\quad\text{for all $0\le k < p^{\eta}$.}
\end{equation*}
In order to finish the proof we must show that $f$ satisfies conditions~\eqref{cociclo doble} and~\eqref{cociclo horizontal}. By identity~\eqref{recursiva general definicion} and Proposi\-tion~\ref{coro de prop para definir f}, we have
\begin{equation*}
f(1^{\times k},b+c)- f(1^{\times k},b)-f(1^{\times k},c) = A^k\alpha_1(\gamma)(b,c)-\alpha_1(\gamma)\bigl(1^{\times k}\cdot b,1^{\times k}\cdot c\bigr)\qquad\text{for $0\le k< p^{\eta}$.}
\end{equation*}
Thus, condition~\eqref{cociclo doble} holds. To prove~\eqref{cociclo horizontal} we set $a\coloneqq 1^{\times l}$ and $b\coloneqq 1^{\times l'}$, with $0\le l,l'<p^{\eta}$, and we com\-pute the terms in that expression. First note that, by  Corollary~\ref{previa a cociclos1} and the fact that $P(A)(f_0-\gamma) + Q(A)\gamma = 0$, we have
\begin{equation}\label{pepe}
\sum_{j=0}^{p^{\eta}-1} A^{p^{\eta}-1-j} \tilde{f}(j\cdot c) = 0\qquad\text{for all $c\in H$.}
\end{equation}
Hence,
\begin{equation}\label{cuenta previa}
f(1^{\times r_{\!p^{\eta}}(k)},c)= \sum_{j=0}^{r_{\!p^{\eta}}(k)-1}A^{r_{\!p^{\eta}}(k)-j-1}\tilde{f}(j\cdot c)= \sum_{j=0}^{k-1}A^{k-j-1}\tilde{f}(j\cdot c) \qquad\text{for all $k\in \mathds{N}_0$,}
\end{equation}
where $r_{\!p^{\eta}}(k)$ is the remainder of the integer division of $k$ by $p^{\eta}$. Let $L(l,l')$ be as in Remark~\ref{def de L(l,l')} and let $M>l$ be such that the class of $M$ in $\mathds{Z}_{p^{\eta}}$ is $L(l,l')$. On the left hand side of the equality~\eqref{cociclo horizontal}, we have
\begin{equation}\label{lado izquierdo}
f(a+b,c) = f(1^{\times L(l,l')},c) = f(1^{\times r_{\!p^{\eta}(M)}},c) = \sum_{j=0}^{M-1}A^{M-j-1}\tilde{f}(j\cdot c),
\end{equation}
where the last equality holds by equality~\eqref{cuenta previa}. On the other hand, by equalities~\eqref{a.b=}, we have $a\cdot b  = 1^{\times M-l}$. So, by Remark~\ref{es inversa} and  formula~\eqref{recursiva general definicion},
\begin{equation}\label{lado derecho 1}
A^{l(a\cdot b)}f(a,c)= A^{r_{\!p^{\eta}}(M-l)} f(1^{\times l},c)= \sum_{j=0}^{l-1} A^{M-j-1}\tilde{f}(j\cdot c);
\end{equation}
while, by~\eqref{(j+1).c}, \eqref{1^times j.c} and~\eqref{cuenta previa}, we have
\begin{equation}\label{lado derecho segundo}
f\bigl(a\cdot b,a\cdot c\bigr) = f\bigl(1^{\times M-l},1^{\times l}\cdot c\bigr) = f\bigl(1^{\times M-l},l\cdot c\bigr)= \sum_{j=0}^{M-l-1}A^{M-l-j-1}\tilde{f}\bigl(j\cdot (l\cdot c)\bigr)=\sum_{i=l}^{M-1}A^{M-i-1}\tilde{f}(i\cdot c).
\end{equation}
Comparing~\eqref{lado izquierdo} with the sum of~\eqref{lado derecho 1} and~\eqref{lado derecho segundo}, we obtain~\eqref{cociclo horizontal}, which concludes the proof.
\end{proof}

\begin{theorem}\label{cobordes} A $2$-cocycle of the form $(\alpha_1(\gamma),f)\in \wh{C}^{02}_N(H,I)\oplus \wh{C}^{11}_N(H,I)$ is a $2$-coboundary if and only~if there exists $t_0\in I$, such that
\begin{equation}\label{coborde actua}
f(1,1) = (p^{\eta}\Ide-p^{\nu}\Ide+\Ide-A)t_0\qquad\text{and}\qquad \gamma = p^{\eta}t_0.
\end{equation}
\end{theorem}

\begin{proof} Let $t\in \wh{C}^{01}_N(H,I)$. By Remark~\ref{generadores canonicos de los cociclos} we know that
\begin{equation*}
\partial_{\mathrm{v}}^{02}(t)=\alpha_1(\gamma) \Longleftrightarrow \partial_{\mathrm{v}}^{02}(t)(1,p^{\eta}-1) = \gamma\text{ and } \partial_{\mathrm{v}}^{02}(t)(1,j)=0,\text{ for $0\le j< p^{\eta}-1$.}
\end{equation*}
An inductive computation, using this and that $\partial_{\mathrm{v}}^{02}(t)(1,j)=t(j)-t(j+1)+t(1)$, shows that
\begin{equation}\label{equiv a beta1 es borde}
\partial_{\mathrm{v}}^{02}(t)=\alpha_1(\gamma)\Longleftrightarrow  t(j)=jt(1),\text{ for } 1\le j< p^{\eta} \text{ and } t(0)=t(p^{\eta}) = p^{\eta} t(1) - \gamma.
\end{equation}
Since $t(0)=0$, if $\partial_{\mathrm{v}}^{02}(t)=\alpha_1(\gamma)$, this implies $\gamma=p^{\eta}t(1)$. Assume that $f=\partial_{\mathrm{h}}^{11}(t)$. Then, by the definition of~$\partial_{\mathrm{h}}^{11}$~and the first equal\-ity in~\eqref{def de diamante general}, we have
\begin{equation*}
f(a,b)= \partial_{\mathrm{h}}^{11}(t)(a,b)= t(a\cdot b) -a\blackdiamond t(b)=t(a\cdot b) - A^{l(a)}t(b).
\end{equation*}
A direct computation using this and \eqref{equiv a beta1 es borde}, yields
\begin{equation*}
f(1,1)=t(1-p^{\nu}) - At(1) = (p^{\eta}\Ide - p^{\nu}\Ide+ \Ide-A)t(1),
\end{equation*}
which concludes the proof that if $\bigl(\partial_{\mathrm{v}}^{02}(t),\partial_{\mathrm{h}}^{11}(t)\bigr) = (\alpha_1(\gamma),f)$, then~\eqref{coborde actua} is satisfied with $t_0\coloneqq t(1)$. Conversely,~as\-sume~\eqref{coborde actua} is satisfied and define $t\in \wh{C}^{01}_N(H,I)$ by $t(j)\coloneqq jt_0$, for $0\le j<p^{\eta}$. Since $t(p^{\eta}) = 0 =  p^{\eta} t_0 - \gamma$, we can apply condition~\eqref{equiv a beta1 es borde}, in order to obtain that $\partial_{\mathrm{v}}^{02}(t) = \alpha_1(\gamma)$. Moreover,
\begin{equation*}
\partial_{\mathrm{h}}^{11}(t)(1,1)=t(1-p^{\nu}) - At(1) = (p^{\eta}\Ide - p^{\nu}\Ide+ \Ide-A)t_0.
\end{equation*}
Consequently, $\partial_{\mathrm{h}}^{11}(t)(1,1)= f(1,1)$, which concludes the proof, since $\partial_{\mathrm{v}}^{02}(t) = \alpha_1(\gamma)$ and, by Proposition~\ref{necesarias con Yleft cero}, in a $2$-cocycle of the form $(\alpha_1(\gamma),f)$, the map $f$ is univocally determined by $f(1,1)$  and $\gamma$.
\end{proof}

Let $F_1,F_2\colon I^2 \to I$ and $G\colon I\to I^2$, be the maps given by
\begin{equation*}
F_1(y_1,y_2) \coloneqq p^{\eta}y_1-V y_2,\quad F_2(y_1,y_2) \coloneqq P(A)(y_1-y_2)+Q(A)y_2 \quad\text{and}\quad G(y) \coloneqq (Vy, p^{\eta}y),
\end{equation*}
where $V\coloneqq p^{\eta}\Ide- p^{\nu}\Ide +\Ide-A$ and $P$ and $Q$ are the polynomials introduced in~\eqref{los polinomios1}. The map
\begin{equation*}
F\colon \ker(F_1)\cap \ker(F_2) \to \wh{C}^{02}_N(H,I)\oplus \wh{C}^{11}_N(H,I),
\end{equation*}
given by $F(f_0,\gamma)\coloneqq \bigl(\alpha_1(\gamma),f_{f_0,\gamma}\bigr)$, where $\alpha_1(\gamma)$ is as in equality~\eqref{sime} and $f_{f_0,\gamma}$ is as in Definition~\ref{definicion de breve f}, is a linear map. In fact, this follows immediately from the definitions of $\alpha_1(\gamma)$ and $f_{f_0,\gamma}$, since, by Proposition~\ref{formula mas simple de tilde f}, the map~$\tilde{f}$ is linear in $f_0$ and $\gamma$.

\begin{corollary}\label{calculo de la homo} The map $F$~in\-duces a linear iso\-morphism
\begin{equation*}
\frac{\ker(F_1)\cap \ker(F_2)}{\ima(G)}\simeq \ho_{\blackdiamond}^2(H,I).
\end{equation*}
\end{corollary}

\begin{proof} This follows immediately from Remark~\ref{si cociclo da lo mismo} and Theorems~\ref{cociclos} and~\ref{cobordes}.
\end{proof}

For each $u\in \mathds{Z}$, we let $\ell(u)$ denote the unique representative of $l(u)\in \mathds{Z}_{p^{\eta}}$ such that $0\le\ell(u)<p^{\eta}$. By~\eqref{formula para l}, we have
\begin{equation}\label{calculo de ell(p^{eta-nu})}
\ell(p^{\eta-\nu})=\begin{cases} p^{\eta-\nu}&\text{if $p$ is odd or $\eta=\nu$,}\\ 2^{\eta-\nu} + 2^{\eta-1}& \text{if $p=2$ and $1<\nu<\eta$.}\end{cases}
\end{equation}
 Note that $p=2$ and $1=\nu<\eta$ gives the excluded case $(\nu,\eta)=(1,2)$.
\begin{lemma}\label{sacar c} For all $0\le c<p^{\eta}$, we have
\begin{equation*}
\sum_{j=0}^{\ell(p^{\eta-\nu})-1} A^{\ell(p^{\eta-\nu})-1-j}\tilde{f}(j\cdot c) = c\sum_{j=0}^{\ell(p^{\eta-\nu})-1} A^{\ell(p^{\eta-\nu})-1-j}\tilde{f}(j\cdot 1).
\end{equation*}
\end{lemma}

\begin{proof} By~\eqref{calculo de ell(p^{eta-nu})} we know that $p^{\eta-\nu}\mid \ell(p^{\eta-\nu})$. Hence, we can apply Proposition~\ref{previa a cociclos}, in order to obtain the result.
\end{proof}

\begin{example}\label{calculo con A=Ide} Let $F_1$, $F_2$ and $G$ be as above Corollary~\ref{calculo de la homo}, let $P$ and $Q$ be as in~\eqref{los polinomios1} and let $A=\Ide$.~A~direct~com\-putation shows that
\begin{equation*}
P(\Ide) = p^{\eta}\frac{p^{\eta+\nu}-p^{\nu}+2}{2}\Ide \quad\text{and}\quad Q(\Ide)= p^{\nu}\frac{p^{\eta+\nu}-p^{\nu}+2}{2}\Ide.
\end{equation*}
Hence,
\begin{equation*}
F_1(y_1,y_2)= p^{\eta}(y_1-y_2) + p^{\nu}y_2,\quad F_2(y_1,y_2) = \frac{p^{\eta+\nu}-p^{\nu}+2}{2}F_1(y_1,y_2)\quad\text{and} \quad G(y) = (p^{\eta}y-p^{\nu}y,p^{\eta}y),
\end{equation*}
and so, by Corollary~\ref{calculo de la homo}, we have $\ho_{\blackdiamond}^2(H,I)\simeq \frac{\ker(F_1)}{\ima (G)}$. Let $I_{p^{\nu}}\coloneqq \ker (p^{\nu}\Ide)$. Clearly, the map  $\varphi\colon \ker(F_1)\to I\oplus I_{p^{\nu}}$, given by
\begin{equation*}
\varphi(y_1,y_2)\coloneqq (y_1-y_2, y_2 +p^{\eta-\nu}(y_1-y_2))
\end{equation*}
is bijective and induces an iso\-morphism $\frac{\ker(F_1)}{\ima (G)}\simeq \frac{I}{p^{\nu}I}\oplus I_{p^{\nu}}$.
\end{example}

Note that when $p^{\eta}I=0$, there is an action of $\mathds{Z}_{p^{\eta}}=\langle \sigma \rangle$, on $I$, given by $\sigma y\coloneqq Ay+p^{\nu}y$. When $p^{\nu}I = 0$ this action becomes $\sigma y = Ay$. In fact to check this it suffices to note that $(A+p^{\nu}\Ide)^{p^{\eta}} y = A^{p^{\eta}} y = y$, which hold because $p^{\eta} y = 0$ and $A^{p^{\eta}}=\Ide$. We will use these facts in Example~\ref{ejemplo 1} and Proposition~\ref{ejemplo complicado}.

\begin{example}\label{ejemplo 1} Assume that $p^{\nu}I=0$. Then $F_1(y_1,y_2)=A y_2-y_2$, and so, $\ker(F_1) = I\oplus \ker(A-\Ide)$. Moreover, since $p^{\nu}I=0$, the formulas for $P(A)$ and $Q(A)$, introduced in~\eqref{los polinomios1}, become
\begin{equation*}
P(A)=\sum_{j=0}^{p^{\eta}-1}A^j \quad \text{and}\quad Q(A)=\sum_{j=0}^{p^{\eta}-1}jp^{2\nu-\eta}A^j.
\end{equation*}
Thus, if $(y_1,y_2)\in \ker (F_1)$, then
\begin{equation*}
Ay_2=y_2,\quad P(A)y_2=p^{\eta}y_2=0\quad \text{and}\quad Q(A)y_2= p^{2\nu-\eta}\binom{p^{\eta}}{2}y_2=\frac{p^{2\nu}(p^{\eta}-1)}{2}y_2=0.
\end{equation*}
Hence $F_2(y_1,y_2)=P(A)(y_1-y_2)+Q(A)y_2=P(A)y_1$, and so
\begin{equation*}
\ker(F_1)\cap \ker(F_2) = \ker(P(A))\oplus \ker(A-\Ide).
\end{equation*}
Since $G(y) = (y-Ay,0)$, from Corollary~\ref{calculo de la homo} if follows that
\begin{equation*}
\ho_{\blackdiamond}^2(H,I) \simeq \frac{\ker(P(A))}{(A-\Ide)I}\oplus \ker(A-\Ide)= \ho^1(\mathds{Z}_{p^{\eta}},I)\oplus \ho^0(\mathds{Z}_{p^{\eta}},I),
\end{equation*}
where $\ho^0(\mathds{Z}_{p^{\eta}},I)$ and $\ho^1(\mathds{Z}_{p^{\eta}},I)$ denote the first cohomology groups of $\mathds{Z}_{p^{\eta}}=\langle \sigma \rangle$, with~coeffi\-cients in $I$.
\end{example}

\begin{proposition}\label{ejemplo complicado} Assume that $p\mid (A-\Ide)^{p-1}$ and $p^{\eta}I=0$. Then
\begin{equation*}
\ho_{\blackdiamond}^2(H,I) \simeq \frac{I}{(A+p^{\nu}\Ide-\Ide)I}\oplus \bigl(\ker(A+p^{\nu}\Ide-\Ide)\cap \ker(p^{\nu}\Ide)\bigr) = \ho^1(\mathds{Z}_{p^{\eta}},I)\oplus \ho^0(\mathds{Z}_{p^{\eta}},I_{p^{\nu}}),
\end{equation*}
where $I_{p^{\nu}}\coloneqq \ker(p^{\nu}\Ide)$.
\end{proposition}

\begin{proof} We first prove that $P(A)=0$. Note that $A^{p^{\eta}}=\Ide$ and $p^{\eta}I=0$ imply that
\begin{equation*}
\sum_{j=0}^{p^{\eta}-1}(A+p^{\nu}\Ide)^j=\sum_{j=0}^{p^{\eta}-1}A^j+ \sum_{j=0}^{p^{\eta}-1}j p^{\nu}A^j=\sum_{j=0}^{p^{\eta}-1}(1+jp^{\nu})A^j=P(A),
\end{equation*}
and so it suffices to prove that $\sum_{j=0}^{p^{\eta}-1}(A+p^{\nu}\Ide)^j=0$. We set $\tilde A\coloneqq A-\Ide$ and compute
\begin{equation*}
\sum_{j=0}^{p^{\eta}-1}(A+p^{\nu}\Ide)^j=\sum_{j=0}^{p^{\eta}-1}(\Ide+\tilde A+p^{\nu}\Ide)^j =\sum_{j=0}^{p^{\eta}-1}\sum_{i=0}^{j}\binom ji (\tilde A+p^{\nu}\Ide)^i=\sum_{i=0}^{p^{\eta}-1}\binom{p^{\eta}}{i+1}(\tilde A+p^{\nu}\Ide)^i.
\end{equation*}
Hence, in order to verify that $P(A)=0$, it suffices to check that $p^{\eta}\mid \binom{p^{\eta}}{i}(\tilde A+p^{\nu}\Ide)^{i-1}$ for all $i=1,\dots,p^{\eta}$. Write $i=s p^t$ with $p\nmid s$ and $t\ge 0$. On one hand, it is well known that $p^{\eta-t}\mid \binom{p^{\eta}}{i}$. On the other hand, since
\begin{equation*}
(p-1)t\le p^t-1\le sp^t-1=i-1,
\end{equation*}
and, by hypothesis $p\mid (\tilde A+p^{\nu}\Ide)^{p-1}$, we have $p^t\mid (\tilde A+p^{\nu}\Ide)^{i-1}$. Thus $P(A)=0$, as desired. It follows that
\begin{equation*}
F_1(y_1,y_2)=(A+p^{\nu}\Ide-\Ide)y_2 \quad\text{and}\quad F_2(y_1,y_2)=P(A)(y_1-y_2)+Q(A)y_2 = Q(A)y_2,
\end{equation*}
and consequently,
\begin{equation}\label{ejemplocohomologia}
\ker(F_1)\cap \ker(F_2) = I\oplus \bigl(\ker(A+p^{\nu}\Ide-\Ide)\cap \ker(Q(A))\bigr).
\end{equation}
Assume that $(y_1,y_2)\in \ker(F_1)$, which implies that $Ay_2=(1-p^{\nu})y_2$. We are going to prove that
\begin{equation}\label{implicacion}
Q(A)y_2 =0\quad \text{if and only if}\quad p^{\nu}y_2=0.
\end{equation}
For this note that $A^j y_2=(1-jp^{\nu})y_2$, for all $j\ge 0$. Hence,
\begin{equation*}
Q(A)y_2=\sum_{j=0}^{p^{\eta}-1}jp^{2\nu -\eta}(1-jp^{\nu})y_2+p^{\nu}y_2= \left(p^{2\nu -\eta}\binom{p^{\eta}}{2}-p^{3\nu -\eta}\left(2\binom{p^{\eta}}{3}+\binom{p^{\eta}}{2}\right)+p^{\nu}\right)y_2.
\end{equation*}
Since $p^{\eta}$ divides
\begin{equation*}
2p^{3\nu -\eta}\binom{p^{\eta}}{3}=p^{\eta}\frac{p^{3\nu-\eta}(p^{\eta}-1)(p^{\eta}-2)}{3} \qquad\text{and}\qquad p^{3\nu -\eta} \binom{p^{\eta}}{2}=p^{2\nu}\frac{p^{\nu}(p^{\eta}-1)}{2},
\end{equation*}
we arrive at
\begin{equation*}
Q(A)y_2 =\begin{cases}  p^{\nu}y_2 & \text{if $p$ is odd or $\eta< 2\nu$}, \\ (2^{\nu} - 2^{2\nu-1})y_2 & \text{if $p=2$ and $\eta = 2\nu>2$.}\end{cases}
\end{equation*}
By this formula, if $p$ is odd or $\eta<2\nu$, then $Q(A)y_2 =0$ if and only if $p^{\nu}y_2=0$. This is also true if $p=2$ and $\eta = 2\nu$, since $2^{\nu}-2^{2\nu-1} = 2^{\nu}(1-2^{\nu-1})$ and $1-2^{\nu-1}$ is odd, because $\nu>1$. This proves~\eqref{implicacion}.   Finally, by Corollary~\ref{calculo de la homo},
\begin{equation*}
\ho_{\blackdiamond}^2(H,I) \simeq \frac{I}{(A+p^{\nu}\Ide-\Ide)I}\oplus \bigl(\ker(A+p^{\nu}\Ide-\Ide)\cap \ker(p^{\nu}\Ide)\bigr) = \ho^1(\mathds{Z}_{p^{\eta}},I)\oplus \ho^0(\mathds{Z}_{p^{\eta}},I_{p^{\nu}}),
\end{equation*}
as we want.
\end{proof}

Note that the conditions in Proposition~\ref{ejemplo complicado} are fulfilled when $I=\mathds{Z}_{p^t}$ is a cyclic $p$-group with $t\le \eta$, or, in a more general setting, if $I=(\mathds{Z}_{p^t})^k$ for some $t\le \eta$ and $k\le p-1$. In fact in this case $A\in \End(I)=M_k(\mathds{Z}_{p^t})$, and can be mapped canonically into $M_k(\mathds{Z}_{p})$. The image of $A-\Ide$ is nilpotent with nilpotency index lower than or equal to $k\le p-1$, and so $p|(A-\Ide)^{p-1}$, as required.

\begin{example} Assume $H=\mathds{Z}_{p^{\eta}}$ is a trivial linear cycle set,  thus $\eta=\nu$. Assume $I=\mathds{Z}_p\oplus \mathds{Z}_p$. Then we are in the situation of Example~\ref{ejemplo 1} and of Proposition~\ref{ejemplo complicado}. Moreover, $A$ is equivalent to a matrix of the form $\left(\begin{smallmatrix} 1 & a \\ 0 & 1\end{smallmatrix}\right)$. If $a\ne 0$, then we obtain
\begin{equation*}
\ho_{\blackdiamond}^2(H,I) \simeq \frac{I}{(A-\Ide)I}\oplus \bigl(\ker(A-\Ide) = (0\oplus \mathds{Z}_p)\oplus (\mathds{Z}_p \oplus 0).
\end{equation*}
Hence, for $f_0=(0,z_1)$ and $\gamma=(z_2,0)$, we obtain
\begin{equation*}
h\blackdiamond y=A^hy=\begin{pmatrix} 1 & ha \\0 & 1\end{pmatrix}y\quad\text{and}
\quad  f_{f_0,\gamma}(1^{\times h},c)=c\sum_{j=0}^{h-1}A^j f_0=cz_1\binom{\binom h2 a}{h}.
\end{equation*}
\end{example}

\subsection[Cocycles in the general case]{Cocycles in the general case}

In this subsection we compute the $2$-cocycles $(\alpha_1(\gamma),f) \in \wh{C}^{02}_N(H,I)\oplus \wh{C}^{11}_N(H,I)$ without assuming that $\Yleft=0$. Thus the equalities~\eqref{cociclo doble} and~\eqref{cociclo horizontal con D} are satisfied. Recall from the beginning of Subsection~\ref{subsection 2.1}, that, for each $h\in \mathds{Z}_{p^{\eta}}$, there exists a unique $l(h)\in \mathds{Z}_{p^{\eta}}$, such that $1^{\times l(h)}=h$.

\begin{remark}\label{con B simplificacion} Let $(\alpha_1(\gamma),f)\in \wh{C}^{02}_N(H,I)\oplus \wh{C}^{11}_N(H,I)$ be a $2$-cocycle. By Proposition~\ref{necesarias}, we know that
\begin{equation*}
A^{j+1}\gamma=A^{j}\gamma +p^{\eta}A^j \gamma-p^{\nu}A^j \gamma - p^{\eta} A^j f(1,1),\qquad\text{ for all $j\in \mathds{N}_0$.}
\end{equation*}
An induc\-tive argument using this, and that  $2\nu\ge \eta$ and $p^{\eta}B = 0$, proves that
\begin{equation}\label{BAjgamma}
BA^j\gamma = (1-jp^{\nu})B\gamma\qquad\text{ for all $j\in \mathds{N}_0$.}
\end{equation}
Since $A^{p^{\eta}} = \Ide$ and $p^{\eta}B = 0$, this formula is true for all~$j\in \mathds{Z}$. Note also that, by~\eqref{BAjgamma},
\begin{equation}\label{BAjAgamma}
BA^jA\gamma = (1-(j+1)p^{\nu})B\gamma = (1-jp^{\nu})(1-p^{\nu})B\gamma = (1-jp^{\nu}) BA\gamma \quad\text{for all $j\in \mathds{N}_0$.}
\end{equation}
\end{remark}

\begin{remark}\label{b_k=1 si y...} Here and subsequently, for $k\in \mathds{Z}$, we set $b_k\coloneqq k + \binom{k+1}{2}p^{\nu}\in \mathds{Z}_{p^{\eta}}$. Note that $1^{\times(k+1)} = 1+b_k$ (see formula~\eqref{i times j}). So, $b_k=-1$ if and only if $k\equiv -1\pmod{p^{\eta}}$ and $k\equiv k'\pmod{p^{\eta}}$, if and only if $b_k = b_{k'}$.
\end{remark}

\begin{proposition}\label{necesaria para Yleft} Assume $(\alpha_1(\gamma),f)\in \wh{C}^{02}_N(H,I)\oplus \wh{C}^{11}_N(H,I)$ is a $2$-cocycle. Then, for $0\le k < p^{\eta}$, we have
\begin{equation}\label{recursiva general con B}
f(1^{\times k},c)=\sum_{j=0}^{k-1}A^{k-1-j}f(1,j\cdot c)+(k\cdot c)B\sum_{l=0}^{k-1} A^lf(1,b_l).
\end{equation}
Moreover,
\begin{equation}\label{condicion para f con yleft}
0=\sum_{j=0}^{p^{\eta}-1}A^{p^{\eta}-1-j}f(1,j\cdot c)+c B\sum_{l=0}^{p^{\eta}-1} A^l f(1,b_l)+c B \gamma.
\end{equation}
Note that, by Proposition~\ref{necesarias}, the map $f$ is uniquely determined~by~$f(1,1)$ and $\gamma$.
\end{proposition}

\begin{proof} For $k=0$, the equality~\eqref{recursiva general con B} holds since $f(0,c) = 0$. Assume it is true for some $0\le k<p^{\eta}$. Since,~in~$\mathds{Z}_{p^{\eta}}$,
\begin{equation*}
b_k = k + \binom{k+1}{2}p^{\nu} = k+ \binom{k}{2}p^{\nu}+kp^{\nu}\qquad\text{and}\qquad b_kp^{\nu} = kp^{\nu},
\end{equation*}
by~\eqref{i times j}, the fact that $1+b_k=1^{\times(k+1)}$, and~\eqref{1^times j.c}, we have,
\begin{equation}\label{seusadespues}
1\cdot b_k = (1-p^{\nu})b_k = k + \binom{k}{2}p^{\nu} = 1^{\times k}\quad\text{and}\quad (1+b_k)\cdot c = (1+k)\cdot c.
\end{equation}
Consequently, we have $l(1\cdot b_k) = k$ in $\mathds{Z}_{p^{\eta}}$. So, applying condition~\eqref{cociclo horizontal con D} with $(a,b,c)=(1,b_k,c)$, and using the fact that $1^{\times(k+1)} = 1+b$ and $l(1+b) = 1+k$ in $\mathds{Z}_{p^{\eta}}$, we have
\begin{align*}
f(1^{\times (k+1)},c) & = f(1+b,c)\\
&= A^{l(1\cdot b)} f(1,c)+ f\bigl(1\cdot b,1\cdot c\bigr) +((1+b)\cdot c)B A^{l(1\cdot b)}f(1,b)+((1+b)\cdot c)B A^{l(1+b)}\alpha_1(\gamma)(1,b)\\
& = A^k f(1,c)+ f\bigl(1^{\times k},1\cdot c\bigr)+((1+k)\cdot c)B A^kf(1,b)+((1+k)\cdot c)B A^{k+1}\alpha_1(\gamma)(1,b)\\
& = \sum_{j=0}^k A^{k-j}f(1,j\cdot c) + ((1+k)\cdot c) B\sum_{l=0}^k  A^lf(1,b_l) + ((1+k)\cdot c)B A^{k+1}\alpha_1(\gamma)(1,b),
\end{align*}
where in the last equality we have used the inductive hypothesis, equality~\eqref{(j+1).c} and that
\begin{equation*}
A^k f(1,c) + \sum_{j=0}^{k-1} A^{k-1-j}f(1,j\cdot (1\cdot c)) =  A^k f(1,c)+\sum_{j=0}^{k-1}A^{k-1-j}f(1,(j+1)\cdot c) = \sum_{j=0}^{k}A^{k-j}f(1,j\cdot c).
\end{equation*}
In order to finish the proof it suffices to note that $A^{p^{\eta}}=\Ide$, $p^{\eta}\cdot c=c$, and that, by Proposition~\ref{calculos de beta 1} and Remark~\ref{b_k=1 si y...}, for $k<p^{\eta}-1$, we have $\alpha_1(\gamma)(1,b_k)=0$, while, for $k=p^{\eta}-1$, we have $\alpha_1(\gamma)(1,b_k)= \alpha_1(\gamma)(1,p^{\eta}-1)= \gamma$.
\end{proof}

Let $\gamma,f_0\in I$. As in Subsection~\ref{subsection caso < = 0}, our next goal is to construct a $2$-cocycle $(\alpha_1(\gamma),f)\in \wh{C}^{02}_N(H,I) \oplus \wh{C}^{11}_N(H,I)$ with $f(1,1)\coloneqq f_0$, whenever it is possible. In the rest of this subsection, we assume that the condition in item~(3)~of Proposition~\ref{periodicoo} is satisfied.

\begin{definition}\label{definicion de breve f con B} Let $\tilde{f}$ and $b_l$ be as in Proposition~\ref{prop para definir f} and Remark~\ref{b_k=1 si y...}, respectively. Motivated by formula~\eqref{recursiva general con B}, we define $f\colon H\times H\to I$ by
\begin{equation}\label{recursiva general con B definicion}
f(1^{\times k},c)\coloneqq \sum_{j=0}^{k-1}A^{k-1-j}\tilde{f}(j\cdot c) + (k\cdot c)B\sum_{l=0}^{k-1} A^l\tilde{f}(b_l)\qquad\text{for $0\le k<p^{\eta}$ and $c\in H$.}
\end{equation}
Note that, by  equality~\eqref{hat{f} en forma recursive}, the map $f$ is determined by $\gamma$ and $f_0$. If convenient, we will~write $f_{f_0,\gamma}$ instead of $f$.
\end{definition}

We define $R(X)\in \mathds{Z}[X]$ and $S\in \mathds{Z}_{p^{\eta}}$, by
\begin{equation}\label{mas polinomios}
R(X)\coloneqq \sum_{j=1}^{p^{\eta}-1}jX^{j-\binom{j+1}{2}p^{\nu}}\qquad\text{and}\qquad S\coloneqq\begin{cases} -\frac{p^{\eta}+p^{\nu}}{2} & \text{if $p>3$ or $\eta <2 \nu$,}\\-\frac{p^{\eta}+p^{\nu}}{2}+\binom{p^{\nu}+1}{3}p^{\nu} & \text{if $2\nu = \eta$ and $p\in\{2,3\}$.}\end{cases}
\end{equation}
For computations it is useful to note that, if $p$ is odd and $p^{\nu}I = 0$, then
\begin{equation}\label{R simple}
R(A)=\sum_{j=1}^{p^{\eta}-1}j A^{j-\binom{j+1}{2}p^{\nu}}=\sum_{t=1}^{p^{\eta}-1}\left(t+\binom{t+1}{2}p^{\nu}\right)A^t =\sum_{t=1}^{p^{\eta}-1} tA^t,
\end{equation}
where the second equality holds by the mutually inverse bijections $j\mapsto j-\binom{j+1}{2}p^{\nu}$ and $t\mapsto t+\binom{t+1}{2}p^{\nu}$ of $\mathds{Z}_{p^{\eta}}\setminus \{0\}$.

\begin{remark} A direct computation shows that
\begin{equation}
S = \begin{cases} -2^{\eta-1}-2^{\nu-1}+\frac{2^{2\nu}-1}{3} 2^{2\nu-1} = -2^{\nu-1} & \text{if $p=2$ and $2\nu =\eta$,}\\ -\frac{3^{2\nu}+3^{\nu}}{2} + \frac{3^{2\nu}-1}{2} 3^{2\nu-1} = -\frac{3^{2\nu}+3^{\nu}}{2} + 3^{2\nu-1} & \text{if $p=3$ and $2\nu =\eta$.}\end{cases}
\end{equation}
\end{remark}

\begin{proposition}\label{periodico en c igual a 1} For $f_0,\gamma\in I$, we have:
\begin{equation}\label{R y S}
B\sum_{l=0}^{p^{\eta}-1} A^l \tilde{f}(b_l) = BR(A)f_0+SB\gamma,
\end{equation}
where $\tilde{f}$ is as in Proposition~\ref{prop para definir f} and the $b_l$'s are as in Remark~\ref{b_k=1 si y...}.
\end{proposition}

\begin{proof} Since $l\mapsto b_l$ is a bijection with inverse $k\mapsto k-\binom{k+1}{2}p^{\nu}$, we have
\begin{equation*}
B\sum_{l=0}^{p^{\eta}-1} A^l\tilde{f}(b_l) = B\sum_{k=0}^{p^{\eta}-1} A^{k-\binom{k+1}{2}p^{\nu}} \tilde{f}(k) = B\sum_{k=2}^{p^{\eta}-1} A^{k-\binom{k+1}{2}p^{\nu}}\sum_{b=1}^{k-1} \Gamma(\gamma,b) + BR(A)f_0,
\end{equation*}
where the last equality holds by~\eqref{hat{f} en forma recursive} and the fact that $\tilde{f}(k) = \hat{f}(k)$ and $\Gamma(\gamma,0) = 0$ (by Remark~\ref{primeros calculos de Gamma}). Let
\begin{equation*}
\wt{S}\coloneqq B\sum_{k=2}^{p^{\eta}-1} A^{k-\binom{k+1}{2} p^{\nu}}\sum_{b=1}^{k-1} \Gamma(\gamma,b)\in I.
\end{equation*}
In order to finish the proof we must show that $\wt{S}=SB\gamma$. For each $b\in H$, let $r(b)$ be as in Proposition~\ref{calculos de beta 1}. By that pro\-po\-sition,
\begin{equation}\label{Gamma(gamma,b)}
\Gamma(\gamma,b)=A\alpha_1(\gamma)(b,1)-\alpha_1(\gamma)\bigl(1\cdot b,1\cdot 1\bigr)= \begin{cases} A\gamma -\gamma & \text{if $b=p^{\eta}-1$,}\\ -\gamma & \text{if $r(b)>p^{\nu}-1$,}\\ 0 & \text{if $r(b)<p^{\nu}-1$.}\end{cases}
\end{equation}
Hence, by~\eqref{BAjgamma} and the fact that $p^{2\nu}B=0$, we have $BA^{k-\binom{k+1}{2} p^{\nu}}\Gamma(\gamma,b) =(1-kp^{\nu})B\Gamma(\gamma,b)$, and so,
\begin{equation*}
\wt{S}= B\sum_{k=2}^{p^{\eta}-1} \sum_{b=1}^{k-1} (1-kp^{\nu})\Gamma(\gamma,b)= B \sum_{b=1}^{p^{\eta}-2} \sum_{k=b+1}^{p^{\eta}-1} (1-kp^{\nu})\Gamma(\gamma,b) = B \sum_{b=1}^{p^{\eta}-2}\left(-1-b+\binom {b+1}2 p^{\nu}\right)\Gamma(\gamma,b),
\end{equation*}
where the last equality holds since
\begin{equation*}
\sum_{k=b+1}^{p^{\eta}-1} (1-kp^{\nu})\equiv -1-b+\binom {b+1}2 p^{\nu}\pmod{p^{\eta}}.
\end{equation*}
Note that $r(b)\equiv (1-p^{\nu})b\pmod{p^{\eta}}$ implies $b\equiv (1+p^{\nu})r(b)\pmod{p^{\eta}}$, and that the image by $r$, of $\{1\le b\le p^{\eta}-2\}$, is the set $\{1\le s < p^{\eta}-1\}\setminus\{p^{\nu}-1\}$. Using these facts, \eqref{BAjgamma} and~\eqref{Gamma(gamma,b)}, we obtain
\begin{align*}
\wt{S} & = B \sum_{s= p^{\nu}}^{p^{\eta}-1}\left(-1-(1+p^{\nu})s+\binom {(1+p^{\nu})s+1}2 p^{\nu}\right)(-\gamma)\\
& =  B \sum_{s= p^{\nu}}^{p^{\eta}-1}\left(1+(1+p^{\nu})s-\binom {(1+p^{\nu})s+1}2 p^{\nu}\right)\gamma\\
&= -p^{\nu}B\gamma+\frac{p^{\nu}-p^{\eta}}{2}B\gamma-\sum_{s= p^{\nu}}^{p^{\eta}-1}\binom {s+1+p^{\nu}s} 2 p^{\nu}B\gamma\\
&= -\frac{p^{\eta}+p^{\nu}}{2}B\gamma-\sum_{s= p^{\nu}}^{p^{\eta}-1}\binom {s+1} 2 p^{\nu}B\gamma -\sum_{s= p^{\nu}}^{p^{\eta}-1}\binom {p^{\nu}s} 2 p^{\nu}B\gamma.
\end{align*}
Note now that
$$
\sum_{s= p^{\nu}}^{p^{\eta}-1}\binom {s+1} 2 p^{\nu}B=\binom{p^{\eta}+1}{3}p^{\nu} B- \binom{p^{\nu}+1}{3}p^{\nu}B = - \binom{p^{\nu}+1}{3}p^{\nu}B
$$
and
$$
\sum_{s=p^{\nu}}^{p^{\eta}-1}\binom{p^{\nu}s} 2 p^{\nu}B = \sum_{s= p^{\nu}}^{p^{\eta}-1}\frac{sp^{2\nu}(sp^{\nu}-1)}{2}B = 0,
$$
where the last equality is trivial if $p$ is odd or $\eta<2\nu$, and hold if $p=2$ and $\eta=2\nu>2$, because
$$
\sum_{s=2^{\nu}}^{2^{\eta}-1}\frac{s2^{2\nu}(s2^{\nu}-1)}{2}B = -2^{2\nu-1} \sum_{s=2^{\nu}}^{2^{\eta}-1} s B = - 2^{\eta-1} \left(2^{\eta-1}(2^{\eta}-1)-2^{\nu-1}(2^{\nu}-1)\right)B = 0.
$$
Consequently,
$$
\wt{S} = -\frac{p^{\eta}+p^{\nu}}{2}B\gamma + \binom{p^{\nu}+1}{3}p^{\nu}B\gamma.
$$
This finishes the proof, since $\binom{p^{\nu}+1}{3}p^{\nu}B\gamma=0$, when $p>3$ or $\eta<2\nu$.
\end{proof}

\begin{corollary}\label{corolario de periodico en c igual a 1} For all $c\in \mathds{Z}$, we  have:
\begin{equation*}
\sum_{j=0}^{p^{\eta}-1}A^{p^{\eta}-1-j}\tilde{f}(j\cdot c)+ cB\sum_{l=0}^{p^{\eta}-1} A^l \tilde{f}(b_l) = c\bigl(A^{-1}P(A)(f_0-\gamma)+ A^{-1}Q(A)\gamma+ BR(A)f_0+SB\gamma\bigr).
\end{equation*}
\end{corollary}

\begin{proof} By Corollary~\ref{previa a cociclos1} and Proposition~\ref{periodico en c igual a 1}.
\end{proof}

The main results of this subsection are Theorems~\ref{cociclos con B} and~\ref{cobordes Bne0}, and Corollary~\ref{calculo de la homo B'}. We carry out some of the necessary lengthy technical computations in the following lemma.

\begin{lemma}\label{lema tecnico} Let $\gamma, f_0\in I$ and let $f$ as in Definition~\ref{definicion de breve f con B}. Assume that $(f_0,\gamma)$ satisfies the condition in item~(3) of Proposition~\ref{periodicoo}. Let $0\le l,l'<p^{\eta}$ and set $a\coloneqq 1^{\times l}$, $b\coloneqq 1^{\times l'}$, let $L(l,l')$ be as in Remark~\ref{def de L(l,l')} and let $M\in \mathds{Z}$ be the representative of $L(l,l')$ with $0\le M<p^{\eta}$. Then,
\begin{align}
& (M\cdot c)(l\cdot b)BA^{M-l}B+(l\cdot c)A^{M-l}B=(M\cdot c)BA^{M-l}&&\text{for all $c\in H$,}\label{igualdad para dos terminos}\\
& b_j+c_j = b_{M-l+j} &&\text{for $j\in \mathds{Z}$}\label{provisoria}
\end{align}
and
\begin{equation}\label{eq1 lema tecnico}
\begin{aligned}
B\sum_{j=0}^{l-1} A^{M-l+j+1}\alpha_1(\gamma)(1,b_j+ c_j) & = B\sum_{j=0}^{l-1} A^{M-l+j} \tilde{f}\left(b_j\right) - B\sum_{j=0}^{l-1} A^{M-l+j} \tilde{f}\left(b_j+c_j\right)\\
& + B \sum_{j=0}^{l-1}A^{M-1-j}\tilde{f}(j\cdot b)+BA^{M}\alpha_1(\gamma)(a,b),
\end{aligned}
\end{equation}
where $b_j$ is as in Remark~\ref{b_k=1 si y...} and  $c_j\coloneqq (l-1-j)\cdot b$.
\end{lemma}

\begin{proof} We first prove~\eqref{igualdad para dos terminos}. We claim that
\begin{equation}\label{f1}
B(A-AB)^l=(1+lp^{\nu})A^lB\qquad\text{for all $l\in \mathds{Z}_{p^{\eta}}$.}
\end{equation}
In fact, by~\eqref{igualdad exponencial general} we know that $B(A-AB)=(1+p^{\nu})AB$ and an inductive argument using this shows that~\eqref{f1} holds for $l\in \mathds{N}_0$. Since the right side of~\eqref{f1} is periodic of period $p^{\eta}$, the claim is true. Combining~\eqref{condicion compuesta} with~\eqref{f1} we obtain that
\begin{equation}\label{BAB simplificado}
\left(l+\binom l2 p^{\nu}\right)BA^lB=BA^l-(1+lp^{\nu})A^lB\qquad\text{for all $l\in \mathds{Z}_{p^{\eta}}$.}
\end{equation}
Now notice that by equalities~\eqref{i times j}, \eqref{1^times j.c} and~\eqref{a.b=},
\begin{equation*}
l\cdot b =l\cdot 1^{\times l'} = 1^{\times l}\cdot 1^{\times l'} = 1^{\times M-l}= M-l+\binom{M-l}{2}p^{\nu}.
\end{equation*}
Hence, by~\eqref{BAB simplificado} we have
\begin{equation*}
(M\cdot c)(l\cdot b)BA^{M-l}B= (M\cdot c)\left( BA^{M-l}-(1+(M-l)p^{\nu})A^{M-l}B\right).
\end{equation*}
Moreover, by~\eqref{(j+1).c},
\begin{equation*}
(M\cdot c)(1+(M-l)p^{\nu})= (l-M)\cdot (M\cdot c) = l\cdot c,
\end{equation*}
and so,
\begin{equation*}
(M\cdot c)(l\cdot b)BA^{M-l}B+(l\cdot c)A^{M-l}B=(M\cdot c)BA^{M-l},
\end{equation*}
which finishes the proof of~\eqref{igualdad para dos terminos}. We next prove~\eqref{provisoria}. A direct computation using~\eqref{i times j} and the comment above Remark~\ref{es inversa}, yields
\begin{align*}
b_j+c_j & = j + \binom{j+1}{2}p^{\nu} + (l-1-j)\cdot 1^{\times l'}\\
& = j + \binom{j+1}{2}p^{\nu} + \bigl(1-(l-1-j)p^{\nu}\bigr)\left(l'+\binom{l'}{2}p^{\nu}\right)\\
& = j + \binom{j+1}{2}p^{\nu} + l'+\binom{l'}{2}p^{\nu} - (l-1-j)l'p^{\nu}\\
& = j + l' - ll'p^{\nu} + \binom{j+l'+1}{2}p^{\nu}.
\end{align*}
When $p>2$ or $\eta<2\nu$, we have
\begin{equation*}
b_j+c_j = j + l' - ll'p^{\nu} + \binom{j+l'+1}{2}p^{\nu} = j + l' - ll'p^{\nu} + \binom{j+l'-ll'p^{\nu}+1}{2}p^{\nu} = b_{M-l+j},
\end{equation*}
where the last equality holds by Proposition~\ref{para condiciones sin potencia} and Remark~\ref{b_k=1 si y...}; while a similar computation, using Proposition~\ref{para condiciones sin potencia2} and Remark~\ref{b_k=1 si y...}, shows that equality~\eqref{provisoria} also holds when $p=2$ and $\eta=2\nu$. This finishes the proof of~\eqref{provisoria}. Finally, we prove~\eqref{eq1 lema tecnico}. Since
\begin{equation*}
B\sum_{j=0}^{l-1} A^{M-l+j} \tilde{f}(c_j) = B\sum_{j=0}^{l-1} A^{M-l+j} \tilde{f}((l-1-j)\cdot b) = B \sum_{j=0}^{l-1}A^{M-1-j}\tilde{f}(j\cdot b),
\end{equation*}
by condition~\eqref{cociclo doble para a uno f tilde}, we have
\begin{align*}
B\sum_{j=0}^{l-1} A^{M-l+j}\bigl(A\alpha_1(\gamma)(b_j,c_j) & -\alpha_1(\gamma)(1\cdot b_j,1\cdot c_j)\bigr) =  B\sum_{j=0}^{l-1} A^{M-l+j}\bigl(\tilde{f}(b_j+c_j)-\tilde{f}(b_j)-\tilde{f}(c_j)\bigr)\\
& = B\sum_{j=0}^{l-1} A^{M-l+j} \tilde{f}\left(b_j+c_j\right) - B\sum_{j=0}^{l-1} A^{M-l+j} \tilde{f}\left(b_j\right) - B\sum_{j=0}^{l-1}A^{M-1-j}\tilde{f}(j\cdot b).
\end{align*}
Thus, in order to check~\eqref{eq1 lema tecnico}, we must prove that
\begin{equation}\label{a probar1}
B\sum_{j=0}^{l-1} A^{M-l+j+1}\alpha_1(\gamma)(1,b_j+c_j) = B\sum_{j=0}^{l-1} A^{M-l+j}\bigl(\alpha_1(\gamma)(1\cdot b_j,1\cdot c_j) -A\alpha_1(\gamma)(b_j,c_j)\bigr)+ BA^{M}\alpha_1(\gamma)(a,b).
\end{equation}
Since $\alpha_1(\gamma)(1\cdot b_0,1\cdot c_0) = \alpha_1(\gamma)(1\cdot 0,1\cdot c_0) = 0$ and, by equalities~\eqref{(j+1).c} and~\eqref{seusadespues},
$$
1\cdot c_j=c_{j-1}\quad\text{and}\quad 1\cdot b_j= j + \binom{j}{2}p^{\nu}=1+b_{j-1},\qquad\text{for $0<j<p^{\eta}$,}
$$
we have
\begin{align*}
B\sum_{j=0}^{l-1} A^{M-l+j}\bigl(\alpha_1(\gamma)(1\cdot b_j,1\cdot c_j)-A\alpha_1(\gamma)(b_j,c_j)\bigr) & = B\sum_{j=0}^{l-2} A^{M-l+j+1}\alpha_1(\gamma)(1+b_j,c_j)\\
&-B\sum_{j=0}^{l-1} A^{M-l+j+1}\alpha_1(\gamma)(b_j,c_j).
\end{align*}
Moreover $BA^M\alpha_1(\gamma)(a,b)=BA^M\alpha_1(\gamma)(1+b_{l-1},c_{l-1})$, since $1+b_{l-1} = a$ and $c_{l-1}=b$. Hence, \eqref{a probar1} becomes
\begin{equation}\label{nuevo a probar1}
B\sum_{j=0}^{l-1} A^{M-l+j+1}\alpha_1(\gamma)(1,b_j+c_j) = B\sum_{j=0}^{l-1} A^{M-l+j+1}\bigl(\alpha_1(\gamma)(1+b_j,c_j)-\alpha_1(\gamma)(b_j,c_j)\bigr).
\end{equation}
Since
$$
\alpha_1(\gamma)(b_j,c_j)-\alpha_1(\gamma)(1+b_j,c_j)-\alpha_1(\gamma)(1,b_j)+\alpha_1(\gamma)(1,b_j+c_j) = \partial_{\mathrm{v}}^{03}(\alpha_1(\gamma)) (1,b_j,c_j)=0,
$$
in order to check~\eqref{nuevo a probar1}, we are reduced to prove that
\begin{equation*}
B\sum_{j=0}^{l-1} A^{M-l+j+1}\alpha_1(\gamma)(1,b_j)= 0.
\end{equation*}
But this is true, since $l<p^{\eta}$ and $\alpha_1(\gamma)(1,b_j)= 0$, when $j<p^{\eta}-1$ (see Proposition~\ref{calculos de beta 1} and Remark~\ref{b_k=1 si y...}).
\end{proof}

\begin{remark}\label{si cociclo da lo mismo1} Suppose $(\alpha_1(\gamma),f) \in \wh{C}^{02}_N(H,I)\oplus \wh{C}^{11}_N(H,I)$ is a $2$-cocycle and set $f_0\coloneqq f(1,1)$. Then,~by Propo\-sition~\ref{necesaria para Yleft}, the map $f_{f_0,\gamma}$ in Definition~\ref{definicion de breve f con B}, coincides with $f$.
\end{remark}

\begin{theorem}\label{cociclos con B} Let $P$, $Q$, $R$ and $S$ be as in~\eqref{los polinomios1} and~\eqref{mas polinomios}. Each $2$-cocycle \hbox{$(\beta,g)\in \wh{C}^{02}_N(H,I)\oplus \wh{C}^{11}_N(H,I)$} is equiva\-lent modulo coboundaries to a $2$-cocycle of the form $(\alpha_1(\gamma), f)$. Moreover, if we set $f_0\coloneqq f(1,1)$, then
\begin{equation}\label{necesarias y suficientes con B}
A^{-1}P(A)(f_0-\gamma)+ A^{-1}Q(A)\gamma+ BR(A)f_0+(S+1)B\gamma=0\quad\text{and}\quad p^{\eta}f_0 = (\Ide-A-p^{\nu}\Ide+p^{\eta}\Ide)\gamma
\end{equation}	
Conversely, if $\gamma, f_0\in I$ satisfy \eqref{necesarias y suficientes con B}, then the pair $\bigl(\alpha_1(\gamma),f_{f_0,\gamma}\bigr)$ is a $2$-cocycle (see Definition~\ref{definicion de breve f con B}).
\end{theorem}

\begin{proof} Let $(\beta,g)\in \wh{C}^{02}_N(H,I)\oplus \wh{C}^{11}_N(H,I)$ be a $2$-cocycle. By Corollary~\ref{basta tomar b_1(gamma)}, there exist $\gamma$ and $f$ such that  $(\alpha_1(\gamma), f)$ is a $2$-cocycle equivalent to $(\beta, g)$. By~Proposi\-tion~\ref{necesarias}, the second condition in~\eqref{necesarias y suficientes con B} is satisfied. We next prove the first one. In fact, we have
$$
A^{-1}P(A)(f_0-\gamma)+ A^{-1}Q(A)\gamma+ BR(A)f_0+(S+1)B\gamma =\sum_{j=0}^{p^{\eta}-1}A^{p^{\eta}-1-j}f(1,j\cdot 1)+B\sum_{l=0}^{p^{\eta}-1} A^l f(1,b_l)+B\gamma = 0,
$$
where the first equality holds by Remark~\ref{hat f para cociclo} and Corollary~\ref{corolario de periodico en c igual a 1}, and the last one, by equality~\eqref{condicion para f con yleft}. Assume~con\-verse\-ly that $f_0$ and $\gamma$ satisfy~\eqref{necesarias y suficientes con B} and let $f=f_{f_0,\gamma}$ be as in~\eqref{recursiva general con B definicion}. By definition,
\begin{equation*}
f(0,c) = f(1^{\times 0},c) = 0\quad\text{for all $c\in H$.}
\end{equation*}
Moreover, since $j\cdot 0 = 0$, for all $j\in H$, we have $\tilde{f}(j\cdot 0) = \tilde{f}(0) = \hat{f}(0) = 0$, for all $j\in H$. Hence
\begin{equation*}
f(1^{\times k},0) = 0\quad\text{for all $0\le k < p^{\eta}$.}
\end{equation*}
Thus, the map $f$ is normal. In order to finish the proof we must show that conditions~\eqref{cociclo doble} and~\eqref{cociclo horizontal con D} are satisfied. Let $b,c\in H$ and $0\le k<p^{\eta}$. By the first equality in~\eqref{compatibilidades cdot suma y suma cdot},
$$
(k\cdot (b+c)) B\sum_{l=0}^{k-1} A^l\tilde{f}(b_l) = (k\cdot b)B\sum_{l=0}^{k-1} A^l\tilde{f}(b_l) + (k\cdot c)B\sum_{l=0}^{k-1} A^l\tilde{f}(b_l).
$$
Using this and Proposition~\ref{coro de prop para definir f}, we obtain that
\begin{equation*}
f(1^{\times k},b+c)-f(1^{\times k},b)-f(1^{\times k},c)=A^k \alpha_1(\gamma)(b,c)-\alpha_1(\gamma)(1^{\times k}\cdot b,1^{\times k}\cdot c),
\end{equation*}
which proves condition~\eqref{cociclo doble}. It remains to prove condition~\eqref{cociclo horizontal con D}. In order to do this we take $a\coloneqq 1^{\times l}$ and $b\coloneqq 1^{\times l'}$ with $0\le l,l'<p^{\eta}$, and we take $M$ as in the previous lemma. By Remark~\ref{def de L(l,l')}, we have
\begin{equation}\label{ecua 1}
a+b = 1^{\times L(l,l')} = 1^{\times M} \qquad\text{and}\qquad a\cdot b = a^{\times -1}\times (a + b)= 1^{\times -l}\times 1^{\times M} = 1^{\times M-l}.
\end{equation}
From this and equality~\eqref{1^times j.c}, we conclude that
\begin{equation}\label{ecua 2}
A^{l(a\cdot b)} = A^{M-l},\quad A^{l(a+b)} = A^M,\quad a\cdot c = 1^{\times l}\cdot c = l\cdot c\quad\text{and}\quad (a+b)\cdot c = 1^{\times M}\cdot c = M\cdot c.
\end{equation}
Using equalities~\eqref{ecua 1} and~\eqref{ecua 2}, we obtain that~\eqref{cociclo horizontal con D} is equivalent to
\begin{equation}\label{cociclo horizontal con D simple}
-f(1^{\times M},c) + A^{M-l}f(1^{\times l},c) + f(1^{\times M-l},l\cdot c) + (M\cdot c)BA^{M-l}f(1^{\times l},b) + (M\cdot c)BA^{M}\alpha_1(\gamma)(a,b) = 0.
\end{equation}

\smallskip

\noindent {\bf Case \pmb{$0\le l\le M < p^{\eta}$}.}\enspace We expand the first four terms in~\eqref{cociclo horizontal con D simple} using the definition of $f$:
\begin{align}
&-f(1^{\times M},c) = -\sum_{j=0}^{M-1}A^{M-1-j}\tilde{f}(j\cdot c) - (M\cdot c)B \sum_{j=0}^{M-1} A^j \tilde{f}\left(b_j\right),\label{primer termino}\\
& A^{M-l}f(1^{\times l},c) = \sum_{j=0}^{l-1}A^{M-1-j}\tilde{f}(j\cdot c) + (l\cdot c)A^{M-l}B\sum_{j=0}^{l-1} A^j \tilde{f}\left(b_j\right),\label{segundo termino}\\
&f(1^{\times M-l},l\cdot c) = \sum_{j=l}^{M-1}A^{M-1-j}\tilde{f}(j\cdot c) + (M\cdot c)B \sum_{j=0}^{M-l-1} A^j \tilde{f}\left(b_j\right),\label{tercer termino}\\
& (M\cdot c) BA^{M-l} f(1^{\times l},b) = (M\cdot c) B \sum_{j=0}^{l-1}A^{M-1-j}\tilde{f}(j\cdot b) + (M\cdot c)(l\cdot b) BA^{M-l}B \sum_{j=0}^{l-1} A^j\tilde{f}\left(b_j\right),\label{cuarto termino}
\end{align}
where in the computing of $f(1^{\times M-l},l\cdot c)$ we have used that $j\cdot (l\cdot c) = (j+l)\cdot c$ and $(M-l)\cdot (l\cdot c) = M\cdot c$ (see~\eqref{(j+1).c}). In identity~\eqref{cociclo horizontal con D simple}, the first terms at the right side of the equal sign in~\eqref{primer termino}, \eqref{segundo termino} and~\eqref{tercer termino} cancel out. Moreover, it is clear that the sum of the second term of~\eqref{primer termino} and the second term of~\eqref{tercer termino} gives
\begin{equation*}
-(M\cdot c)B\displaystyle\sum_{j=0}^{M-1} A^j \tilde{f}\left(b_j\right)+(M\cdot c)B\displaystyle\sum_{j=0}^{M-l-1} A^j \tilde{f}\left(b_j\right) = -(M\cdot c) B\sum_{j=M-l}^{M-1} A^j \tilde{f}\left(b_j\right)
\end{equation*}
and, by equality~\eqref{igualdad para dos terminos}, the sum of the second~term in~\eqref{segundo termino} and the second term in~\eqref{cuarto termino} gives
\begin{equation*}
(l\cdot c)A^{M-l}B\sum_{j=0}^{l-1} A^j \tilde{f}\left(b_j\right)+(M\cdot c)(l\cdot b)BA^{M-l}B\sum_{j=0}^{l-1} A^j \tilde{f}\left(b_j\right)=(M\cdot c)B\sum_{j=0}^{l-1} A^{M-l+j} \tilde{f}\left(b_j\right).
\end{equation*}
Hence~\eqref{cociclo horizontal con D simple} is equivalent to
\begin{align*}
0 & = - (M\cdot c)B \sum_{j=M-l}^{M-1} A^j \tilde{f}\left(b_j\right)
+ (M\cdot c)B \sum_{j=0}^{l-1} A^{M-l+j} \tilde{f}\left(b_j\right)\\ &
+(M\cdot c) B \sum_{j=0}^{l-1}A^{M-1-j}\tilde{f}(j\cdot b)+(M\cdot c) BA^M\alpha_1(\gamma)(a,b).
\end{align*}
Since one can choose $c$ such that $M\cdot c=1$, this equality is equivalent to
\begin{equation}\label{equivalente a 3.5}
B \sum_{j=M-l}^{M-1} A^j \tilde  f\left(b_j\right) - B\sum_{j=0}^{l-1} A^{M-l+j} \tilde{f}\left(b_j\right)-B \sum_{j=0}^{l-1}A^{M-1-j}\tilde{f}(j\cdot b)  =  BA^M\alpha_1(\gamma)(a,b).
\end{equation}
On the other hand, by~\eqref{provisoria},
\begin{equation*}
B\sum_{j=M-l}^{M-1} A^j \tilde{f}\left(b_j\right) = B\sum_{j=0}^{l-1} A^{M-l+j} \tilde{f}\left(b_{M-l+j}\right) = B\sum_{j=0}^{l-1} A^{M-l+j} \tilde{f}\left(b_j+c_j\right).
\end{equation*}
From this and equality~\eqref{eq1 lema tecnico}, it follows that~\eqref{equivalente a 3.5} is equivalent to
\begin{equation}\label{ultima equivalencia}
B\sum_{j=0}^{l-1} A^{M-l+j+1}\alpha_1(\gamma)(1,b_j+c_j) = 0.
\end{equation}
But this is true, since, by equality~\eqref{provisoria},
\begin{equation*}
B\sum_{j=0}^{l-1} A^{M-l+j+1}\alpha_1(\gamma)(1,b_j+c_j)= B\sum_{j=0}^{l-1} A^{M-l+j+1}\alpha_1(\gamma)(1,b_{M-l+j})= B\sum_{j=M-l}^{M-1} A^{j+1}\alpha_1(\gamma)\left(1,b_j\right),
\end{equation*}
which is $0$, since $M<p^{\eta}$ and $\alpha_1(\gamma)(1,b_j)= 0$, when $j<p^{\eta}-1$ (see Proposition~\ref{calculos de beta 1} and Remark~\ref{b_k=1 si y...}).

\smallskip

\noindent {\bf Case \pmb{$0\le M<l<p^{\eta}$}.}\enspace We expand the first four terms in~\eqref{cociclo horizontal con D simple} using the definition of $f$:
\begin{align}
& -f(1^{\times M},c) = -\sum_{j=0}^{M-1}A^{M-1-j}\tilde{f}(j\cdot c) - (M\cdot c)B \displaystyle\sum_{j=0}^{M-1} A^j \tilde{f}\left(b_j\right) ,\label{primer termino l mayor}\\
& A^{M-l}f(1^{\times l},c)= \sum_{j=0}^{l-1}A^{M-1-j}\tilde{f}(j\cdot c) + (l\cdot c)A^{M-l}B \sum_{j=0}^{l-1} A^j \tilde{f}\left(b_j\right),\label{segundo termino l mayor}\\
& f(1^{\times M-l},l\cdot c) = \sum_{j=l}^{p^{\eta}+M-1}A^{M-1-j}\tilde{f}(j\cdot c) + (M\cdot c)B \sum_{j=0}^{p^{\eta}+M-l-1} A^j \tilde{f}\left(b_j\right),\label{tercer termino l mayor}\\
& (M\cdot c) BA^{M-l} f(1^{\times l},b) = (M\cdot c) B \sum_{j=0}^{l-1}A^{M-1-j}\tilde{f}(j\cdot b) + (M\cdot c)(l\cdot b)BA^{M-l}B \sum_{j=0}^{l-1} A^j \tilde{f}\left(b_j\right),\label{cuarto termino l mayor}
\end{align}
where in the computing of $f(1^{\times M-l},l\cdot c)$ we have used that $j\cdot (l\cdot c) = (j+l)\cdot c$ and $(M-l)\cdot (l\cdot c) = M\cdot c$ (see~\eqref{(j+1).c}). Clearly the sum of the first terms at the right side of the equal sign in~\eqref{primer termino l mayor}, \eqref{segundo termino l mayor} and~\eqref{tercer termino l mayor}, yields
\begin{equation*}
\sum_{j=M}^{p^{\eta}+M-1}A^{M-1-j}\tilde{f}(j\cdot c) = \sum_{j=0}^{p^{\eta}-1}A^{M-1-j}\tilde{f}(j\cdot c)=(M\cdot c) \sum_{j=0}^{p^{\eta}-1} A^{p^{\eta}-1-j}\tilde{f}(j\cdot 1),
\end{equation*}
where the last equality follows from Corollary~\ref{anularse implica invarianza}. Moreover, it is clear that the sum of the second term of~\eqref{primer termino l mayor} and the second term of~\eqref{tercer termino l mayor} gives
\begin{equation*}
(M\cdot c)B\sum_{j=0}^{p^{\eta}+M-l-1} A^j \tilde{f}\left(b_j\right)-(M\cdot c)B\sum_{j=0}^{M-1} A^j \tilde{f}\left(b_j\right) = (M\cdot c) B\sum_{j=0}^{p^{\eta}-1} A^j \tilde{f}\left(b_j\right)-(M\cdot c)B\sum_{j=p^{\eta}+M-l}^{p^{\eta}+M-1} A^j \tilde{f}\left(b_j\right),
\end{equation*}
and, as before, by~\eqref{igualdad para dos terminos}, the sum of the second term in~\eqref{segundo termino l mayor} and~the second term in~\eqref{cuarto termino l mayor} gives
\begin{equation*}
(l\cdot c)A^{M-l}B\sum_{j=0}^{l-1} A^j \tilde{f}\left(b_j\right)+(M\cdot c)(l\cdot b)BA^{M-l}B\sum_{j=0}^{l-1} A^j \tilde{f}\left(b_j\right)=(M\cdot c) B\sum_{j=0}^{l-1} A^{M-l+j} \tilde{f}\left(b_j\right).
\end{equation*}
Hence,~\eqref{cociclo horizontal con D simple} is equivalent to
\begin{align*}
0 & = (M\cdot c)\sum_{j=0}^{p^{\eta}-1}A^{p^{\eta}-1-j}\tilde{f}(j\cdot 1) + (M\cdot c)B \sum_{j=0}^{p^{\eta}-1} A^j \tilde{f}\left(b_j\right) - (M\cdot c)B \sum_{j=p^{\eta}+M-l}^{p^{\eta}+M-1} A^j \tilde{f}\left(b_j\right)\\
&  + (M\cdot c)B \sum_{j=0}^{l-1} A^{M-l+j} \tilde{f}\left(b_j\right) + (M\cdot c) B \sum_{j=0}^{l-1}A^{M-1-j}\tilde{f}(j\cdot b)+(M\cdot c) BA^{M}\alpha_1(\gamma)(a,b).
\end{align*}
Since one can choose $c$ with $M\cdot c=1$, this is equivalent to
\begin{equation}\label{equivalente a 3.5'}
\begin{aligned}
0 & = \sum_{j=0}^{p^{\eta}-1}A^{p^{\eta}-1-j}\tilde{f}(j\cdot 1) +  B \sum_{j=0}^{p^{\eta}-1} A^j \tilde{f}\left(b_j\right)- B \sum_{j=p^{\eta}+M-l}^{p^{\eta}+M-1} A^j \tilde{f}\left(b_j\right) \\
& + B \sum_{j=0}^{l-1} A^{M-l+j} \tilde{f}\left(b_j\right) + B \sum_{j=0}^{l-1}A^{M-1-j}\tilde{f}(j\cdot b) + BA^{M}\alpha_1(\gamma)(a,b).
\end{aligned}
\end{equation}
Now, by equality~\eqref{provisoria},
\begin{equation*}
B\sum_{j=p^{\eta}+M-l}^{p^{\eta}+M-1} A^j \tilde{f}\left(b_j\right) = B\sum_{j=0}^{l-1} A^{M-l+j} \tilde{f}\left(b_{M-l+j}\right) = B\sum_{j=0}^{l-1} A^{M-l+j} \tilde{f}\left(b_j+c_j\right),
\end{equation*}
which combined with equality~\eqref{eq1 lema tecnico}, shows that~\eqref{equivalente a 3.5'} is equivalent to
\begin{equation}\label{ultima equivalencia l mayor}
0=\sum_{j=0}^{p^{\eta}-1}A^{p^{\eta}-1-j}\tilde{f}(j\cdot 1)+ B\sum_{j=0}^{p^{\eta}-1} A^j\tilde{f}\left(b_j\right)+B\sum_{j=0}^{l-1} A^{M-l+j+1}\alpha_1(\gamma)(1,b_j+ c_j).
\end{equation}
Using again that $b_j+c_j=b_{M-l+j}$, we obtain
\begin{equation*}
B\sum_{j=0}^{l-1} A^{M-l+j+1}\alpha_1(\gamma)(1,b_j+c_j)= B\sum_{j=0}^{l-1} A^{M-l+j+1}\alpha_1(\gamma)(1,b_{M-l+j})= B\sum_{j=p^{\eta}+M-l}^{p^{\eta}+M-1} A^{j+1}\alpha_1(\gamma)\left(1,b_j\right)=B\gamma,
\end{equation*}
since, by Proposition~\ref{calculos de beta 1} and Remark~\ref{b_k=1 si y...}, $\alpha_1(\gamma)(1,b_{p^{\eta}-1}) =\gamma$ and $\alpha_1(\gamma)(1,b_j) = 0$ if $j\not\equiv -1\pmod{p^{\eta}}$. Now,~\eqref{ultima equivalencia l mayor} follows from the first equality in~\eqref{necesarias y suficientes con B} and Corollary~\ref{corolario de periodico en c igual a 1} with $c=1$, which concludes the proof of~\eqref{cociclo horizontal con D} in this case.
\end{proof}

\begin{theorem}\label{cobordes Bne0} A $2$-cocycle of the form $(\alpha_1(\gamma),f)\in \wh{C}^{02}_N(H,I)\oplus \wh{C}^{11}_N(H,I)$ is a $2$-coboundary if and only~if there exists $t_0\in I$, such that
\begin{equation}\label{coborde actua B}
f(1,1) =  (\Ide-A+(p^{\eta}-p^{\nu})\Ide+(p^{\nu}-1)BA)t_0 \qquad \text{and}\qquad \gamma= p^{\eta}t_0.
\end{equation}
\end{theorem}

\begin{proof} Let $t\in \wh{C}^{01}_N(H,I)$. The argument at the beginning of the proof of Theorem~\ref{cobordes} shows that~\eqref{equiv a beta1 es borde} is satisfied, and that, if $\partial_{\mathrm{v}}^{02}(t)=\alpha_1(\gamma)$, then $\gamma=p^{\eta}t(1)$. Assume now that $f = \partial_{\mathrm{h}}^{11}(t) + D_{01}^{11}(t)$. Then, by the definitions of $\partial_{\mathrm{h}}^{11}$ and~$D_{01}^{11}$, and equal\-ities~\eqref{def de diamante general}, we have
\begin{equation*}
f(a,b)=\partial_{\mathrm{h}}^{11}(t)(a,b)+D_{01}^{11}(t)(a,b)=t\bigl(a\cdot b\bigr) - A^{l(a)}t(b) - (a\cdot b)BA^{l(a)}t(a).
\end{equation*}
A direct computation using this and the formula for $t$ obtained in~\eqref{equiv a beta1 es borde}, yields
$$
f(1,1)=(\partial_{\mathrm{h}}^{11}(t)+ D_{01}^{11}(t))(1,1) =t(1-p^{\nu}) - (A+(1-p^{\nu})BA)t(1) = (p^{\eta} +1 - p^{\nu})t(1)  - (A+(1-p^{\nu})BA)t(1).
$$
So, if $\bigl(\partial_{\mathrm{v}}^{02}(t),\partial_{\mathrm{h}}^{11}(t)+ D_{01}^{11}(t)\bigr) = (\alpha_1(\gamma),f)$, then condition~\eqref{coborde actua B} is satisfied with $t_0\coloneqq t(1)$. Conversely, assume that~\eqref{coborde actua B} is satisfied and define $t\in \wh{C}^{01}_N(H,I)$ by $t(j)\coloneqq jt_0$, for $0\le j<p^{\eta}$. Since $t(p^{\eta}) = 0 =  p^{\eta} t_0 - \gamma$,~by~con\-di\-tion~\eqref{equiv a beta1 es borde}, we know that $\partial_{\mathrm{v}}^{02}(t) = \alpha_1(\gamma)$. Moreover, computing as above, we obtain
\begin{equation*}
\partial_{\mathrm{h}}^{11}(t)(1,1) + D_{01}^{11}(t)(1,1) = (p^{\eta} +1 - p^{\nu})t(1)  - (A+(1-p^{\nu})BA)t(1).
\end{equation*}
Consequently, $\partial_{\mathrm{h}}^{11}(t)(1,1)+ D_{01}^{11}(t)(1,1)= f(1,1)$, which concludes the proof, since $\partial_{\mathrm{v}}^{02}(t) = \alpha_1(\gamma)$ and, by Proposition~\ref{necesaria para Yleft}, in a $2$-cocycle of the form $(\alpha_1(\gamma),f)$, the map $f$ is univocally determined by $f(1,1)$ and $\gamma$.
\end{proof}

Let $F_1,F_2\colon I^2 \to I$ and $G\colon I\to I^2$ be the maps defined by
\begin{align*}
&F_1(y_1,y_2) \coloneqq p^{\eta}y_1- V y_2,\\
& F_2(y_1,y_2) \coloneqq P(A)(y_1-y_2)+Q(A)y_2+ABR(A)y_1+(S+1)ABy_2
\shortintertext{and}
& G(y) \coloneqq (Vy+(p^{\nu}-1)BAy,p^{\eta}y),
\end{align*}
where $V\coloneqq p^{\eta}\Ide- p^{\nu}\Ide +\Ide-A$ and $P$, $Q$, $R$ and $S$ are as in \eqref{los polinomios1} and~\eqref{mas polinomios}. The argument given above~Corol\-lary~\ref{calculo de la homo}, shows that the map
\begin{equation*}
F\colon \ker(F_1)\cap \ker(F_2) \to \wh{C}^{02}_N(H,I)\oplus \wh{C}^{11}_N(H,I),
\end{equation*}
given by $F(f_0,\gamma)\coloneqq \bigl(\alpha_1(\gamma),f_{f_0,\gamma}\bigr)$, where $\alpha_1(\gamma)$ is as in~\eqref{sime} and $f_{f_0,\gamma}$ is as in Definition~\ref{definicion de breve f con B}, is a linear map.

\begin{corollary}\label{calculo de la homo B'} The map $F$ induces an iso\-morphism
\begin{equation*}
\frac{\ker(F_1)\cap \ker(F_2)}{\ima(G)}\simeq \ho_{\blackdiamond}^2(H,I).
\end{equation*}
\end{corollary}

\begin{proof} This follows immediately from Remark~\ref{si cociclo da lo mismo1} and Theorems~\ref{cociclos con B} and~\ref{cobordes Bne0}.
\end{proof}

\section{Examples}\label{seccion 4}

In Subsection~\ref{seccion 4.1}, we explicitly compute all extensions with $\Yleft=0$, of $H$, endowed with the trivial linear cycle set structure (case $\eta=\nu$), by a finite cyclic $p$-group $I=\mathds{Z}_{p^r}$. In Subsection~\ref{seccion 4.2}, we construct some extensions when $H$ is non-trivial (case $\nu<\eta$) or $\Yleft\ne 0$.

\begin{lemma}\label{para ejemplos} Let $p\in \mathds{N}$ be an odd prime and let $r\ge 1$. Let $0\le a<p^r$ and $\eta\in \mathds{N}$. Then
\begin{equation*}
a^{p^{\eta}}\equiv 1\pmod{p^r} \Longleftrightarrow a^{p^{\eta_0}}\equiv 1\pmod{p^r} \Longleftrightarrow  a\equiv 1\pmod{p^{r-\eta_0}},
\end{equation*}
where $\eta_0\coloneqq \min(r-1,\eta)$
\end{lemma}

\begin{proof} If $a\notin U(\mathds{Z}_{p^r})$, none of the conditions are met. Else, by Euler Theorem, $a^{(p-1)p^{r-1}}\equiv 1\pmod{p^r}$. Hence, $a^{p^{\eta}}\equiv 1\pmod{p^r}$ if and only if $a^{p^{\eta_0}}\equiv 1\pmod{p^r}$. When $r=1$, this finishes the proof. So we assume that $r>1$. Let~$s$ be a generator of $U(\mathds{Z}_{p^r})$ and write $\ov{a} = s^l$, where $\ov{a}$ is the class of $a$ in $\mathds{Z}_{p^r}$. Note that $a^{p^{\eta_0}} \equiv 1\pmod{p^r}$ if and only if $(p-1)p^{r-1-\eta_0}\mid l$. Let $\hat{a}$ and $\hat{s}$ be the classes in $\mathds{Z}_{p^{r-1}}$, of $a$ and $s$, respectively. Since $\hat{s}$ is a generator of $U(\mathds{Z}_{p^{r-1}})$ and $\hat{a} = \hat{s}^l$, the same argument as above proves that $(p-1)p^{(r-2)-(\eta_0-1)}= (p-1)p^{r-1-\eta_0}\mid l$ if and only if $a^{p^{\eta_0-1}}\equiv 1\pmod{p^{r-1}}$. An inductive argument concludes the proof.
\end{proof}

\begin{remark}\label{complemento para ejemplos} The previous lemma and its proof hold for $p=2$ when $1\le r\le 2$ (because then $U(\mathds{Z}_{2^r})$ is cyclic).
\end{remark}

\begin{lemma}\label{para ejemploscon2} Let $r>2$, $0\le a<2^r$ and $\eta\in \mathds{N}$. Then
\begin{equation*}
a^{2^{\eta}}\equiv 1\pmod{2^r} \Longleftrightarrow a^{2^{\eta_0}}\equiv 1\pmod{2^r} \Longleftrightarrow  a^2\equiv 1\pmod{2^{r+1-\eta_0}},
\end{equation*}
where $\eta_0\coloneqq \min(r-2,\eta)$.
\end{lemma}

\begin{proof}  If $a\notin U(\mathds{Z}_{2^r})$, none of the conditions are met. Else $a^{2^{r-2}}\equiv 1\pmod{2^r}$, and so $a^{2^{\eta}}\equiv 1\pmod{2^r}$ if and only if $a^{2^{\eta_0}}\equiv 1\pmod{2^r}$. When $r=3$, this finishes the proof. So we assume that $r>3$. Let $s\in U(\mathds{Z}_{2^r})$ be such that $U(\mathds{Z}_{2^r}) = \{s^i,-s^i:0\le i<2^{r-2}\}$. Write $\ov{a} = \pm s^l$, where $\ov{a}$ is the class of $a$ in $\mathds{Z}_{2^r}$. Since $\eta_0\ge 1$, we have $a^{2^{\eta_0}} \equiv 1\pmod{2^r}$ if and only if $2^{r-2-\eta_0}\mid l$. Let $\hat{a}$ and $\hat{s}$ be the classes in $\mathds{Z}_{2^{r-1}}$, of $a$ and $s$, respectively. Since $U(\mathds{Z}_{2^{r-1}}) = \{\hat{s}^i,-\hat{s}^i:0\le i<2^{r-3}\}$ and $\hat{a} = \pm\hat{s}^l$, the same argument as above proves that $2^{(r-3)-(\eta_0-1)}= 2^{r-2-\eta_0}\mid l$ if and only if $a^{2^{\eta_0-1}}\equiv 1\pmod{2^{r-1}}$. An inductive argument concludes the proof.
\end{proof}

\begin{remark}\label{A para ejemploscon2} It is well known that $a^2\equiv 1\pmod{2^{r+1-\eta_0}}$ if and only if $a\equiv \pm 1\pmod{2^{r-\eta_0}}$
\end{remark}

\subsection[Cases with \texorpdfstring{$H$}{H]} trivial and \texorpdfstring{$\Yleft=0$}{-<=0}]{Cases with \texorpdfstring{$\pmb{H}$}{H} trivial and $\texorpdfstring{\pmb{\Yleft=0}}{-<=0}$}\label{seccion 4.1}

Let $H$ be as at the beginning of Section~\ref{seccion 2}. Here we assume that $\nu = \eta$ (i.e. that $H$ is trivial) and we compute all the classes of extensions $(\iota,I\times_{\beta,f}^{\blackdiamond,\Yleft} H,\pi)$ of $H$ by a cyclic $p$-group $I$, when $\Yleft=0$ (which, by Remark~\ref{condicion para I incluido en Soc}, is equivalent to say that $\iota(I) \subseteq \Soc (I\times_{\beta,f}^{\blackdiamond,\Yleft} H)$). By Proposition~\ref{condiciones sin potencia general} and Corollary~\ref{equiv caso general}, for this we must compute all the endomorphisms $A\colon I\to I$, such that $A^{p^{\eta}} = \Ide$ and then, for each one, the cohomology group $\ho_{\blackdiamond}^2(H,I)$. For the last task, we will use Corollary~\ref{calculo de la homo}. Thus, in each of the cases that we study, we must calculate the maps $F_1$, $F_2$, $G$ and $F$, introduced above that corollary, in order to obtain a family of pairs $(f_0,\gamma)\in \ker(F_1)\cap \ker(F_2)$, which parametrizes a complete family $F(f_0,\gamma) = (\alpha_1(\gamma),f_{f_0})$, of representatives of $2$-cocycles modulo coboundaries in ${\mathcal{C}}_{\blackdiamond}(H,I)$. In each one of the examples, the family of parameters $(f_0,\gamma)$ is in bijective correspondence with a set of pairs $(z_1,z_2)$ where $z_1$ and $z_2$ belong to subquotients of $\mathds{Z}_{p^r}$. For the sake of simplicity in all the cases we limit ourselves to writing the representative~$(f_0,\gamma)$ as a function of $z_1$ and $z_2$. For instance, the phrase ``In this case $(f_0,\gamma)\in\{(p^{r-\eta}z_1,z_2):0\le z_1,z_2<p^{\eta}\}$'' means ``In this case a complete family $(\alpha_1(\gamma),f_{f_0})$ of representatives of $2$-cocycles modulo coboundaries in ${\mathcal{C}}_{\blackdiamond}(H,I)$ is parameterized by the pairs $(f_0,\gamma)\in \{(p^{r-\eta}z_1,z_2):0\le z_1,z_2<p^{\eta}\}$''.

\smallskip

Since $F(f_0,\gamma)\coloneqq \bigl(\alpha_1(\gamma),f_{f_0,\gamma}\bigr)$, we next compute, for convenience, the map $f_{f_0,\gamma}$, under the condition that $\nu=\eta$. Since $1^{\times h} = h = j\cdot h$, for all $j,h\in H$, by Definition~\ref{definicion de breve f} and Proposition~\ref{previa a cociclos}, we have
\begin{equation}\label{f simplificado}
f_{f_0,\gamma}(h,c)=f_{f_0,\gamma}(1^{\times h},c)=\sum_{j=0}^{h-1} A^{h-1-j}\tilde {f}(c)=c\sum_{j=0}^{h-1} A^{h-1-j}\tilde {f}(1) = c\sum_{j=0}^{h-1} A^j f_0.
\end{equation}
Hence $f_{f_0,\gamma}$ does not depend on $\gamma$, and we will write $f_{f_0}$ instead of $f_{f_0,\gamma}$. Furthermore, due to the simplicity of formula~\eqref{sime}, we will determine only $\gamma$ in all cases, rather than $\alpha_1(\gamma)$. In many of these cases, there are simplifications of the formula for $f_{f_0}$ given in~\eqref{f simplificado}. In all the computations we will assume that $0\le h<p^{\eta}$, for all $h\in H$, and we will use freely the notations of Corollary~\ref{calculo de la homo}.

\subsubsection[Case \texorpdfstring{$H=\mathds{Z}_{p^{\eta}}$}{H=Zpn} and \texorpdfstring{$I=\mathds{Z}_{p^r}$}{I=Zpr} with \texorpdfstring{$r\le \eta$}{r<=n}, if \texorpdfstring{$p$}{p} is odd, and \texorpdfstring{$r\le \min(2,\eta)$}{r<= min(2,n)}, if \texorpdfstring{$p=2$}{p=2}]{Case \texorpdfstring{$\pmb{H=\mathds{Z}_{p^{\eta}}}$}{H=Zpn} and \texorpdfstring{$\pmb{I=\mathds{Z}_{p^r}}$}{I=Zpr} with \texorpdfstring{$\pmb{r\le \eta}$}{r<=n}, if \texorpdfstring{$\pmb{p}$}{p} is odd, and \texorpdfstring{$\pmb{r\le \min(2,\eta)}$}{r<= min(2,n)}, if \texorpdfstring{$\pmb{p=2}$}{p=2}}

\begin{proposition} Let $H\coloneqq\left(\mathds{Z}_{p^{\eta}},+\right)$, endowed with the trivial cycle set structure, and let $I\coloneqq \mathds{Z}_{p^r}$. Assume~$r\le \eta$, if $p$ is odd, and $r\le \min(2,\eta)$, if $p=2$. There exists $0\le k<p^{r-1}$, such that $h\blackdiamond y =  (1+pk)^h y = \sum_{i=0}^h \binom{h}{i} p^ik^iy$, and there are two cases:
\begin{description}

\item[$k=0$] In this case $(f_0,\gamma)\in\{(z_1,z_2):0\le z_1,z_2<p^r\}$. Moreover $f_{f_0}(h,h')=f_0 hh'$.

\item[$k\ne 0$]In this case $(f_0,\gamma)\in\{(z_1,p^{r-u-1}z_2):0\le z_1,z_2<p^{u+1}\}$, where $0\le u<r-1$ is such that $k=p^us$ with $p\nmid s$. Moreover, $f_{f_0}(h,h') = \sum_{l=0}^{h-1} \binom{h}{l+1} p^{(u+1)l}s^l f_0 h'$.
\end{description}

\end{proposition}

\begin{proof} Let $A$ be as in~Propo\-sition~\ref{condiciones sin potencia general}. Since~$A$ is the multiplication by an invertible element $a\in \mathds{Z}_{p^r}$ and $A^{p^{\eta}} = \Ide$, we know that $a = 1+pk$ with $0\le k< p^{r-1}$ (see Lemma~\ref{para ejemplos} and Remark~\ref{complemento para ejemplos}). Since $1^{\times h}=h$, from the first equality in~\eqref{def de diamante general}, we obtain the formula for $h \blackdiamond y$. Let $F_1$, $F_2$ and $G$ be as above Corollary~\ref{calculo de la homo} and let $\ov{(f_0,\gamma)}\in \frac{\ker F_1\cap \ker F_2}{\ima G}$. From equality~\eqref{f simplificado}, we obtain
\begin{equation}\label{cociclo1 ej 3}
f_{f_0}(h,h') =\sum_{j=0}^{h-1} (1+pk)^j f_0h'= \sum_{j=0}^{h-1} \sum_{l=0}^j\binom{j}{l} p^lk^lf_0h'= \sum_{l=0}^{h-1} \sum_{j=l}^{h-1} \binom{j}{l} p^lk^lf_0h' = \sum_{l=0}^{h-1} \binom{h}{l+1} p^lk^lf_0h'.
\end{equation}
We claim that
\begin{equation}\label{GFF en ejemplo 3}
G(y) = (-pky,0),\quad F_1(y_1,y_2) = pky_2\quad\text{and}\quad F_2(y_1,y_2) = 0.
\end{equation}
In fact, the formulas for $G$ and $F_1$ follow by a direct computation, using that $p^{\eta}I=0$, because $r\le\eta$. In order to obtain the formula for $F_2$, we first need to calculate $P(a)$ and $Q(a)$, where $P$ and $Q$ are as in~\eqref{los polinomios1}. Since $p^r\mid p^{\eta}$ and $a^j = (1+pk)^j = \sum_{l=0}^j \binom{j}{l}p^lk^l$, we have
\begin{equation*}
P(a) = \sum_{j=0}^{p^{\eta}-1} (jp^{\eta}+1)a^j = \sum_{j=0}^{p^{\eta}-1} a^j = \sum_{j=0}^{p^{\eta}-1}\sum_{l=0}^j \binom{j}{l}p^lk^l = \sum_{l=0}^{p^{\eta}-1}\sum_{j=l}^{p^{\eta}-1} \binom{j}{l}p^lk^l = \sum_{l=0}^{p^{\eta}-1}\binom{p^{\eta}}{l+1} p^lk^l.
\end{equation*}
Write $l+1 = p^tv$ with $p\nmid v$ and $t\ge 0$. It is well known that $p^{\eta-t}\mid \binom{p^{\eta}}{l+1}$. Hence, the exponent $w$, of $p$ in $\binom{p^{\eta}}{l+1} p^l$, is greater than or equal to $\eta-t+l$. Consequently $P(a) = 0$, because $\eta-t+l\ge \eta\ge r$, since $l\ge t$ and $r\le \eta$. We next compute $Q(a)$. By the very definition of $Q(a)$, we have
\begin{equation*}
Q(a) =\sum_{j=0}^{p^{\eta}-1}jp^{2\nu-\eta}a^j + p^{\nu}a =\sum_{j=0}^{p^{\eta}-1}jp^{\eta}a^j + p^{\eta} a =  0.
\end{equation*}
Thus, $F_2=0$, as we want. Hence~\eqref{GFF en ejemplo 3} holds, and so
\begin{equation*}
\frac{\ker F_1\cap \ker F_2}{\ima G} = \begin{cases} \mathds{Z}_{p^r}\oplus \mathds{Z}_{p^r} & \text{if $k=0$,}\\ \frac{\mathds{Z}_{p^r}}{p^{u+1}\mathds{Z}_{p^r}}\oplus p^{r-u-1}\mathds{Z}_{p^r} & \text{if $k=p^us$ with $0\le u< r-1$ and $p\nmid s$.}\end{cases}
\end{equation*}
The statement follows easily from this fact and formula~\eqref{cociclo1 ej 3}.
\end{proof}

\subsubsection[Case \texorpdfstring{$H=\mathds{Z}_{p^{\eta}}$}{H=Zpn} and \texorpdfstring{$I=\mathds{Z}_{p^r}$}{I=Zpr} with \texorpdfstring{$r>\eta$}{r>n}, if \texorpdfstring{$p$}{p} is odd, and \texorpdfstring{$(\eta,r)=(1,2)$}{(n,r)=(1,2)}, if \texorpdfstring{$p=2$}{p=2}]{Case \texorpdfstring{$\pmb{H=\mathds{Z}_{p^{\eta}}}$}{H=Zpn} and \texorpdfstring{$\pmb{I=\mathds{Z}_{p^r}}$}{I=Zpr} with \texorpdfstring{$\pmb{r>\eta}$}{r>n}, if \texorpdfstring{$\pmb{p}$}{p} is odd, and \texorpdfstring{$\pmb{(\eta,r)=(1,2)}$}{(n,r)=(1,2)}, if \texorpdfstring{$\pmb{p=2}$}{p=2}}

\begin{proposition}\label{prop 4.10} Let $H\coloneqq\left(\mathds{Z}_{p^{\eta}},+\right)$ endowed with the trivial cycle set structure and let $I\coloneqq \mathds{Z}_{p^r}$. Assume $\eta<r$, if $p$ is odd, and $(\eta,r) = (1,2)$, if $p=2$. Then, there exists $0\le k<p^{\eta}$, such that $h\blackdiamond y = (1+p^{r-\eta}k)^h y$. If $k\ne 0$, then we write $k=p^us$ with $0\le u<\eta$ and $p\nmid s$. There are three cases:
\begin{description}

\item[$k=0$] In this case $(f_0,\gamma)\in \{(p^{r-\eta}z_1,z_2):0\le z_1,z_2<p^{\eta}\}$. Moreover, $f_{f_0}(h,h') = f_0hh'$.

\item[$k\ne 0$ and $r+u\ge 2\eta$] If $p\ne 2$, then $(f_0,\gamma)\in\{(p^{r-\eta}(z_1-sz_2),p^{\eta-u}z_2):0\le z_1<p^{\eta}\text{ and }
0\le z_2<p^u\}$. More\-over,
\begin{equation*}
h \blackdiamond y =y+p^{r-\eta+u}shy\quad\text{and}\quad f_{f_0}(h,h') = f_0hh'.
\end{equation*}
If $p=2$, then $k=1$ and $(f_0,\gamma)\in \{(2z_1-z_2,z_2):0\le z_1, z_2<2\}$. More\-over,
$$
h \blackdiamond y = (1+2h)y\quad\text{and}\quad f_{f_0}(h,h') = f_0hh'.
$$

\item[$k\ne 0$ and $r+u< 2\eta$] Then $p\ne 2$ and $(f_0,\gamma)\in\{(p^{r-\eta}sz_1,p^{\eta-u}(z_2-z_1)): 0\le z_1<p^u \text{ and } 0\le z_2<p^{r-\eta+u}\}$. More\-over,
\begin{equation*}
h \blackdiamond y =\sum_{i=0}^h \binom{h}{i} p^{(r-\eta+u)i}s^iy\quad\text{and}\quad f_{f_0}(h,h') =\sum_{l=0}^{h-1} p^{(r-\eta+u)l}s^l \binom{h}{l+1} f_0h'.
\end{equation*}
\end{description}
\end{proposition}

\begin{proof} Let $A$ be as in~Propo\-sition~\ref{condiciones sin potencia general}. Since~$A$ is the multiplication by an invertible element $a\in \mathds{Z}_{p^r}$ and $A^{p^{\eta}} = \Ide$, by Lemma~\ref{para ejemplos} and Remark~\ref{complemento para ejemplos}, we know that $a = 1+p^{r-\eta}k$ with $0\le k<p^{\eta}$. Since $1^{\times h}=h$, by the first equality in~\eqref{def de diamante general}, we have
\begin{equation}\label{accion1 ej 4}
h \blackdiamond y = (1+p^{r-\eta}k)^h y = \sum_{i=0}^h \binom{h}{i} p^{(r-\eta)i}k^iy.
\end{equation}
Let $F_1$, $F_2$ and $G$ be as above Corollary~\ref{calculo de la homo} and let $\ov{(f_0,\gamma)}\in \frac{\ker F_1\cap \ker F_2}{\ima G}$. By~equal\-ity~\eqref{f simplificado}, we have
\begin{equation}\label{cociclo1 ej 4}
f_{f_0}(h,h') = \sum_{j=0}^{h-1}a^j f_0h' = \sum_{j=0}^{h-1} \sum_{l=0}^j \binom{j}{l} p^{(r-\eta)l}k^lf_0h'= \sum_{l=0}^{h-1} \binom{h}{l+1} p^{(r-\eta)l}k^lf_0h'.
\end{equation}
We claim that
\begin{equation}\label{GFF ej 4}
G(y) = (-p^{r-\eta}ky,p^{\eta}y),\quad  F_1(y_1,y_2) = p^{\eta}y_1+p^{r-\eta}ky_2\quad\text{and}\quad F_2(y_1,y_2) = \begin{cases} \left(\frac{p^{3\eta}-p^{2\eta}}{2}+p^{\eta}\right)y_1 &\text{if $p\ne 2$,}\\ 2k(y_1-y_2) &\text{if $p = 2$.}\end{cases}
\end{equation}
In fact, the formulas for $G$ and $F_1$ follow by a direct computation. In order to obtain the formula for $F_2$, we first need to calculate $P(a)$ and $Q(a)$, where $P$ and $Q$ are as in~\eqref{los polinomios1}. Since $a^j = (1+p^{r-\eta}k)^j = \sum_{l=0}^j \binom{j}{l}p^{(r-\eta)l}k^l$,
\begin{equation*}
P(a) = \sum_{j=0}^{p^{\eta}-1} (jp^{\eta}+1)a^j = \sum_{j=0}^{p^{\eta}-1} jp^{\eta} + \sum_{j=0}^{p^{\eta}-1}\sum_{l=0}^j \binom{j}{l}p^{(r-\eta)l}k^l,
\end{equation*}
where in the last equality we have used that $p^r\mid p^{\eta}p^{(r-\eta)l}$, for all $l\ge 1$. Computing the expression at the right side of the second equality, we obtain
\begin{equation*}
P(a) = \frac{(p^{\eta}-1)p^{2\eta}}{2} +\sum_{l=0}^{p^{\eta}-1}\sum_{j=l}^{p^{\eta}-1} \binom{j}{l}p^{(r-\eta)l}k^l = \frac{(p^{\eta}-1)p^{2\eta}}{2} + \sum_{l=0}^{p^{\eta}-1}\binom{p^{\eta}}{l+1} p^{(r-\eta)l}k^l.
\end{equation*}
If $p=2$, then $\eta=1$, $r=2$, and a direct computation gives $P(a) = 2k$. Suppose $p$ odd and write $l+1 = p^tv$ with $p\nmid v$ and $t\ge 0$. It is well known that $p^{\eta-t}\mid \binom{p^{\eta}}{l+1}$. Hence, the exponent $w$, of $p$ in $\binom{p^{\eta}}{l+1} p^{(r-\eta)l}$, is greater than or equal to $\eta-t+(r-\eta)l$. Thus, if $l\ge 1$, then
\begin{equation*}
w\ge \eta-t+(r-\eta)l = r-t+(r-\eta)(l-1)\ge r-t + l-1\ge r,
\end{equation*}
where we have used that $l-1 = vp^t-2\ge t$, for $l\ge 1$. Consequently, in this case, $p^r\mid \binom{p^{\eta}}{l+1}p^{(r-\eta)l}$, and so
\begin{equation*}
P(a) = \frac{(p^{\eta}-1)p^{2\eta}}{2} + p^{\eta}.
\end{equation*}
We next compute $Q(a)$. By the very definition of $Q(a)$, we have
\begin{equation*}
Q(a) = p^{\eta}\sum_{j=0}^{p^{\eta}-1} j a^j  + p^{\eta} = p^{\eta}\sum_{j=0}^{p^{\eta}-1} j \sum_{l=0}^j \binom{j}{l}p^{(r-\eta)l}k^l + p^{\eta}.
\end{equation*}
If $p=2$, then $\eta=1$, $r=2$, and a direct computation gives $Q(a) = 0$. Suppose $p$ odd. Since $p^r\mid p^{\eta}p^{(r-\eta)l}$ for $l\ge 1$, the formula for $Q(a)$ reduces to
\begin{equation*}
Q(a) = p^{\eta}\sum_{j=1}^{p^{\eta}-1}j + p^{\eta} = \frac{(p^{\eta}-1)p^{2\eta}}{2} + p^{\eta} = P(a).
\end{equation*}
The formula for $F_2$ follows from these facts by an immediate computation, and so the claim~\eqref{GFF ej 4} holds. Assume that $k=0$. Then
\begin{equation*}
\frac{\ker F_1\cap \ker F_2}{\ima G} = p^{r-\eta}\mathds{Z}_{p^r}\oplus \frac{\mathds{Z}_{p^r}}{p^{\eta}\mathds{Z}_{p^r}}.
\end{equation*}
Consequently, we can take $(f_0,\gamma)\in\{(p^{r-\eta}z_1,z_2):0\le z_1,z_2<p^{\eta}\}$, and the formulas for $h\blackdiamond y$ and $f_{f_0}$ follow directly from~\eqref{accion1 ej 4} and~\eqref{cociclo1 ej 4}.

Assume now that $k\ne 0$. If $p=2$, then $k=1$, and so $G(y)=(2y,2y)$, $F_1(y_1,y_2)=F_2(y_1,y_2)=2(y_1+y_2)$ and a straightforward computation shows that
$$
\ker F_1\cap \ker F_2=\ker F_1=\{(y_1,y_2): y_1\equiv y_2 \pmod{2}\}\subseteq \mathds{Z}_4\oplus \mathds{Z}_4\quad\text{and}\quad\ima(G)= \{(0,0),(2,2)\}.
$$
Another straightforward computation shows that then we can take $(f_0,\gamma)\in\{(2z_1-z_2,z_2): 0\le z_1,z_2 <2 \}$.~More\-over, using that $h\in \mathds{Z}_2$, we obtain the formulas for $h\blackdiamond y$ and $f_{f_0}$.

For the rest of the proof we can and will assume that $p$ is odd and $k\ne 0$. Assume furthermore that $r+u\ge 2\eta$, and define the map $\Phi\colon I\oplus I\to I\oplus I$, by
$$
\Phi(y_1,y_2)\coloneqq (y_1-p^{r+u-2\eta}sy_2,y_2).
$$
Set $\wt{G}\coloneqq \Phi^{-1}\xcirc G$, $\wt{F}_1\coloneqq F_1\xcirc \Phi$ and  $\wt{F}_2\coloneqq F_2\xcirc \Phi$. Clearly $\Phi$ induces an iso\-morphism $\ov{\Phi}\colon \frac{\ker \wt{F}_1\cap \ker \wt{F}_2}{\ima \wt{G}}\to \frac{\ker F_1\cap \ker F_2}{\ima G}$. A direct computation, using~\eqref{GFF ej 4}, shows that
\begin{equation*}
\wt{G}(y) = (0,p^{\eta}y),\quad \wt{F}_1(y_1,y_2) = p^{\eta}y_1\quad\text{and}\quad \wt{F}_2(y_1,y_2) = \left(\frac{(p^{\eta}-1)p^{2\eta}}{2}+ p^{\eta}\right)y_1-p^{r+u-\eta}sy_2.
\end{equation*}
Hence,
\begin{equation*}
\frac{\ker \wt{F}_1\cap \ker \wt{F}_2}{\ima \wt{G}} =  p^{r-\eta}\mathds{Z}_{p^r}\oplus \frac{p^{\eta-u}\mathds{Z}_{p^r}}{p^{\eta}\mathds{Z}_{p^r}}.
\end{equation*}
Take $(p^{r-\eta}z_1,\ov{p^{\eta-u}z_2})\in  p^{r-\eta}\mathds{Z}_{p^r}\oplus \frac{p^{\eta-u}\mathds{Z}_{p^r}}{p^{\eta} \mathds{Z}_{p^r}}$, where $0\le z_1<p^{\eta}$ and $0\le z_2<p^u$. Then
\begin{equation*}
\ov{\Phi}(p^{r-\eta}z_1,\ov{p^{\eta-u}z_2}) = \ov{(p^{r-\eta}z_1-p^{r-\eta}sz_2,p^{\eta-u}z_2)}.
\end{equation*}
So, if $k\ne 0$ and $r+u\ge 2\eta$, then we can take $(f_0,\gamma)\in \{(p^{r-\eta}(z_1-sz_2),p^{\eta-u}z_2):0\le z_1<p^{\eta}\text{ and }0\le z_2<p^u\}$. Moreover, since $2(r-\eta)+u\ge r$, equalities~\eqref{accion1 ej 4} and~\eqref{cociclo1 ej 4} become
\begin{align*}
& h \blackdiamond y = \sum_{i=0}^h \binom{h}{i} p^{(r-\eta+u)i}s^iy=y+hp^{r-\eta+u}sy
\shortintertext{and}
&f_{f_0}(h,h') = \sum_{l=0}^{h-1} p^{(r-\eta)(l+1)+ul}s^l(z_1-sz_2)\binom{h}{l+1}h' = p^{r-\eta}(z_1-sz_2)hh' =f_0hh'.
\end{align*}
Assume finally that $k\ne 0$ and $r+u< 2\eta$. We define the map $\Phi\colon I\oplus I\to I\oplus I$, by $\Phi(y_1,y_2)\coloneqq (sy_1, y_2-p^{2\eta-r-u}y_1)$, and we set $\wt{G}\coloneqq \Phi^{-1}\xcirc G$, $\wt{F}_1\coloneqq F_1\xcirc \Phi$ and $\wt{F}_2\coloneqq F_2\xcirc \Phi$. Clearly $\Phi$ induces an isomorphism
\begin{equation*}
\ov{\Phi}\colon \frac{\ker \wt{F}_1\cap \ker \wt{F}_2}{\ima \wt{G}}\to \frac{\ker F_1\cap \ker F_2}{\ima G}.
\end{equation*}
A direct computation, using~\eqref{GFF ej 4}, shows that
\begin{equation*}
\wt{G}(y) = (-p^{r-\eta+u}y,0),\quad \wt{F}_1(y_1,y_2) = p^{r-\eta+u}sy_2\quad\text{and}\quad \wt{F}_2(y_1,y_2) = p^{\eta}sy_1.
\end{equation*}
Consequently,
\begin{equation*}
\frac{\ker \wt{F}_1\cap \ker \wt{F}_2}{\ima \wt{G}} = \frac{p^{r-\eta}\mathds{Z}_{p^r}}{p^{r-\eta+u}\mathds{Z}_{p^r}}\oplus p^{\eta-u}\mathds{Z}_{p^r}.
\end{equation*}
Take $(\ov{p^{r-\eta}z_1},p^{\eta-u}z_2)\in \frac{p^{r-\eta}\mathds{Z}_{p^r}}{p^{r-\eta+u}\mathds{Z}_{p^r}}\oplus p^{\eta-u}\mathds{Z}_{p^r}$, where $0\le z_1< p^u$ and $0\le z_2<p^{r-\eta+u}$. Then,
\begin{equation*}
\ov{\Phi}(\ov{p^{r-\eta}z_1},p^{\eta-u}z_2) = \ov{(p^{r-\eta}sz_1,p^{\eta-u}z_2-p^{\eta-u}z_1)}.
\end{equation*}
So, if $k\ne 0$ and $r+u< 2\eta$, then we can take $(f_0,\gamma)\in\{(p^{r-\eta}sz_1,p^{\eta-u}(z_2-z_1)):0\le z_1<p^u\text{ and }0\le z_2<p^{r-\eta+u}\}$. The formulas for $h\blackdiamond y$ and $f_{f_0}$ are not simplified and are given by~\eqref{accion1 ej 4} and~\eqref{cociclo1 ej 4}.
\end{proof}

\subsubsection[Case \texorpdfstring{$H=\mathds{Z}_{2^{\eta}}$}{H=Z2n} and \texorpdfstring{$I=\mathds{Z}_{2^r}$}{Z2r} with \texorpdfstring{$3\le r\le \eta+1$}{3<=r<=n+1}]{Case \texorpdfstring{$\pmb{H=\mathds{Z}_{2^{\eta}}}$}{H=Z2n} and \texorpdfstring{$\pmb{I=\mathds{Z}_{2^r}}$}{Z2r} with \texorpdfstring{$\pmb{3\le r\le \eta+1}$}{3<=r<=n+1}}

\begin{proposition} Let $H\coloneqq\left(\mathds{Z}_{2^{\eta}},+\right)$, endowed with the trivial cycle set structure, and let $I\coloneqq \mathds{Z}_{2^r}$, where we assume that $3\le r\le \eta+1$. There exists $a = \pm (1+4k)$ with $0\le k<2^{r-2}$, such that $h\blackdiamond y = a^h y $. When $k\ne 0$, we write $k=2^us$ with $0\le u<r-2$ and $2\nmid s$. There are four cases:
\begin{description}

\item[$k=0$ and $a=1$] If $r\le\eta$, then $(f_0,\gamma)\in\{(z_1,z_2) : 0\le z_1,z_2<2^r\}$. Moreover,
\begin{equation*}
h\blackdiamond y = y \quad\text{and}\quad f_{f_0}(h,h')=f_0hh'.
\end{equation*}
If $r=\eta+1$, then $(f_0,\gamma)\in\{(2z_1,z_2) : 0\le z_1,z_2<2^{r-1}\}$. Moreover,
\begin{equation*}
h\blackdiamond y = y\quad\text{and}\quad f_{f_0}(h,h') =f_0hh'.
\end{equation*}

\item[$k=0$ and $a=-1$] In this case $(f_0,\gamma)\in\{(z_1,2^{r-1}z_2+2^{\eta-1}z_1) :  0\le z_1,z_2<2\}$. Moreover,
\begin{equation*}
h\blackdiamond y = (-1)^h y \quad\text{and}\quad f_{f_0}(h,h') = \begin{cases} f_0h'  & \text{if $h$ is odd,}\\ 0 & \text{if $h$ is even.}\end{cases}
\end{equation*}

\item[$k\ne 0$ and $a=1+4k$] If $r\le \eta$, then $(f_0,\gamma)\in\{(z_1,2^{r-u-2}z_2):0\le z_1,z_2<2^{u+2}\}$. Moreover,
\begin{equation*}
h\blackdiamond y = \sum_{i=0}^h \binom{h}{i} 4^ik^iy \quad\text{and}\quad f_{f_0}(h,h') =\sum_{l=0}^{h-1} \binom{h}{l+1} 4^lk^lf_0h'.
\end{equation*}
If $r=\eta+1$, then $(f_0,\gamma)\in\{(2z_1,2^{r-u-2}z_2-2^{\eta-u-1}z_1) : 0\le z_1<2^{u+1} \text{ and } 0\le z_2<2^{u+2}\}$. Moreover,
\begin{equation*}
h\blackdiamond y = \sum_{i=0}^h \binom{h}{i} 4^ik^iy \quad\text{and}\quad f_{f_0}(h,h') =\sum_{l=0}^{h-1} \binom{h}{l+1}4^lk^lf_0h'.
\end{equation*}

\item[$k\ne 0$ and $a=-1-4k$] In this case $(f_0,\gamma)\in\{(z_1,2^{r-1}z_2+2^{\eta-1}z_1) : 0\le z_1, z_2<2\}$. Moreover,
\begin{equation*}
h\blackdiamond y = (-1)^h \sum_{i=0}^h \binom{h}{i} 4^ik^iy \quad\text{and}\quad f_{f_0}(h,h') =\sum_{l=0}^{h-1} A_{hl} 4^lk^lf_0h',
\end{equation*}
where $A_{hl}\coloneqq \sum_{j=l}^{h-1} (-1)^{j} \binom{j}{l}$.
\end{description}
\end{proposition}

\begin{proof} Let $A$ be as in Propo\-sition~\ref{condiciones sin potencia general}. Since~$A$ is the multiplication by an invertible element $a\in \mathds{Z}_{2^r}$ and $A^{p^{\eta}} = \Ide$, by Lemma~\ref{para ejemploscon2} and Remark~\ref{A para ejemploscon2}, we know that $a = \pm (1+4k)$ with $0\le k<2^{r-2}$. Since $1^{\times h}=h$, by the first~con\-di\-tion~\eqref{def de diamante general}, we have:
\begin{equation}\label{calculo diamante p=2 ej 6}
h \blackdiamond y =\begin{cases} (1+4k)^h y = \sum_{i=0}^h \binom{h}{i} 4^ik^iy & \text{if $a = 1+4k$,}\\ (-1)^h(1+4k)^h y = (-1)^h\sum_{i=0}^h \binom{h}{i} 4^ik^iy & \text{if $a = -1-4k$.}\end{cases}
\end{equation}
Let $F_1$, $F_2$ and $G$ be as above Corollary~\ref{calculo de la homo}, and let $\ov{(f_0,\gamma)}\in \frac{\ker F_1\cap \ker F_2}{\ima G}$. From equality~\eqref{f simplificado}, we obtain
\begin{equation*}
f_{f_0}(h,h') = h'\sum_{j=0}^{h-1}a^j f_0,
\end{equation*}
Assume that $a=1+4k$. Then $a^j = \sum_{l=0}^j \binom{j}{l}4^lk^l$, and so
\begin{equation}\label{cociclo1caso1 ej 6}
f_{f_0}(h,h') = \sum_{j=0}^{h-1} \sum_{l=0}^j \binom{j}{l} 4^lk^lf_0h' = \sum_{l=0}^{h-1} \sum_{j=l}^{h-1} \binom{j}{l} 4^lk^lf_0h'= \sum_{l=0}^{h-1} \binom{h}{l+1} 4^lk^lf_0h'.
\end{equation}
Assume now that $a = -1-4k$. Then $a^j = \sum_{l=0}^j (-1)^j\binom{j}{l}4^lk^l$, and so
\begin{equation}\label{cociclo1caso2 ej 6}
f_{f_0}(h,h') = \sum_{j=0}^{h-1} (-1)^j\sum_{l=0}^{j} \binom{j}{l} 4^lk^lf_0h' = \sum_{l=0}^{h-1}\sum_{j=l}^{h-1} (-1)^j\binom{j}{l} 4^lk^lf_0h'= \sum_{l=0}^{h-1} A_{hl} 4^lk^lf_0h'.
\end{equation}
Note that
\begin{equation}\label{pepe1}
A_{h0}=\begin{cases} 0 & \text{if $h$ is even,}\\ 1 & \text{if $h$ is odd,}\end{cases}\qquad\quad\text{and}\qquad A_{h1}=\begin{cases} -\frac{h}{2} & \text{if $h$ is even,}\\ \phantom{-}\frac{h-1}{2} & \text{if $h$ is odd}.\end{cases}
\end{equation}
We claim that
\begin{equation}\label{GFF ej 6}
\begin{aligned}
& G(y) = \begin{cases} (-4ky,2^{\eta}y) &\text{if $a = 1+4k$,}\\ (2y+4ky,2^{\eta}y) &\text{if $a = -1-4k$,}\end{cases}\\
&  F_1(y_1,y_2) = \begin{cases} 2^{\eta}y_1+4ky_2 &\text{if $a = 1+4k$,}\\ 2^{\eta}y_1-(2+4k)y_2 &\text{if $a = -1-4k$,}\end{cases} \\
&  F_2(y_1,y_2) = \begin{cases} 2^{\eta}y_1  &\text{if $a = 1+4k$,}\\ 2^{\eta} y_2 &\text{if $a = -1-4k$.}\end{cases}
\end{aligned}
\end{equation}
The formulas for $G$ and $F_1$ are straightforward. In order to get the formula for $F_2$, we first compute $P(a)$ and~$Q(a)$, where $P$ and $Q$ are as in~\eqref{los polinomios1}. Assume first that $a=1+4k$. Since $r<2\eta$ and $2^r\mid 2^{\eta}4^l$, for all $l\ge 1$, we have
\begin{equation}\label{se anula la suma}
2^{\eta} \sum_{j=0}^{2^{\eta}-1} j a^j = \sum_{j=0}^{2^{\eta}-1} j \sum_{l=0}^j \binom{j}{l}(2^{\eta}4^l)k^l=\sum_{j=0}^{2^{\eta}-1} j2^{\eta} =2^{\eta}\binom{2^{\eta}}{2}=(2^{\eta}-1)2^{2\eta-1}=0.
\end{equation}
Consequently
\begin{equation*}
P(a) = \sum_{j=0}^{2^{\eta}-1} (j2^{\eta}+1)a^j = \sum_{j=0}^{2^{\eta}-1}a^j= \sum_{j=0}^{2^{\eta}-1}\sum_{l=0}^j \binom{j}{l}4^lk^l=
\sum_{l=0}^{2^{\eta}-1}\sum_{j=l}^{2^{\eta}-1} \binom{j}{l}4^lk^l =\sum_{l=0}^{2^{\eta}-1}\binom{2^{\eta}}{l+1} 4^lk^l.
\end{equation*}
Write $l+1 = 2^tv$ with $2\nmid v$ and $t\ge 0$. It is well known that $2^{\eta-t}\mid \binom{2^{\eta}}{l+1}$. Hence, the exponent $w$, of $2$ in $\binom{2^{\eta}}{l+1} 4^l$, is greater than or equal to $\eta-t+2l$. Thus, if $l\ge 1$, then
\begin{equation*}
w \ge \eta-t+2l = \eta-t+2(2^tv-1)=\eta -t+2^{t+1}v-2 \ge \eta+1\ge r,
\end{equation*}
where we have used that if $l\ge 1$, then $2^{t+1}v > t+2$. So, $2^r\mid \binom{2^{\eta}}{l+1}4^l$, for $l\ge 1$, and consequently, $P(a) = 2^{\eta}$.
We next compute $Q(a)$. By~\eqref{se anula la suma}, we have
\begin{equation*}
Q(a) = 2^{\eta}\sum_{j=1}^{2^{\eta}-1} j a^j  + 2^{\eta} =  2^{\eta}.
\end{equation*}
Assume now that $a=-1-4k$. Since $a^j = (-1-4k)^j = (-1)^j\sum_{l=0}^j \binom{j}{l}4^lk^l$ and $2^r\mid 2^{\eta}4^l$ for $l\ge 1$, we have
\begin{equation*}
P(a) = \sum_{j=0}^{2^{\eta}-1} (j2^{\eta}+1)a^j = \sum_{j=0}^{2^{\eta}-1} (-1)^j 2^{\eta}j + \sum_{j=0}^{2^{\eta}-1}(-1)^j\sum_{l=0}^j \binom{j}{l}4^lk^l.
\end{equation*}
Computing the expression at the right side of the second equality and using~\cite{Spivey}*{identity 217}, we obtain
\begin{equation*}
P(a) = -2^{2\eta-1} + \sum_{l=0}^{2^{\eta}-1} \left(\sum_{j=l}^{2^{\eta}-1} (-1)^j\binom{j}{l}\right)4^lk^l= -2^{2\eta-1}+\sum_{l=0}^{2^{\eta}-1} 4^lk^l B_l,
\end{equation*}
where $B_l\coloneqq \sum_{k=1}^l \binom{2^{\eta}}{k}\bigl(\frac{-1}{2}\bigr)^{l+1-k}$. Clearly $B_0=0$ and $B_1=-2^{\eta-1}$. Moreover, using that $2B_{l+1} =  B_l- \binom{2^{\eta}}{l+1}$, for all $l\ge 1$ (and an inductive argument), we obtain that
\begin{equation}\label{divisibilidad de B sub l}
2^{2l-\eta-2} B_l\in \mathds{Z}\qquad\text{for all $l\ge 2$.}
\end{equation}
Hence $4^lB_l\in \mathds{Z}$ and $2^r\mid 4^lB_l$, for $l\ge 1$, which implies that $P(a) = -2^{2\eta-1} =0$. Arguing as above we obtain that
\begin{equation*}
Q(a) = -2^{2\eta-1}  + 2^{\eta}=2^{\eta}.
\end{equation*}
The formula for $F_2$ follows from these facts by an immediate computation, which finishes the proof of~\eqref{GFF ej 6}. Assume now that $k=0$. We have two cases, $a=1$ and $a=-1$. If $a=1$, then
\begin{equation*}
\frac{\ker F_1\cap \ker F_2}{\ima G} = \begin{cases} \mathds{Z}_{2^r}\oplus \mathds{Z}_{2^r} &\text{if $r\le \eta$,}\\ 2\mathds{Z}_{2^r}\oplus \frac{\mathds{Z}_{2^r}}{2^{r-1}\mathds{Z}_{2^r}} &\text{if $r=\eta+1$.}\end{cases}
\end{equation*}
Hence, we can take $\{(f_0,\gamma)\in (z_1,z_2):0\le z_1,z_2<2^r\}$, when $r\le \eta$; and $(f_0,\gamma)\in\{(2 z_1,z_2):0\le z_1,z_2<2^{r-1}\}$, when $r=\eta+1$. For each one of these $2$-co\-cycles, equalities~\eqref{calculo diamante p=2 ej 6} and~\eqref{cociclo1caso1 ej 6} become
\begin{equation*}
h\blackdiamond y = y \quad\text{and}\quad  f_{f_0}(h,h') = \begin{cases}  z_1hh' & \text{if $r\le \eta$,}\\ 2z_1hh' & \text{if $r=\eta+1$.} \end{cases}
\end{equation*}
We now consider the case $a = -1$. We define the map $\Phi\colon I\oplus I\to I\oplus I$, by $\Phi(y_1,y_2)\coloneqq (y_1,y_2+2^{\eta-1}y_1)$, and we set $\wt{G}\coloneqq \Phi^{-1}\xcirc G$, $\wt{F}_1\coloneqq F_1\xcirc \Phi$ and  $\wt{F}_2\coloneqq F_2\xcirc \Phi$. Clearly $\Phi$ induces an isomorphism $\ov{\Phi}\colon \frac{\ker \wt{F}_1\cap \ker \wt{F}_2}{\ima \wt{G}}\to \frac{\ker F_1\cap \ker F_2}{\ima G}$. A direct computation, using~\eqref{GFF ej 6}, shows that
\begin{equation*}
\wt{G}(y) = (2y,0),\quad \wt{F}_1(y_1,y_2) = -2y_2 \quad \text{and}\quad \wt{F}_2(y_1,y_2) =  2^{\eta} y_2.
\end{equation*}
Consequently,
\begin{equation*}
\frac{\ker \wt{F}_1\cap \ker \wt{F}_2}{\ima \wt{G}} =  \frac{\mathds{Z}_{2^r}}{2\mathds{Z}_{2^r}}\oplus 2^{r-1}\mathds{Z}_{2^r}.
\end{equation*}
Take $(\ov{z_1},2^{r-1}z_2)\in \frac{\mathds{Z}_{2^r}}{2\mathds{Z}_{2^r}}\oplus 2^{r-1}\mathds{Z}_{2^r}$, where $0\le z_1,z_2< 2$. Since
\begin{equation*}
\ov{\Phi}(\ov{z_1},2^{r-1}z_2) = \ov{(z_1,2^{r-1}z_2+2^{\eta-1}z_1)},
\end{equation*}
we can take $(f_0,\gamma)\in(z_1,2^{r-1}z_2+2^{\eta-1}z_1):0\le z_1,z_2< 2\}$. For each $(f_0,\gamma)$, the formulas~\eqref{calculo diamante p=2 ej 6} and~\eqref{cociclo1caso2 ej 6} become
\begin{equation*}
h\blackdiamond y = (-1)^h y \quad \text{and} \quad f_{f_0}(h,h') = \begin{cases} z_1h' & \text{if $h$ is odd,}\\ 0 & \text{if $h$ is even.}\end{cases}
\end{equation*}
We now assume that $a=1+4k$ with $k\ne 0$, and that $r\le \eta$. In this case conditions~\eqref{GFF ej 6} becomes
$G(y)=(-4ky,0)$, $F_1(y_1,y_2)=4ky_2$ and $F_2(y_1,y_2)=0$. Hence
$$
\frac{\ker F_1\cap \ker F_2}{\ima G}=\frac{\mathds{Z}_{2^r}}{2^{u+2}\mathds{Z}_{2^r}} \oplus 2^{r-u-2}\mathds{Z}_{2^r},
$$
and we can take $(f_0,\gamma)\in \{(z_1,2^{r-u-2}z_2):0\le z_1, z_2<2^{u+2}\}$. For each $(f_0,\gamma)$, the~formu\-las~\eqref{calculo diamante p=2 ej 6} and~\eqref{cociclo1caso1 ej 6} remain the same.
Assume now that $a=1+4k$ with $k\ne 0$ and that $r=\eta+1$. We define the map $\Phi\colon I\oplus I\to I\oplus I$, by $\Phi(y_1,y_2)\coloneqq (y_1,y_2-2^{\eta-u-2}y_1)$, and we set $\wt{G}\coloneqq \Phi^{-1}\xcirc G$, $\wt{F}_1\coloneqq F_1\xcirc \Phi$ and $\wt{F}_2\coloneqq F_2\xcirc \Phi$. Clearly $\Phi$ induces an isomorphism
\begin{equation*}
\ov{\Phi}\colon \frac{\ker \wt{F}_1\cap \ker \wt{F}_2}{\ima \wt{G}}\to \frac{\ker F_1\cap \ker F_2}{\ima G}.
\end{equation*}
A direct computation, using~\eqref{GFF ej 6}, shows that
\begin{equation*}
\wt{G}(y) = (-4ky,0),\quad \wt{F}_1(y_1,y_2) = 4ky_2 \quad\text{and}\quad\wt{F}_2(y_1,y_2) =  2^{\eta} y_1.
\end{equation*}
Consequently,
\begin{equation*}
\frac{\ker \wt{F}_1\cap \ker \wt{F}_2}{\ima \wt{G}} =  \frac{2\mathds{Z}_{2^r}}{2^{u+2}\mathds{Z}_{2^r}}\oplus 2^{r-u-2}\mathds{Z}_{2^r},
\end{equation*}
and  we take $(\ov{2z_1},2^{r-u-2}z_2)\in \frac{2\mathds{Z}_{2^r}}{2^{u+2}\mathds{Z}_{2^r}}\oplus 2^{r-u-2}\mathds{Z}_{2^r}$,
with $0\le z_1<2^{u+1}$ and $0\le z_2<2^{u+2}$. Since
\begin{equation*}
\ov{\Phi}(\ov{2 z_1},2^{r-u-2}z_2) = \ov{(2z_1,2^{r-u-2}z_2-2^{\eta-u-1}z_1)},
\end{equation*}
we can take $(f_0,\gamma)\in \{(2z_1,2^{r-u-2}z_2-2^{\eta-u-1}z_1):0\le z_1<2^{u+1}\text{ and }0\le z_2<2^{u+2}\}$. For each $(f_0,\gamma)$, the~formu\-las~\eqref{calculo diamante p=2 ej 6} and~\eqref{cociclo1caso1 ej 6} remain the same. Finally, we assume that $a=-1-4k$ with $k\ne 0$. Similarly as before, we define the map $\Phi\colon I\oplus I\to I\oplus I$, by $\Phi(y_1,y_2) \coloneqq (y_1,y_2+2^{\eta-1}y_1)$, and we set $\wt{G}\coloneqq \Phi^{-1}\xcirc G$, $\wt{F}_1\coloneqq F_1\xcirc \Phi$ and~$\wt{F}_2\coloneqq F_2\xcirc \Phi$. Clearly $\Phi$ induces an isomorphism $\ov{\Phi}\colon \frac{\ker \wt{F}_1\cap \ker \wt{F}_2}{\ima \wt{G}}\to \frac{\ker F_1\cap \ker F_2}{\ima G}$. A direct computation, using equalities~\eqref{GFF ej 6} and that $r<2\eta$, shows that
\begin{equation*}
\wt{G}(y) = (2(1+2k)y,0),\quad \wt{F}_1(y_1,y_2) = -2(1+2k)y_2\quad\text{and} \quad\wt{F}_2(y_1,y_2) =  2^{\eta} y_2.
\end{equation*}
Consequently,
\begin{equation*}
\frac{\ker \wt{F}_1\cap \ker \wt{F}_2}{\ima \wt{G}} = \frac{\mathds{Z}_{2^r}}{2\mathds{Z}_{2^r}}\oplus 2^{r-1}\mathds{Z}_{2^r}.
\end{equation*}
Take $(\ov{z_1},2^{r-1}z_2)\in \frac{\mathds{Z}_{2^r}}{2\mathds{Z}_{2^r}}\oplus 2^{r-1}\mathds{Z}_{2^r}$, where $0\le z_1, z_2< 2$. Since
\begin{equation*}
\ov{\Phi}(\ov{z_1},2^{r-1}z_2) = \ov{(z_1,2^{r-1}z_2+2^{\eta-1}z_1)},
\end{equation*}
we can take $(f_0,\gamma)\in \{(z_1,2^{r-1}z_2+2^{\eta-1}z_1):0\le z_1,z_2<2\}$. For each $(f_0,\gamma)$, the formulas~\eqref{calculo diamante p=2 ej 6} and~\eqref{cociclo1caso2 ej 6} remain the same.
\end{proof}

\subsubsection[Case \texorpdfstring{$H=\mathds{Z}_{2^{\eta}}$}{H=Z2n} and \texorpdfstring{$I=\mathds{Z}_{2^r}$}{Z2r} with \texorpdfstring{$r>\eta+1$}{r>n+1}]{Case \texorpdfstring{$\pmb{H=\mathds{Z}_{2^{\eta}}}$}{H=Z2n} and \texorpdfstring{$\pmb{I=\mathds{Z}_{2^r}}$}{I=Z2r} with \texorpdfstring{$\pmb{r>\eta+1}$}{r>n+1}}

\begin{proposition} Let $H\coloneqq\left(\mathds{Z}_{2^{\eta}},+\right)$, endowed with the trivial cycle set structure, and let $I\coloneqq \mathds{Z}_{2^r}$, where we assume that $r>\eta+1$. There exists $a = \pm (1+2^{r-\eta}k)$ with $0\le k<2^{\eta}$, such that $h\blackdiamond y = a^h y$. Moreover, if $k\ne 0$, then we write $k=2^us$ with $0\le u<\eta$ and $2\nmid s$. There are four cases:
\begin{description}

\item[$k=0$ and $a=1$] In this case $(f_0,\gamma)\in\{(2^{r-\eta}z_1,z_2) : 0\le z_1,z_2<2^{\eta}\}$. Moreover, $h\blackdiamond y = y$ and $f_{f_0}(h,h') = f_0 hh'$.

\item[$k=0$ and $a=-1$] In this case $(f_0,\gamma)\in\{(z_1,2^{r-1}z_2+2^{\eta-1}z_1):0\le z_1,z_2<2\}$. Moreover,
\begin{equation*}
h\blackdiamond y = (-1)^h y \quad\text{and}\quad f_{f_0}(h,h') = \begin{cases} f_0h'  & \text{if $h$ is odd,}\\ 0 & \text{if $h$ is even.}\end{cases}
\end{equation*}

\item[$k\ne 0$ and $r+u> 2\eta$] If $a=1+2^{r-\eta}k$, then $(f_0,\gamma)\in\{(2^{r-\eta} (z_1-sz_2),2^{\eta-u}z_2): 0\le z_1<2^{\eta} \text{ and } 0\le z_2<2^u\}$. More\-over,
\begin{equation*}
h \blackdiamond y = (1+2^{r-\eta}kh)y\quad\text{and}\quad f_{f_0}(h,h') = f_0hh'.
\end{equation*}
If $a=-1-2^{r-\eta}k$ with $k$ even, then $(f_0,\gamma)\in\{(z_1+2^{r-\eta-1}kz_1,2^{r-1}z_2+2^{\eta-1}z_1):0\le z_1,z_2<2\}$. Moreover, for $f_0=z_1+2^{r-\eta-1}kz_1$,
\begin{equation*}
h \blackdiamond y =(-1)^h (1 + 2^{r-\eta}kh)y\quad\text{and}\quad f_{f_0}(h,h') = \begin{cases} z_1h'+2^{r-\eta-1}kz_1hh'  & \text{if $h$ is odd,}\\ -k2^{r-\eta-1}z_1hh' & \text{if $h$ is even.}\end{cases}
\end{equation*}
If $a=-1-2^{r-\eta}k$ with $k$ odd, then $(f_0,\gamma)\in\{(0,2^{r-1}z_2) : 0\le z_2<2\}$. Moreover,
\begin{equation*}
h \blackdiamond y =(-1)^h (1 + 2^{r-\eta}kh)y\quad\text{and}\quad f_{f_0}(h,h') = 0.
\end{equation*}

\item[$k\ne 0$ and $r+u\le 2\eta$] If $a=1+2^{r-\eta}k$, then $(f_0,\gamma)\!\in\!\{(2^{r-\eta}sz_1,2^{\eta-u}(z_2-z_1)):0\!\le\! z_1\!<\!2^u\text{ and } 0\!\le\! z_2\! <\!2^{r+u-\eta}\}$. Moreover,
\begin{equation*}
h \blackdiamond y =\sum_{i=0}^h \binom{h}{i} 2^{(r-\eta)i}k^iy\quad\text{and}\quad f_{f_0}(h,h') =  \sum_{l=0}^{h-1} \binom{h}{l+1}  2^{(r-\eta)l} k^lf_0h'.
\end{equation*}
If $a=-1-2^{r-\eta}k$ with $k$ odd, then $(f_0,\gamma)\in\{(0,2^{r-1}z_2): 0\le z_2<2\}$.~More\-over,
\begin{equation*}
h \blackdiamond y =(-1)^h\sum_{i=0}^h \binom{h}{i} 2^{(r-\eta)i}k^iy\quad\text{and}\quad f_{f_0}(h,h') = 0.
\end{equation*}
If $a=-1-2^{r-\eta}k$ with $k$ even, then $(f_0,\gamma)\in\{(z_1,2^{r-1}z_2+2^{\eta-1}cz_1) : 0\le z_1,z_2<2\}$, where $c\coloneqq 1-2^{r-\eta-1}k$. Moreover,
\begin{equation*}
\quad\qquad h\blackdiamond y =(-1)^h\sum_{i=0}^h \binom{h}{i} 2^{(r-\eta)i}k^iy\quad\text{and}\quad f_{f_0}(h,h') = \sum_{l=0}^{h-1}  A_{hl} 2^{(r-\eta)l}k^lf_0h',
\end{equation*}
where $A_{hl}\coloneqq \sum_{j=l}^{h-1} (-1)^{j} \binom{j}{l}$.
\end{description}
\end{proposition}

\begin{proof} Note that $r>\eta+1$ implies that $r\ge 3$. Let $A$ be as in Propo\-sition~\ref{condiciones sin potencia general}. Since~$A$ is the multiplication by an invertible element $a\in \mathds{Z}_{2^r}$ and $A^{p^{\eta}} = \Ide$, by Lemma~\ref{para ejemploscon2} and Remark~\ref{A para ejemploscon2}, we have $a = \pm (1+2^{r-\eta}k)$ with $0\le k<2^{\eta}$. Moreover, since $1^{\times h}=h$, by the first equality in~\eqref{def de diamante general}, we have:
\begin{equation}\label{calculo diamante p=2}
h \blackdiamond y =\begin{cases} (1+2^{r-\eta}k)^h y = \sum_{i=0}^h \binom{h}{i} 2^{(r-\eta)i}k^iy & \text{if $a = 1+2^{r-\eta}k$,}\\ (-1)^h(1+2^{r-\eta}k)^h y = (-1)^h\sum_{i=0}^h \binom{h}{i} 2^{(r-\eta)i}k^iy & \text{if $a = -1-2^{r-\eta}k$.}\end{cases}
\end{equation}
Let $F_1$, $F_2$ and $G$ be as above Corollary~\ref{calculo de la homo}, and let $\ov{(f_0,\gamma)}\in \frac{\ker F_1\cap \ker F_2}{\ima G}$. From equality~\eqref{f simplificado}, we obtain
\begin{equation*}
f_{f_0}(h,h') = \sum_{j=0}^{h-1}a^j f_0h'.
\end{equation*}
Assume that $a=1+2^{r-\eta}k$. Then $a^j = \sum_{l=0}^j \binom{j}{l}2^{(r-\eta)l}k^l$, and so
\begin{equation}\label{cociclo1caso1}
f_{f_0}(h,h') = \sum_{j=0}^{h-1} \sum_{l=0}^j \binom{j}{l} 2^{(r-\eta)l}k^lf_0h' = \sum_{l=0}^{h-1} \sum_{j=l}^{h-1} \binom{j}{l} 2^{(r-\eta)l}k^lf_0h'= \sum_{l=0}^{h-1} \binom{h}{l+1} 2^{(r-\eta)l}k^lf_0h'.
\end{equation}
Assume now that $a = -1-2^{r-\eta}k$. Then $a^j = \sum_{l=0}^j (-1)^j\binom{j}{l}2^{(r-\eta)l}k^l$, and so
\begin{equation}\label{cociclo1caso2}
f_{f_0}(h,h') = \sum_{j=0}^{h-1} (-1)^j\sum_{l=0}^j \binom{j}{l} 2^{(r-\eta)l}k^lf_0h'= \sum_{l=0}^{h-1} \sum_{j=l}^{h-1} (-1)^j\binom{j}{l}  2^{(r-\eta)l} k^lf_0h'= \sum_{l=0}^{h-1} A_{hl} 2^{(r-\eta)l}k^l f_0h'.
\end{equation}
We claim that
\begin{equation}\label{GFF ej 5}
\begin{aligned}
& G(y) = \begin{cases} (-2^{r-\eta}ky,2^{\eta}y) &\text{if $a = 1+2^{r-\eta}k$,}\\ (2y+2^{r-\eta}ky,2^{\eta}y) &\text{if $a = -1-2^{r-\eta}k$,}\end{cases} \\
&  F_1(y_1,y_2) = \begin{cases} 2^{\eta}y_1+2^{r-\eta}ky_2 &\text{if $a = 1+2^{r-\eta}k$,}\\ 2^{\eta}y_1-(2+2^{r-\eta}k)y_2 &\text{if $a = -1-2^{r-\eta}k$,} \end{cases}\\
&  F_2(y_1,y_2) = \begin{cases}\phantom{-} \bigl(2^{3\eta-1}-2^{2\eta-1} + 2^{\eta} - 2^{r-1}k\bigr)y_1 + 2^{r-1}k y_2  &\text{if $a = 1+2^{r-\eta}k$,}\\ -\bigl(2^{2\eta-1} + 2^{r-1}k\bigr) y_1 + \bigl(2^{\eta} + 2^{r-1}k\bigr) y_2 &\text{if $a = -1-2^{r-\eta}k$.}\end{cases}
\end{aligned}
\end{equation}
In fact, the formulas for $G$ and $F_1$ are straightforward. In order to obtain the formula for $F_2$, we first compute $P(a)$ and $Q(a)$, where $P$ and $Q$ are as in~\eqref{los polinomios1}. Assume first that $a=1+2^{r-\eta}k$. Since $a^j = (1+2^{r-\eta}k)^j = \sum_{l=0}^j \binom{j}{l}2^{(r-\eta)l}k^l$ and $2^r\mid 2^{\eta}2^{(r-\eta)l}$ for all $l\ge 1$, we have
\begin{equation*}
P(a) = \sum_{j=0}^{2^{\eta}-1} (j2^{\eta}+1)a^j = \sum_{j=0}^{2^{\eta}-1} j2^{\eta} + \sum_{j=0}^{2^{\eta}-1}\sum_{l=0}^j \binom{j}{l}2^{(r-\eta)l}k^l.
\end{equation*}
Computing the expression at the right side of the second equality, we obtain
\begin{equation*}
P(a) = \frac{(2^{\eta}-1)2^{2\eta}}{2} +\sum_{l=0}^{2^{\eta}-1}\sum_{j=l}^{2^{\eta}-1} \binom{j}{l}2^{(r-\eta)l}k^l = (2^{\eta}-1)2^{2\eta-1} + \sum_{l=0}^{2^{\eta}-1}\binom{2^{\eta}}{l+1} 2^{(r-\eta)l}k^l.
\end{equation*}
Write $l+1 = 2^tv$ with $2\nmid v$ and $t\ge 0$. It is well known that $2^{\eta-t}\mid \binom{2^{\eta}}{l+1}$. Hence, the exponent $w$, of $2$ in $\binom{2^{\eta}}{l+1} 2^{(r-\eta)l}$, is greater than or equal to $\eta-t+(r-\eta)l$. Consequently, if $l\ge 2$, then
\begin{equation*}
w \ge \eta-t+(r-\eta)l = r-t+(r-\eta)(l-1)\ge r-t + 2(l-1)\ge r,
\end{equation*}
where we have used that if $l\ge 2$, then $2(l-1) = 2(2^tv-2)\ge t$. So, $2^r\mid \binom{2^{\eta}}{l+1}2^{(r-\eta)l}$, for $l\ge 2$, and consequently,
\begin{equation*}
P(a) = (2^{\eta}-1)2^{2\eta-1} + 2^{\eta} + (2^{\eta}-1)2^{r-1}k = (2^{\eta}-1)2^{2\eta-1} + 2^{\eta} -2^{r-1}k.
\end{equation*}
We next compute $Q(a)$. By the very definition of $Q(a)$, we have
\begin{equation*}
Q(a) = 2^{\eta}\sum_{j=1}^{2^{\eta}-1} j a^j  + 2^{\eta} = 2^{\eta}\sum_{j=1}^{2^{\eta}-1} j \sum_{l=0}^j \binom{j}{l}2^{(r-\eta)l}k^l + 2^{\eta}.
\end{equation*}
Since $2^r\mid 2^{\eta}2^{(r-\eta)l}$ for $l\ge 1$, the formula for $Q(a)$ reduces to
\begin{equation*}
Q(a) = 2^{\eta}\sum_{j=1}^{2^{\eta}-1}j  + 2^{\eta} = (2^{\eta}-1)2^{2\eta-1} + 2^{\eta}.
\end{equation*}
Assume now that $a=-1-2^{r-\eta}k$. Since $a^j = (-1-2^{r-\eta}k)^j = (-1)^j\sum_{l=0}^j \binom{j}{l}2^{(r-\eta)l}k^l$ and $2^r\mid 2^{\eta}2^{(r-\eta)l}$ for $l\ge 1$, we have
\begin{equation*}
P(a) = \sum_{j=0}^{2^{\eta}-1} (j2^{\eta}+1)a^j = \sum_{j=0}^{2^{\eta}-1} (-1)^j j2^{\eta} + \sum_{j=0}^{2^{\eta}-1}(-1)^j\sum_{l=0}^j \binom{j}{l}2^{(r-\eta)l}k^l.
\end{equation*}
Computing the expression at the right side of the second equality and using~\cite{Spivey}*{identity 217}, we obtain
\begin{equation*}
P(a) = -2^{2\eta-1} + \sum_{l=0}^{2^{\eta}-1} \left(\sum_{j=l}^{2^{\eta}-1} (-1)^j\binom{j}{l}\right) 2^{(r-\eta)l}k^l= -2^{2\eta-1} + \sum_{l=0}^{2^{\eta}-1} 2^{(r-\eta)l}k^lB_l,
\end{equation*}
where $B_l\coloneqq \sum_{k=1}^l \binom{2^{\eta}}{k}\bigl(-\frac{1}{2}\bigr)^{l+1-k}$. We know that $B_0=0$, $B_1=-2^{\eta-1}$ and $2^{2l-\eta-2}B_l\in \mathds{Z}$, for $l\ge 2$ (see~\eqref{divisibilidad de B sub l}).~Since
$$
(r-\eta)l = r-\eta+(r-\eta)(l-1)\ge r-\eta+2l-2,
$$
this implies that $2^{(r-\eta)l}k^lB_l\in \mathds{Z}$ and $2^r\mid 2^{(r-\eta)l}k^lB_l$, for $l\ge 2$. Hence,
\begin{equation*}
P(a) = -2^{2\eta-1} - 2^{r-1}k.
\end{equation*}
Arguing as above we obtain that
\begin{equation*}
Q(a) = -2^{2\eta-1}  + 2^{\eta}.
\end{equation*}
The formula for $F_2$ follows from these facts by an immediate computation, which finishes the proof of~\eqref{GFF ej 5}. Our next purpose is to calculate $\frac{\ker F_1\cap \ker F_2}{\ima G}$ and compute the map $f_{f_0}$ associated with a canonical representative $(f_0,\gamma)$ of each $\ov{(f_0,\gamma)}\in \frac{\ker F_1\cap \ker F_2}{\ima G}$. Assume first that $k=0$. We have two cases, $a=1$ and $a=-1$. If $a=1$, then
\begin{equation*}
\frac{\ker F_1\cap \ker F_2}{\ima G} = 2^{r-\eta}\mathds{Z}_{2^r}\oplus \frac{\mathds{Z}_{2^r}}{2^{\eta}\mathds{Z}_{2^r}},
\end{equation*}
and we can take $(f_0,\gamma)\in \{(2^{r-\eta}z_1,z_2):0\le z_1,z_2<2^{\eta}\}$. For each $(f_0,\gamma)$, equalities~\eqref{calculo diamante p=2} and~\eqref{cociclo1caso1} become
\begin{equation*}
h\blackdiamond y = y \quad\text{and}\quad f_{f_0} = f_0hh'.
\end{equation*}
If $a = -1$, then we define the map $\Phi\colon I\oplus I\to I\oplus I$, by $\Phi(y_1,y_2)\coloneqq (y_1,y_2+2^{\eta-1}y_1)$, and we set $\wt{G}\coloneqq \Phi^{-1}\xcirc G$, $\wt{F}_1\coloneqq F_1\xcirc \Phi$ and  $\wt{F}_2\coloneqq F_2\xcirc \Phi$. Clearly $\Phi$ induces an isomorphism $\ov{\Phi}\colon \frac{\ker \wt{F}_1\cap \ker \wt{F}_2}{\ima \wt{G}}\to \frac{\ker F_1\cap \ker F_2}{\ima G}$. A direct~computa\-tion, using~\eqref{GFF ej 5}, shows that
\begin{equation*}
\wt{G}(y) = (2y,0),\quad \wt{F}_1(y_1,y_2) = -2y_2 \quad \text{and}\quad  \wt{F}_2(y_1,y_2) =  2^{\eta} y_2.
\end{equation*}
Consequently
\begin{equation*}
\frac{\ker \wt{F}_1\cap \ker \wt{F}_2}{\ima \wt{G}} =  \frac{\mathds{Z}_{2^r}}{2\mathds{Z}_{2^r}}\oplus 2^{r-1}\mathds{Z}_{2^r}.
\end{equation*}
Take $(\ov{z_1},2^{r-1}z_2)\in \frac{\mathds{Z}_{2^r}}{2\mathds{Z}_{2^r}}\oplus 2^{r-1}\mathds{Z}_{2^r}$, where $0\le z_1,z_2< 2$. Since
\begin{equation*}
\ov{\Phi}(\ov{z_1},2^{r-1}z_2) = \ov{(z_1,2^{r-1}z_2+2^{\eta-1}z_1)},
\end{equation*}
we can take $(f_0,\gamma)\in\{(z_1,2^{r-1}z_2+2^{\eta-1}z_1):0\le z_1,z_2<2\}$. For each such $(f_0,\gamma)$, the formulas~\eqref{calculo diamante p=2} and~\eqref{cociclo1caso2}~be\-come
\begin{equation*}
h\blackdiamond y = (-1)^h y\quad\text{and}\quad f_{f_0}(h,h') = \begin{cases} f_0h'  & \text{if $h$ is odd,}\\ 0 & \text{if $h$ is even,}\end{cases}
\end{equation*}
as desired. Assume now that $k\ne 0$ and $r+u>2\eta$. Then, the formulas~\eqref{calculo diamante p=2}, \eqref{cociclo1caso1} and~\eqref{cociclo1caso2} become
\begin{align}
& h \blackdiamond y = \begin{cases} (1+2^{r-\eta}kh)y & \text{if $a = 1+2^{r-\eta}k$,}\\ (-1)^h (1 + 2^{r-\eta}kh)y & \text{if $a = -1-2^{r-\eta}k$.}\end{cases}\label{accion2p2}\\
\shortintertext{and}
& f_{f_0}(h,h') = \begin{cases} f_0hh' + 2^{r-\eta-1}k (h-1)f_0hh' & \text{if $a = 1+2^{r-\eta}k$,}\\ f_0h' + 2^{r-\eta-1}k f_0(h-1)h' & \text{if $a = -1-2^{r-\eta}k$ and $h$ is odd,}\\ -2^{r-\eta-1}kf_0hh' & \text{if $a = -1-2^{r-\eta}k$ and $h$ is even.}\end{cases}\label{cociclo2p2}
\end{align}
Our next purpose is to calculate $\frac{\ker F_1\cap \ker F_2}{\ima G}$ and compute the map $f_{f_0}$ associated with a canonical representative $(f_0,\gamma)$ of each $\ov{(f_0,\gamma)}\in \frac{\ker F_1\cap \ker F_2}{\ima G}$. We first consider the case $a = 1+2^{r-\eta}k$. We define the map $\Phi\colon I\oplus I\to I\oplus I$, by $\Phi(y_1,y_2)\coloneqq (y_1-2^{r-2\eta}ky_2,y_2)$, and we set $\wt{G}\coloneqq \Phi^{-1}\xcirc G$, $\wt{F}_1\coloneqq F_1\xcirc \Phi$ and  $\wt{F}_2\coloneqq F_2\xcirc \Phi$. Clearly $\Phi$ induces an isomorphism $\ov{\Phi}\colon \frac{\ker \wt{F}_1\cap \ker \wt{F}_2}{\ima \wt{G}}\to \frac{\ker F_1\cap \ker F_2}{\ima G}$. A direct com\-putation, using~\eqref{GFF ej 5}, shows that
\begin{equation*}
\wt{G}(y) = (0,2^{\eta}y),\quad \wt{F}_1(y_1,y_2) = 2^{\eta}y_1\quad\text{and}\quad  \wt{F}_2(y_1,y_2) = \bigl(2^{3\eta-1}+2^{2\eta-1} + 2^{\eta} - 2^{r-1}k\bigr)y_1 - 2^{r-\eta}ky_2.
\end{equation*}
Consequently,
\begin{equation*}
\frac{\ker \wt{F}_1\cap \ker \wt{F}_2}{\ima \wt{G}} =  2^{r-\eta}\mathds{Z}_{2^r}\oplus \frac{2^{\eta-u}\mathds{Z}_{2^r}}{2^{\eta}\mathds{Z}_{2^r}}.
\end{equation*}
Take $(2^{r-\eta}z_1,\ov{2^{\eta-u}z_2})\in 2^{r-\eta}\mathds{Z}_{2^r}\oplus \frac{2^{\eta-u}\mathds{Z}_{2^r}}{2^{\eta}\mathds{Z}_{2^r}}$, where $0\le z_1<2^{\eta}$ and $0\le z_2<2^u$. Since
\begin{equation*}
\ov{\Phi}(2^{r-\eta} z_1,\ov{2^{\eta-u}z_2}) = \ov{(2^{r-\eta} z_1-2^{r-\eta}sz_2,2^{\eta-u}z_2)},
\end{equation*}
we can take $(f_0,\gamma)\in\{(2^{r-\eta} (z_1-sz_2),2^{\eta-u}z_2):0\le z_1<2^{\eta}\text{ and }0\le z_2<2^u\}$. For each $(f_0,\gamma)$, the formulas~\eqref{accion2p2} and~\eqref{cociclo2p2} become
\begin{align*}
h\blackdiamond y = (1 + 2^{r-\eta}kh)y \quad \text{and}\quad f_{f_0}(h,h') = f_0hh'.
\end{align*}
We next consider the case $a = -1-2^{r-\eta}k$. We define maps $\Phi_1,\Phi_2\colon I\oplus I\to I\oplus I$, by $\Phi_1(y_1,y_2)\coloneqq (y_1+2^{r-2\eta}ky_2,y_2)$ and $\Phi_2(y_1,y_2)\coloneqq (y_1,y_2+2^{\eta-1}y_1)$, and we set $\wt{G}\coloneqq \Phi^{-1}\xcirc G$, $\wt{F}_1\coloneqq F_1\xcirc \Phi$ and  $\wt{F}_2\coloneqq F_2\xcirc \Phi$, where $\Phi\coloneqq \Phi_1\xcirc \Phi_2$. Clearly $\Phi$ induces an isomorphism $\ov{\Phi}\colon \frac{\ker \wt{F}_1\cap \ker \wt{F}_2}{\ima \wt{G}}\to \frac{\ker F_1\cap \ker F_2}{\ima G}$. Using~\eqref{GFF ej 5}, we obtain
\begin{equation*}
\wt{G}(y) = (2y,0),\quad \wt{F}_1(y_1,y_2) = -2y_2 \quad\text{and}\quad \wt{F}_2(y_1,y_2) =  -2^{r-1}ky_1+2^{\eta} y_2.
\end{equation*}
Consequently
\begin{equation*}
\frac{\ker \wt{F}_1\cap \ker \wt{F}_2}{\ima \wt{G}} = \begin{cases} \frac{\mathds{Z}_{2^r}}{2\mathds{Z}_{2^r}}\oplus 2^{r-1}\mathds{Z}_{2^r} & \text{if $k$ is even,}\\ 0 \oplus 2^{r-1}\mathds{Z}_{2^r} & \text{if $k$ is odd.}\end{cases}
\end{equation*}
Assume that $k$ even and take $(\ov{z_1},2^{r-1}z_2)\in \frac{\mathds{Z}_{2^r}}{2\mathds{Z}_{2^r}}\oplus 2^{r-1}\mathds{Z}_{2^r}$, where $0\le z_1,z_2< 2$. Since
\begin{equation*}
\ov{\Phi}(\ov{z_1},2^{r-1}z_2) = \ov{(z_1+2^{r-\eta-1}kz_1,2^{r-1}z_2+2^{\eta-1}z_1)},
\end{equation*}
we can take $(f_0,\gamma)\in \{(z_1+2^{r-\eta-1}kz_1,2^{r-1}z_2+2^{\eta-1}z_1): 0\le z_1,z_2< 2\}$. For each one of these $(f_0,\gamma)$, formulas~\eqref{accion2p2} and~\eqref{cociclo2p2}~be\-come
\begin{equation*}
h\blackdiamond y = (-1)^h (1 + 2^{r-\eta}kh)y \quad\text{and}\quad f_{f_0}(h,h') = \begin{cases} z_1h'+2^{r-\eta-1}kz_1hh'  & \text{if $h$ is odd,}\\ -k2^{r-\eta-1}z_1hh' & \text{if $h$ is even.}\end{cases}
\end{equation*}
On the other hand, when $k$ is odd, we take $(0,2^{r-1}z_2)\in 0\oplus 2^{r-1}\mathds{Z}_{2^r}$, where $0\le z_2< 2$. Since
\begin{equation*}
\ov{\Phi}(0,2^{r-1}z_2) = \ov{(0,2^{r-1}z_2)},
\end{equation*}
we can take $(f_0,\gamma)\in\{(0,2^{r-1}z_2):0\le z_2< 2\}$. For each one of these $(f_0,\gamma)$, the formulas~\eqref{accion2p2} and~\eqref{cociclo2p2} become
\begin{equation*}
h\blackdiamond y = (-1)^h (1 + 2^{r-\eta}kh)y \quad\text{and} \quad  f_{f_0}(h,h') = 0,
\end{equation*}
as desired. It remains to consider the case $k\ne0$ and $r+u\le 2\eta$. Assume first that $a=1+2^{r-\eta}k$. We define the map $\Phi\colon I\oplus I\to I\oplus I$, by $\Phi(y_1,y_2)\coloneqq (sy_1,y_2-2^{2\eta-r-u}y_1)$, and we set $\wt{G}\coloneqq \Phi^{-1}\xcirc G$, $\wt{F}_1\coloneqq F_1\xcirc \Phi$ and~$\wt{F}_2\coloneqq F_2\xcirc\Phi$. Clearly $\Phi$ induces an isomorphism $\ov{\Phi}\colon \frac{\ker \wt{F}_1\cap \ker \wt{F}_2}{\ima \wt{G}}\to \frac{\ker F_1\cap \ker F_2}{\ima G}$. A direct computation, using~\eqref{GFF ej 5} and the fact that $r\le 2\eta$, shows that
\begin{equation*}
\wt{G}(y)=(-2^{r-\eta+u} y,0),\quad \wt{F}_1(y_1,y_2)=2^{r-\eta}k y_2\quad\text{and}\quad \wt{F}_2(y_1,y_2)=2^{\eta}(1-2^{r-\eta-1}k)sy_1+2^{r-1}ky_2.
\end{equation*}
Using that $1-2^{r-\eta-1}k\in \mathds{Z}_{2^r}$ is invertible (because $r>\eta+1$), we obtain
\begin{equation*}
\frac{\ker \wt{F}_1\cap \ker \wt{F}_2}{\ima \wt{G}} = \frac{2^{r-\eta}\mathds{Z}_{2^r}}{2^{r+u-\eta}\mathds{Z}_{2^r}}\oplus 2^{\eta-u}\mathds{Z}_{2^r}.
\end{equation*}
Take $(\ov{2^{r-\eta}z_1},2^{\eta-u} z_2)\in  \frac{2^{r-\eta}\mathds{Z}_{2^r}}{2^{r+u-\eta}\mathds{Z}_{2^r}}\oplus 2^{\eta-u}\mathds{Z}_{2^r}$, where $0\le z_1<2^u$ and $0\le z_2 <2^{r+u-\eta}$. Since
\begin{equation*}
\ov{\Phi}(\ov{2^{r-\eta}z_1},2^{\eta-u} z_2) = \ov{(2^{r-\eta}sz_1,2^{\eta-u}z_2-2^{\eta-u} z_1)},
\end{equation*}
we can take $(f_0,\gamma)\in \{(2^{r-\eta}sz_1,2^{\eta-u}(z_2-z_1)):0\le z_1<2^u\text{ and }0\le z_2 <2^{r+u-\eta}\}$. In this case formu\-las~\eqref{calculo diamante p=2} and~\eqref{cociclo1caso1} remain the same. Assume now that $a=-1-2^{r-\eta}k$. Then, we define the morphism $\Phi\colon I\oplus I\to I\oplus I$, by $\Phi(y_1,y_2)\coloneqq (y_1,y_2+2^{\eta-1}cy_1)$, and we set $\wt{G}\coloneqq \Phi^{-1}\xcirc G$, $\wt{F}_1\coloneqq F_1\xcirc \Phi$ and  $\wt{F}_2\coloneqq F_2\xcirc \Phi$. Clearly $\Phi$ induces an isomorphism $\ov{\Phi}\colon \frac{\ker \wt{F}_1\cap \ker \wt{F}_2}{\ima \wt{G}}\to \frac{\ker F_1\cap \ker F_2}{\ima G}$. Let $b\coloneqq 1+2^{r-\eta-1}k$. A direct computation, using~\eqref{GFF ej 5} and the fact that $2^{\eta}bc=2^{\eta}$, shows that
\begin{equation*}
\wt{G}(y)=(2by,0),\quad \wt{F}_1(y_1,y_2)=-2b y_2\quad\text{and}\quad \wt{F}_2(y_1,y_2)=-2^{r-1}ky_1+2^{\eta}by_2.
\end{equation*}
Consequently
\begin{equation*}
\frac{\ker \wt{F}_1\cap \ker \wt{F}_2}{\ima \wt{G}}=\begin{cases} 0\oplus 2^{r-1}\mathds{Z}_{2^r} & \text{if $k$ is odd,}\\ \frac{\mathds{Z}_{2^r}}{2\mathds{Z}_{2^r}}\oplus 2^{r-1}\mathds{Z}_{2^r} & \text{if $k$ is even.}\end{cases}
\end{equation*}
When $k$ is odd, we take $(0,2^{r-1} z_2)\in  0\oplus 2^{r-1}\mathds{Z}_{2^r}$, where $0\le z_2<2$, and we check that
\begin{equation*}
\ov{\Phi}(0,2^{r-1}z_2) = \ov{(0,2^{r-1}z_2)}.
\end{equation*}
So, we can take $(f_0,\gamma)\in\{(0,2^{r-1}z_2):0\le z_2<2\}$. For these $(f_0,\gamma)$, formulas~\eqref{calculo diamante p=2} and~\eqref{cociclo1caso2} become
\begin{equation*}
h\blackdiamond y =(-1)^h\sum_{i=0}^h \binom{h}{i} 2^{(r-\eta)i}k^iy \qquad\text{and}\qquad f_{f_0}(h,h') =0.
\end{equation*}
Finally, when $k$ is even we take $(\ov{z_1},2^{r-1} z_2)\in\frac{\mathds{Z}_{2^r}}{2\mathds{Z}_{2^r}}\oplus 2^{r-1}\mathds{Z}_{2^r}$, where $0\le z_1,z_2<2$, and we check that
\begin{equation*}
\ov{\Phi}(\ov{z_1},2^{r-1}z_2) = \ov{(z_1,2^{r-1}z_2+2^{\eta-1}cz_1)}.
\end{equation*}
Hence, we can take $(f_0,\gamma)\in\{(z_1,2^{r-1}z_2+2^{\eta-1}cz_1):0\le z_1,z_2<2\}$. For each $(f_0,\gamma)$, formulas~\eqref{calculo diamante p=2} and~\eqref{cociclo1caso2} remain the same.
\end{proof}

\subsection[Examples with \texorpdfstring{$H$}{H} nontrivial or \texorpdfstring{$\Yleft\ne 0$}{-<==}]{Examples with \texorpdfstring{$\pmb H$}{H} nontrivial or \texorpdfstring{$\pmb{\Yleft\ne 0}$}{< ne 0}}\label{seccion 4.2}

Let $p$ be an odd prime number, let $0<\nu\le\eta \le 2\nu$ and let $H$ be as at the beginning of Section~\ref{seccion 2}. In this subsection we consider two examples of abelian groups $I$ endowed with endomorphisms $A,B\colon I\to I$ satisfying conditions~\eqref{igualdad exponencial general} and, for each one of these examples we compute $\ho_{\blackdiamond,\Yleft}^2(H,I)$, where $\blackdiamond\colon H\times I\to I$ and $\Yleft \colon I\times H\to I$ are as in~\eqref{def de diamante general}. For this we will use Corollary~\ref{calculo de la homo B'}. Thus, in each of these examples, we must calculate the maps $F_1$, $F_2$, $G$ and $F$, introduced above that corollary, in order to obtain a family of pairs $(f_0,\gamma)\in \ker(F_1)\cap \ker(F_2)$, which parametrizes a complete family $F(f_0,\gamma) = (\alpha_1(\gamma),f_{f_0})$, of representatives of $2$-cocycles modulo coboundaries in ${\mathcal{C}}_{\blackdiamond,\Yleft}(H,I)$. Let $P$, $Q$, $R$ and $S$ be as in~\eqref{los polinomios1} and~\eqref{mas polinomios}. Assume that~$p^{\nu}I=0$, which clearly implies that $S=0$. Arguing as in Example~\ref{ejemplo 1}, we obtain that $\ker(F_1) = I\oplus \ker(A-\Ide)$, and
\begin{equation*}
Ay_2=y_2,\quad P(A)y_2=0\quad \text{and}\quad Q(A)y_2=0\qquad\text{for all $(y_1,y_2)\in \ker (F_1)$,}
\end{equation*}
which implies that,
$$
G(y)=(y-Ay-BAy,0)\qquad\text{and}\qquad F_2(y_1,y_2)= P(A) y_1+ABR(A) y_1+ABy_2\qquad\text{for $(y_1,y_2)\in \ker(F_1)$}.
$$
Consequently, if $P(A)=0$ and $B R(A)=0$, then
\begin{equation}\label{cohomologia con Yleft}
\ho_{\blackdiamond,\Yleft}^2(H,I)\simeq \frac{\ker(F_1)\cap \ker(F_2)}{\ima(G)}=\frac{I}{\ima(\Ide-A-BA)}\oplus (\ker(\Ide-A)\cap \ker(B)).
\end{equation}

\subsubsection[Example with \texorpdfstring{$I=\mathds{Z}_{p^r}$}{I=Zpr} where \texorpdfstring{$r\le \nu$}{r<= v}]{Example with \texorpdfstring{$\pmb{I=\mathds{Z}_{p^r}}$}{Zpr} where \texorpdfstring{$\pmb{r\le \nu}$}{r<=v}}
Let $I\coloneqq \mathds{Z}_{p^r}$, with $r\le \nu$. Let $A,B\in \End(I)$ satisfying $A^{p^{\eta}}=\Ide$ and $B^2=0$. Note that if $B=0$, then the conditions of item~(1) of Proposition~\ref{general conditions'} are satisfied; while, if $B\ne 0$, then the conditions of item~(2) of Proposition~\ref{general conditions'} are~satis\-fied. Write $A=a \Ide$ and $B=b \Ide$. If $b\ne 0$, write $b=p^s d$ with $p\nmid d$. By Remark~\ref{para ejemplo con I ciclico}, we know that \hbox{$a^{p^{\eta}}\equiv 1\pmod{p^r}$} and $b=0$ or $s<r\le 2s$. Actually, by Lemma~\ref{para ejemplos}, we know that $a^{p^{r-1}}= 1$ and that $a = 1+kp$ with $0\le k< p^{r-1}$. Since $r-1\le \nu$, the first equality implies $a^{p^{\nu}}=1$ and so, by~\eqref{formula para l}, we also have $a^{l(h)} = a^h$, for all $h\in H$. Hence, by the definition of $l(h)$ and~\eqref{def de diamante general}, we have
\begin{equation}\label{accionn3}
h \blackdiamond y = a^{l(h)} y = a^h y \quad\text{and}\quad y\Yleft h\coloneqq hby,\qquad\text{for $h\in H$ and $y\in I$.}
\end{equation}
Consequently,
\begin{equation}\label{accion3}
h \blackdiamond y = (1+kp)^h y = \sum_{i=0}^{\min(h,r-1)} \binom{h}{i} k^ip^iy\qquad\text{for $0\le h < p^{\eta}$ and $y\in I$.}
\end{equation}
For $f_0,\gamma\in I$, we next compute $f=f_{f_0,\gamma}$. According to Definition~\ref{definicion de breve f con B},
$$
f(1^{\times k},c)\coloneqq \sum_{j=0}^{k-1}A^{k-1-j}\tilde{f}(j\cdot c) + (k\cdot c)B\sum_{l=0}^{k-1} A^l\tilde{f}(b_l),\qquad\text{for $0\le k<p^{\eta}$ and $c\in H$,}
$$
where $\tilde{f}$ is given by Proposition~\ref{formula mas simple de tilde f} and the $b_l$'s are as in Remark~\ref{b_k=1 si y...}.

\begin{proposition} Let $H$, $I$, $a$, $b$ and $k$ be as above and write $k=p^u v$ with $p\nmid v$. Define
$$
t_1\coloneqq \begin{cases} u+1 & \text{if $b=0$,}\\ \min\{u+1,s\} & \text{if $b\ne 0$,}\end{cases}\qquad \text{and} \qquad t_2\coloneqq \begin{cases} u+1 & \text{if $b=0$,}\\ t_1 & \text{if $b\ne 0$, $s\ne u+1$,}\\ s+\max\{j:p^j|(v+d+p^svd)\}, & \text{if $b\ne 0$, $s=u+1$.}\end{cases}
$$
Then
\begin{equation}\label{cohomologia ejemplo}
\ho_{\blackdiamond,\Yleft}^2(H,I) \simeq \frac{\ker(F_1)\cap \ker(F_2)}{\ima(G)} = \frac{\mathds{Z}_{p^r}}{p^{t_2}\mathds{Z}_{p^r}}\oplus p^{r-t_1}\mathds{Z}_{p^r},
\end{equation}
and a complete family $(\alpha_1(\gamma),f_{f_0})$ of representatives of $2$-cocycles modulo coboundaries in ${\mathcal{C}}_{\blackdiamond,\Yleft}(H,I)$ is parameterized by the pairs $(f_0,\gamma)\in\{(z_1,p^{r-t_1}z_2):0\le z_1<p^{t_2}\text{ and }0\le z_2<p^{t_1}\}$.
\end{proposition}

\begin{proof} We first prove that $P(A)=0$ and $R(A)=0$. Since $p^{\nu}I=0$, we have
\begin{equation*}
P(a) = \sum_{j=0}^{p^{\eta}-1} (jp^{\nu}+1)a^j = \sum_{j=0}^{p^{\eta}-1} (1+kp)^j = \sum_{j=0}^{p^{\eta}-1}\sum_{l=0}^{j} \binom{j}{l} p^lk^l=\sum_{l=0}^{p^{\eta}-1}\sum_{j=l}^{p^{\eta}-1} \binom{j}{l} p^lk^l =\sum_{l=0}^{p^{\eta}-1}\binom {p^{\eta}}{l+1} p^lk^l.
\end{equation*}
Write $l+1 = p^tv$ with $p\nmid v$ and $t\ge 0$. Since $p^{\eta-t}\mid \binom{p^{\eta}}{l+1}$ and $l\ge t$, for all $l$, we conclude that $p^{\eta}| \binom{p^{\eta}}{l+1}p^l$, for all $l$. Consequently, $P(a)=0$. Moreover, using again that $p^{\nu}I=0$, it follows from~\eqref{R simple}, that
$$
R(a)=\sum_{j=0}^{p^{\eta}-1} j(1+pk)^j = \sum_{j=0}^{p^{\eta}-1}\sum_{l=0}^{j} j\binom{j}{l} p^lk^l=\sum_{l=0}^{p^{\eta}-1}\sum_{j=l}^{p^{\eta}-1} j\binom{j}{l} p^lk^l.
$$
By~\cite{Spivey}*{Identity~216}, we have
$$
\sum_{j=l}^{p^{\eta}-1} j\binom{j}{l}=\binom{p^{\eta}}{l+1}(p^{\eta}-1)-\binom{p^{\eta}}{l+2},
$$
Write $l+2=p^tw$  with $p\nmid w$ and $t\ge 0$. Note that $p>2$ implies that $l\ge t$, for all $l$. Since $p^{\eta-t}\mid \binom{p^{\eta}}{l+2}$, we conclude that $p^{\eta}\mid \binom{p^{\eta}}{l+2}p^l$, for all $l$. Since, moreover $p^{\eta}\mid p^l\binom{p^{\eta}}{l+1}$, we have $R(a)=0$, as desired. Thus, we can~use~\eqref{cohomologia con Yleft}, in order to compute $\ho_{\blackdiamond,\Yleft}^2(H,I)$. Hence, we are reduced to compute $\ima(\Ide-A-BA)$, $\ker(\Ide-A)$ and $\ker(B)$. Assume that~$b\ne 0$. From $A=a\Ide$ and $a=1+p^{u+1}v$ it follows that $\ker(\Ide-A)=p^{r-u-1}\mathds{Z}_{p^r}$, and from $b=p^s d$ it follows that $\ker(B)=p^{r-s}\mathds{Z}_{p^r}$. We also obtain $\Ide-A-BA=(-p^{u+1}v-p^sd-p^{u+s+1}vd)\Ide$, and so $\ima(\Ide-A-AB)=p^{t_2} \mathds{Z}_{p^r}$. Thus, equality~\eqref{cohomologia ejemplo} holds in this case, and so a complete family $(\alpha_1(\gamma),f_{f_0})$ of representatives of $2$-cocycles modulo coboundaries in ${\mathcal{C}}_{\blackdiamond,\Yleft}(H,I)$, is parameterized by the pairs $(f_0,\gamma)\in\{(z_1,p^{r-t_1}z_2):0\le z_1<p^{t_2}\text{ and }0\le z_2<p^{t_1}\}$. The case $b=0$ is left to the reader.
\end{proof}

\subsubsection[Example with \texorpdfstring{$I=\mathds{Z}_{p^r}\oplus \mathds{Z}_{p^r}$}{I=Zpr+Zpr} where \texorpdfstring{$r\le \nu$}{r<=v}]{Example with \texorpdfstring{$\pmb{I=\mathds{Z}_{p^r}\oplus \mathds{Z}_{p^r}}$}{I=Zpr+Zpr} where \texorpdfstring{$\pmb{r\le \nu}$}{r<=v}}

Let $I\coloneqq \mathds{Z}_{p^r}\oplus \mathds{Z}_{p^r}$, with $r\le \nu$. Let $a,b\in \mathds{Z}_{p^r}$ and define $A,B\in \End(I)$ by $A\coloneqq \begin{psmallmatrix} 1&a\\ 0&1 \end{psmallmatrix}$ and $B=\begin{psmallmatrix} 0&b\\ 0&0 \end{psmallmatrix}$. It is clear that $A^j=\begin{psmallmatrix} 1&ja\\ 0&1 \end{psmallmatrix}$, and so $A^{p^{\eta}}=\Ide$. Moreover, $B^2=0$ and $AB=BA$. Note that if $B=0$, then the conditions of item~(1) of Proposition~\ref{general conditions'} are satisfied; while, if $B\ne 0$, then the conditions of item~(2) of Proposition~\ref{general conditions'} are satisfied. Note also that $p^{\nu}I=0$ implies that $A^{p^{\nu}}=\Ide$, and so $A^{l(h)}=A^h$, for all $h\in H$, where $l(h)$ is as in~\eqref{formula para l}. Let $\blackdiamond\colon H\times I\to I$ and $\Yleft \colon I\times H\to I$ be as in~\eqref{def de diamante general}, so that
\begin{align*}
& h \blackdiamond y = A^{h} y =\begin{pmatrix} 1&ah\\ 0&1 \end{pmatrix}\binom{y_1}{y_2}=\binom{y_1+ahy_2}{y_2} && \text{for $h\in H$ and $y=\binom{y_1}{y_2}\in I$}
\shortintertext{and}
& y\Yleft h= hBy=h\begin{pmatrix} 0&b\\ 0&0 \end{pmatrix}\binom{y_1}{y_2}=\binom{hby_2}{0} && \text{for $h\in H$ and $y=\binom{y_1}{y_2}\in I$.}
\end{align*}
For $f_0,\gamma\in I$, we next compute $f\coloneqq f_{f_0,\gamma}$. According to Definition~\ref{definicion de breve f con B},
$$
f(1^{\times k},c)\coloneqq \sum_{j=0}^{k-1}A^{k-1-j}\tilde{f}(j\cdot c) + (k\cdot c)B\sum_{l=0}^{k-1} A^l\tilde{f}(b_l)\qquad\text{for $0\le k<p^{\eta}$ and $c\in H$,}
$$
where $\tilde{f}$ is given by Proposition~\ref{formula mas simple de tilde f} and the $b_l$'s are as in Remark~\ref{b_k=1 si y...}. By the second equality in~\eqref{def de diamante general}, we also~have

\begin{proposition} Let $H$, $I$, $A$, $B$, $\blackdiamond$ and $\Yleft$ be as above. Assume that $a$, $b$ and $c\coloneqq a+b$ are non-zero and write $a=a_1p^{a_2}$, $b=b_1p^{b_2}$ and $c=c_1p^{c_2}$, where $p\nmid a_1$, $p\nmid b_1$ and $p\nmid c_1$. Set $t\coloneqq \min\{a_2,b_2\}$. Then
\begin{equation*}
\ho_{\blackdiamond,\Yleft}^2(H,I) \simeq \frac{\ker(F_1)\cap \ker(F_2)}{\ima(G)} = \left(\frac{\mathds{Z}_{p^r}}{p^{c_2}\mathds{Z}_{p^r}}\oplus \mathds{Z}_{p^r}\right) \oplus \left(\mathds{Z}_{p^r}\oplus p^{r-t}\mathds{Z}_{p^r}\right),
\end{equation*}
and a complete family $(\alpha_1(\gamma),f_{f_0})$, of representatives of $2$-cocycles modulo coboundaries in ${\mathcal{C}}_{\blackdiamond,\Yleft}(H,I)$, is parameterized by the pairs $(f_0,\gamma)\in \bigl\{\bigl(\binom{z_1}{z_2}, \binom{z_3}{p^{r-t}z_4}\bigr) :0\le z_1<p^{c_2},\text{ }0\le z_2,z_3 < p^r\text{ and }0\le z_4<p^t\bigr\}$.
\end{proposition}

\begin{proof} We first prove that $P(A)=0$ and $BR(A)=0$. Since $p^{\nu}I=0$, we have
\begin{equation*}
P(A) = \sum_{j=0}^{p^{\eta}-1} (jp^{\nu}+1) A^j = \sum_{j=0}^{p^{\eta}-1} \begin{pmatrix} 1&ja\\ 0&1 \end{pmatrix} = p^{\eta}\begin{pmatrix} 1&0\\ 0&1 \end{pmatrix}+\binom{p^{\eta}}{2}\begin{pmatrix} 0&a\\ 0&0 \end{pmatrix}=0.
\end{equation*}
Moreover, by~\eqref{R simple},
\begin{equation*}
R(A) = \sum_{j=0}^{p^{\eta}-1} j A^j=\sum_{j=0}^{p^{\eta}-1} \begin{pmatrix} j&j^2 a\\ 0&j \end{pmatrix}=\binom{p^{\eta}}{2}\begin{pmatrix} 1&0\\ 0&1 \end{pmatrix}+\sum_{j=0}^{p^{\eta}-1} j^2 \begin{pmatrix} 0&a\\ 0&0 \end{pmatrix}.
\end{equation*}
Hence, $BR(A)=0$, because $B\begin{psmallmatrix} 0&a\\ 0&0 \end{psmallmatrix}=0$ and $\binom{p^{\eta}}{2}I=0$. So, in order to obtain $\ho_{\blackdiamond,\Yleft}^2(H,I)$, we can use~\eqref{cohomologia con Yleft}. Since
\begin{align*}
&\ker(\Ide-A)=\ker \begin{pmatrix} 0&-a\\ 0&0 \end{pmatrix}=\mathds{Z}_{p^r}\oplus p^{r-a_2}\mathds{Z}_{p^r},\\
&\ker(B)=\ker \begin{pmatrix} 0&b\\ 0&0 \end{pmatrix}=\mathds{Z}_{p^r}\oplus p^{r-b_2}\mathds{Z}_{p^r}
\shortintertext{and}
& \ima(\Ide-A-AB)=\ima \begin{pmatrix} 0&-a-b\\ 0&0 \end{pmatrix}=p^{c_2}\mathds{Z}_{p^r}\oplus \{0\},
\end{align*}
we have
\begin{equation*}
\ho_{\blackdiamond,\Yleft}^2(H,I) = \frac{I}{\ima(\Ide-A-AB)}\oplus (\ker(\Ide-A)\cap \ker(B))=\left(\frac{\mathds{Z}_{p^r}}{p^{c_2} \mathds{Z}_{p^r}}\oplus \mathds{Z}_{p^r}\right) \oplus \left(\mathds{Z}_{p^r}\oplus p^{r-t}\mathds{Z}_{p^r}\right).
\end{equation*}
Consequently, a complete family of $2$-cocycles modulo coboundaries $(\alpha_1(\gamma),f_{f_0})$ of ${\mathcal{C}}_{\blackdiamond,\Yleft}(H,I)$, is parameterized by the pairs $(f_0,\gamma)\in \bigl\{\bigl(\binom{z_1}{z_2}, \binom{z_3}{p^{r-t}z_4}\bigr) :0\le z_1<p^{c_2},\text{ }0\le z_2,z_3 < p^r\text{ and }0\le z_4<p^t\bigr\}$.
\end{proof}

\begin{remark} If one or more of $a$, $b$, $c$ are zero, then the computations are slightly easier.
\end{remark}


\begin{bibdiv}
	\begin{biblist}

\bib{B}{article}{
	title={Extensions, matched products, and simple braces},
	author={Bachiller, David},
	journal={Journal of Pure and Applied Algebra},
	volume={222},
	number={7},
	pages={1670--1691},
	year={2018},
	publisher={Elsevier}
}		


\bib{BG}{article}{
	title={On groups of I-type and involutive Yang--Baxter groups},
	author={David, Nir Ben},
	author={Ginosar, Yuval},
	journal={Journal of Algebra},
	volume={458},
	pages={197--206},
	year={2016},
	publisher={Elsevier}
}				

\bib{CR}{article}{
  title={Regular subgroups of the affine group and radical circle algebras},
  author={Catino, Francesco},
  author={Rizzo, Roberto},
  journal={Bulletin of the Australian Mathematical Society},
  volume={79},
  number={1},
  pages={103--107},
  year={2009},
  publisher={Cambridge University Press}
}

\bib{CJO}{article}{
	author={Ced{\'o}, Ferran},
	author={Jespers, Eric},
	author={Okni{\'n}ski, Jan},
	title={Retractability of set theoretic solutions of the Yang--Baxter equation},
	journal={Advances in Mathematics},
	volume={224},
	number={6},
	pages={2472--2484},
	year={2010},
	publisher={Elsevier}
}

\bib{CJO1}{article}{
  title={Braces and the Yang--Baxter equation},
  author={Ced{\'o}, Ferran},
  author={Jespers, Eric},
  author={Okni{\'n}ski, Jan},
  journal={Communications in Mathematical Physics},
  volume={327},
  number={1},
  pages={101--116},
  year={2014},
  publisher={Springer}
}


\bib{CJR}{article}{
title={Involutive Yang-Baxter groups},
  author={Ced{\'o}, Ferran},
  author={Jespers, Eric},
  author={Del Rio, Angel},
  journal={Transactions of the American Mathematical Society},
  volume={362},
  number={5},
  pages={2541--2558},
  year={2010}
}

\bib{Ch}{article}{
	title={Fixed-point free endomorphisms and Hopf Galois structures},
	author={Childs, Lindsay},
	journal={Proceedings of the American Mathematical Society},
	volume={141},
	number={4},
	pages={1255--1265},
	year={2013},
	review={\MR{3008873}} 	
}

\bib{DG}{article}{
	title={On groups of I-type and involutive Yang--Baxter groups},
	author={David, Nir Ben},
	author={Ginosar, Yuval},
	journal={Journal of Algebra},
	volume={458},
	pages={197--206},
	year={2016},
	publisher={Elsevier},
	review={\MR{3500774}}
}	

\bib{De1}{article}{
	title={Set-theoretic solutions of the Yang--Baxter equation, RC-calculus, and Garside germs},
	author={Dehornoy, Patrick},
	journal={Advances in Mathematics},
	volume={282},
	pages={93--127},
	year={2015},
	publisher={Elsevier},
	review={\MR{3374524}}
}

\bib{DDM}{article}{
	title={Garside families and Garside germs},
	author={Dehornoy, Patrick},
	author={Digne, Fran{\c{c}}ois},
	author={Michel, Jean},
	journal={Journal of Algebra},
	volume={380},
	pages={109--145},
	year={2013},
	publisher={Elsevier},
	review={\MR{3023229}}
}		

\bib{Dr}{article}{
   author={Drinfeld, Vladimir G.},
   title={On some unsolved problems in quantum group theory},
   conference={
      title={Quantum groups},
      address={Leningrad},
      date={1990},
   },
   book={
      series={Lecture Notes in Math.},
      volume={1510},
      publisher={Springer, Berlin},
   },
   date={1992},
   pages={1--8},
   doi={10.1007/BFb0101175},
}

\bib{ESS}{article}{
	author={Etingof, Pavel},
	author={Schedler, Travis},
	author={Soloviev, Alexandre},
	title={Set-theoretical solutions to the quantum Yang-Baxter equation},
	journal={Duke Math. J.},
	volume={100},
	date={1999},
	number={2},
	pages={169--209},
	issn={0012-7094},
	review={\MR{1722951}},
	doi={10.1215/S0012-7094-99-10007-X},
}

\bib{GI2}{article}{
	title={Set-theoretic solutions of the Yang--Baxter equation, braces and symmetric groups},
	author={Gateva-Ivanova, Tatiana},
	journal={Advances in Mathematics},
	volume={338},
	pages={649--701},
	year={2018},
	publisher={Elsevier},
	review={\MR{3861714}}
}

\bib{GI3}{article}{
	title={Skew polynomial rings with binomial relations},
	author={Gateva-Ivanova, Tatiana},
	journal={Journal of Algebra},
	volume={185},
	number={3},
	pages={710--753},
	year={1996},
	publisher={Elsevier},
	review={\MR{1419721}}	
}

\bib{GI4}{article}{
	title={Quadratic algebras, Yang--Baxter equation, and Artin--Schelter regularity},
	author={Gateva-Ivanova, Tatiana},
	journal={Advances in Mathematics},
	volume={230},
	number={4-6},
	pages={2152--2175},
	year={2012},
	publisher={Elsevier},
	review={\MR{2927367}}
}

\bib{GIVB}{article}{
	title={Semigroups of I-Type},
	author={Gateva-Ivanova, Tatiana},
	author={Van den Bergh, Michel},
	journal={Journal of Algebra},
	volume={206},
	number={1},
	pages={97--112},
	year={1998},
	publisher={Elsevier},
	review={\MR{1637256}}
}

\bib{GGV}{article}{
   author={Guccione, Jorge A.},
   author={Guccione, Juan J.},
   author={Valqui, Christian},
   title={Extensions of linear cycle sets and cohomology},
   journal={Eur. J. Math.},
   volume={9},
   date={2023},
   number={1},
   pages={Paper No. 15, 29},
   issn={2199-675X},
   review={\MR{4551665}},
   doi={10.1007/s40879-023-00592-6},
}

\bib{JO}{article}{
	title={Monoids and groups of I-type},
	author={Jespers, Eric},
	author={Okni{\'n}ski, Jan},
	journal={Algebras and representation theory},
	volume={8},
	number={5},
	pages={709--729},
	year={2005},
	publisher={Springer},
	review={\MR{2189580}}	
}


\bib{LV}{article}{
	author={Lebed, Victoria},
	author={Vendramin, Leandro},
	title={Cohomology and extensions of braces},
	journal={Pacific Journal of Mathematics},
	volume={284},
	number={1},
	pages={191--212},
	year={2016},
	publisher={Mathematical Sciences Publishers}
}



\bib{R}{article}{
  title={Braces, radical rings, and the quantum Yang--Baxter equation},
  author={Rump, Wolfgang},
  journal={Journal of Algebra},
  volume={307},
  number={1},
  pages={153--170},
  year={2007},
  publisher={Elsevier}
}

\bib{R3}{article}{
	title={The brace of a classical group},
	author={Rump, Wolfgang},
	journal={Note di Matematica},
	volume={34},
	number={1},
	pages={115--145},
	year={2014},
}

\bib{S}{article}{
   author={Soloviev, Alexander},
   title={Non-unitary set-theoretical solutions to the quantum Yang-Baxter
   equation},
   journal={Math. Res. Lett.},
   volume={7},
   date={2000},
   number={5-6},
   pages={577--596},
   doi={10.4310/MRL.2000.v7.n5.a4},
}

\bib{Spivey}{book}{
   author={Spivey, Michael Z.},
   title={The art of proving binomial identities},
   series={Discrete Mathematics and its Applications (Boca Raton)},
   publisher={CRC Press, Boca Raton, FL},
   date={2019},
   pages={xiv+368},
   isbn={978-0-8153-7942-3},
   review={\MR{3931743}},
   doi={10.1201/9781351215824},
}


\end{biblist}
\end{bibdiv}
\end{document}